\pdfoutput=1

\documentclass[twoside,11pt]{article}
\usepackage{jmlr2e}

\usepackage{xcolor}
\usepackage[utf8]{inputenc} %
\usepackage[T1]{fontenc}    %
\usepackage{hyperref}       %
\usepackage{url}            %
\usepackage{booktabs}       %
\usepackage{amsfonts}       %

\usepackage{booktabs} %
\usepackage{longtable}

\usepackage{hyperref}

\usepackage{algorithm}%
\usepackage{algorithmicx}%
\usepackage{algpseudocode}%

\usepackage{amsmath,amssymb}%
\usepackage{graphicx}
\usepackage{mathrsfs}%

\usepackage{mathrsfs}%
\usepackage[title]{appendix}%
\usepackage{xcolor}%
\usepackage{textcomp}%
\usepackage{manyfoot}%

\usepackage{enumitem}

 \usepackage[english]{babel}
\usepackage{mathtools}
\usepackage{thmtools}
\usepackage{thm-restate}
\usepackage{algorithm}[l]
\usepackage{algpseudocode}
\usepackage{multirow}
\usepackage{chngpage}
\usepackage{xfrac}

\usepackage{mathrsfs}
\usepackage{xspace}
\usepackage{bm}
\usepackage{upgreek}
\usepackage{bbm}

\usepackage{subfig}

\newcommand{\safemath}[2]{\newcommand{#1}{\ensuremath{#2}\xspace}}

\safemath{\bma}{\mathbf{a}}
\safemath{\bmb}{\mathbf{b}}
\safemath{\bmc}{\mathbf{c}}
\safemath{\bmd}{\mathbf{d}}
\safemath{\bme}{\mathbf{e}}
\safemath{\bmf}{\mathbf{f}}
\safemath{\bmg}{\mathbf{g}}
\safemath{\bmh}{\mathbf{h}}
\safemath{\bmi}{\mathbf{i}}
\safemath{\bmj}{\mathbf{j}}
\safemath{\bmk}{\mathbf{k}}
\safemath{\bml}{\mathbf{l}}
\safemath{\bmm}{\mathbf{m}}
\safemath{\bmn}{\mathbf{n}}
\safemath{\bmo}{\mathbf{o}}
\safemath{\bmp}{\mathbf{p}}
\safemath{\bmq}{\mathbf{q}}
\safemath{\bmr}{\mathbf{r}}
\safemath{\bms}{\mathbf{s}}
\safemath{\bmt}{\mathbf{t}}
\safemath{\bmu}{\mathbf{u}}
\safemath{\bmv}{\mathbf{v}}
\safemath{\bmw}{\mathbf{w}}
\safemath{\bmx}{\mathbf{x}}
\safemath{\bmy}{\mathbf{y}}
\safemath{\bmz}{\mathbf{z}}
\safemath{\bmzero}{\mathbf{0}}
\safemath{\bmone}{\mathbf{1}}
\safemath{\bmpi}{\pmb{\pi}}
\safemath{\bmalpha}{\pmb{\alpha}}

\bmdefine{\biad}{a}
\bmdefine{\bibd}{b}
\bmdefine{\bicd}{c}
\bmdefine{\bidd}{d}
\bmdefine{\bied}{e}
\bmdefine{\bifd}{f}
\bmdefine{\bigd}{g}
\bmdefine{\bihd}{h}
\bmdefine{\biid}{i}
\bmdefine{\bijd}{j}
\bmdefine{\bikd}{k}
\bmdefine{\bild}{l}
\bmdefine{\bimd}{m}
\bmdefine{\bind}{n}
\bmdefine{\biod}{o}
\bmdefine{\bipd}{p}
\bmdefine{\biqd}{q}
\bmdefine{\bird}{r}
\bmdefine{\bisd}{s}
\bmdefine{\bitd}{t}
\bmdefine{\biud}{u}
\bmdefine{\bivd}{v}
\bmdefine{\biwd}{w}
\bmdefine{\bixd}{x}
\bmdefine{\biyd}{y}
\bmdefine{\bizd}{z}

\bmdefine{\bixid}{\xi}
\bmdefine{\bilambdad}{\lambda}
\bmdefine{\bimud}{\mu}
\bmdefine{\binud}{\nu}
\bmdefine{\bithetad}{\theta}
\bmdefine{\biomegad}{\omega}
\bmdefine{\biphid}{\phi}

\safemath{\bmia}{\biad}
\safemath{\bmib}{\bibd}
\safemath{\bmic}{\bicd}
\safemath{\bmid}{\bidd}
\safemath{\bmie}{\bied}
\safemath{\bmif}{\bifd}
\safemath{\bmig}{\bigd}
\safemath{\bmih}{\bihd}
\safemath{\bmii}{\biid}
\safemath{\bmij}{\bijd}
\safemath{\bmik}{\bikd}
\safemath{\bmil}{\bild}
\safemath{\bmim}{\bimd}
\safemath{\bmin}{\bind}
\safemath{\bmio}{\biod}
\safemath{\bmip}{\bipd}
\safemath{\bmiq}{\biqd}
\safemath{\bmir}{\bird}
\safemath{\bmis}{\bisd}
\safemath{\bmit}{\bitd}
\safemath{\bmiu}{\biud}
\safemath{\bmiv}{\bivd}
\safemath{\bmiw}{\biwd}
\safemath{\bmix}{\bixd}
\safemath{\bmiy}{\biyd}
\safemath{\bmiz}{\bizd}

\safemath{\bmxi}{\bixid}
\safemath{\bmlambda}{\bilambdad}
\safemath{\bmmu}{\bimud}
\safemath{\bmnu}{\binud}
\safemath{\bmtheta}{\bithetad}
\safemath{\bmomega}{\biomegad}
\safemath{\bmphi}{\biphid}

\safemath{\bA}{\mathbf{A}}
\safemath{\bB}{\mathbf{B}}
\safemath{\bC}{\mathbf{C}}
\safemath{\bD}{\mathbf{D}}
\safemath{\bE}{\mathbf{E}}
\safemath{\bF}{\mathbf{F}}
\safemath{\bG}{\mathbf{G}}
\safemath{\bH}{\mathbf{H}}
\safemath{\bI}{\mathbf{I}}
\safemath{\bJ}{\mathbf{J}}
\safemath{\bK}{\mathbf{K}}
\safemath{\bL}{\mathbf{L}}
\safemath{\bM}{\mathbf{M}}
\safemath{\bN}{\mathbf{N}}
\safemath{\bO}{\mathbf{O}}
\safemath{\bP}{\mathbf{P}}
\safemath{\bQ}{\mathbf{Q}}
\safemath{\bR}{\mathbf{R}}
\safemath{\bS}{\mathbf{S}}
\safemath{\bT}{\mathbf{T}}
\safemath{\bU}{\mathbf{U}}
\safemath{\bV}{\mathbf{V}}
\safemath{\bW}{\mathbf{W}}
\safemath{\bX}{\mathbf{X}}
\safemath{\bY}{\mathbf{Y}}
\safemath{\bZ}{\mathbf{Z}}

\safemath{\bZero}{\mathbf{0}}
\safemath{\bOne}{\mathbf{1}}
\safemath{\bDelta}{\mathbf{\Delta}}
\safemath{\bLambda}{\mathbf{\UpLambda}}
\safemath{\bPhi}{\mathbf{\Upphi}}
\safemath{\bSigma}{\mathbf{\Upsigma}}
\safemath{\bOmega}{\mathbf{\Upomega}}
\safemath{\bTheta}{\mathbf{\Uptheta}}

\bmdefine{\biAd}{A}
\bmdefine{\biBd}{B}
\bmdefine{\biCd}{C}
\bmdefine{\biDd}{D}
\bmdefine{\biEd}{E}
\bmdefine{\biFd}{F}
\bmdefine{\biGd}{G}
\bmdefine{\biHd}{H}
\bmdefine{\biId}{I}
\bmdefine{\biJd}{J}
\bmdefine{\biKd}{K}
\bmdefine{\biLd}{L}
\bmdefine{\biMd}{M}
\bmdefine{\biOd}{N}
\bmdefine{\biPd}{O}
\bmdefine{\biQd}{P}
\bmdefine{\biRd}{R}
\bmdefine{\biSd}{S}
\bmdefine{\biTd}{T}
\bmdefine{\biUd}{U}
\bmdefine{\biVd}{V}
\bmdefine{\biWd}{W}
\bmdefine{\biXd}{X}
\bmdefine{\biYd}{Y}
\bmdefine{\biZd}{Z}

\bmdefine{\biDelta}{\Delta}
\bmdefine{\biLambda}{\Lambda}
\bmdefine{\biPhi}{\Phi}
\bmdefine{\biSigma}{\Sigma}
\bmdefine{\biOmega}{\Omega}
\bmdefine{\biTheta}{\Theta}

\safemath{\bimA}{\biAd}
\safemath{\bimB}{\biBd}
\safemath{\bimC}{\biCd}
\safemath{\bimD}{\biDd}
\safemath{\bimE}{\biEd}
\safemath{\bimF}{\biFd}
\safemath{\bimG}{\biGd}
\safemath{\bimH}{\biHd}
\safemath{\bimI}{\biId}
\safemath{\bimJ}{\biJd}
\safemath{\bimK}{\biKd}
\safemath{\bimL}{\biLd}
\safemath{\bimM}{\biMd}
\safemath{\bimN}{\biNd}
\safemath{\bimO}{\biOd}
\safemath{\bimP}{\biPd}
\safemath{\bimQ}{\biQd}
\safemath{\bimR}{\biRd}
\safemath{\bimS}{\biSd}
\safemath{\bimT}{\biTd}
\safemath{\bimU}{\biUd}
\safemath{\bimV}{\biVd}
\safemath{\bimW}{\biWd}
\safemath{\bimX}{\biXd}
\safemath{\bimY}{\biYd}
\safemath{\bimZ}{\biZd}

\safemath{\bimDelta}{\biDelta}
\safemath{\bimLambda}{\biLambda}
\safemath{\bimPhi}{\biPhi}
\safemath{\bimSigma}{\biSigma}
\safemath{\bimOmega}{\biOmega}
\safemath{\bimTheta}{\biTheta}

\safemath{\setA}{\mathcal{A}}
\safemath{\setB}{\mathcal{B}}
\safemath{\setC}{\mathcal{C}}
\safemath{\setD}{\mathcal{D}}
\safemath{\setE}{\mathcal{E}}
\safemath{\setF}{\mathcal{F}}
\safemath{\setG}{\mathcal{G}}
\safemath{\setH}{\mathcal{H}}
\safemath{\setI}{\mathcal{I}}
\safemath{\setJ}{\mathcal{J}}
\safemath{\setK}{\mathcal{K}}
\safemath{\setL}{\mathcal{L}}
\safemath{\setM}{\mathcal{M}}
\safemath{\setN}{\mathcal{N}}
\safemath{\setO}{\mathcal{O}}
\safemath{\setP}{\mathcal{P}}
\safemath{\setQ}{\mathcal{Q}}
\safemath{\setR}{\mathcal{R}}
\safemath{\setS}{\mathcal{S}}
\safemath{\setT}{\mathcal{T}}
\safemath{\setU}{\mathcal{U}}
\safemath{\setV}{\mathcal{V}}
\safemath{\setW}{\mathcal{W}}
\safemath{\setX}{\mathcal{X}}
\safemath{\setY}{\mathcal{Y}}
\safemath{\setZ}{\mathcal{Z}}
\safemath{\emptySet}{\varnothing}

\safemath{\colA}{\mathscr{A}}
\safemath{\colB}{\mathscr{B}}
\safemath{\colC}{\mathscr{C}}
\safemath{\colD}{\mathscr{D}}
\safemath{\colE}{\mathscr{E}}
\safemath{\colF}{\mathscr{F}}
\safemath{\colG}{\mathscr{G}}
\safemath{\colH}{\mathscr{H}}
\safemath{\colI}{\mathscr{I}}
\safemath{\colJ}{\mathscr{J}}
\safemath{\colK}{\mathscr{K}}
\safemath{\colL}{\mathscr{L}}
\safemath{\colM}{\mathscr{M}}
\safemath{\colN}{\mathscr{N}}
\safemath{\colO}{\mathscr{O}}
\safemath{\colP}{\mathscr{P}}
\safemath{\colQ}{\mathscr{Q}}
\safemath{\colR}{\mathscr{R}}
\safemath{\colS}{\mathscr{S}}
\safemath{\colT}{\mathscr{T}}
\safemath{\colU}{\mathscr{U}}
\safemath{\colV}{\mathscr{V}}
\safemath{\colW}{\mathscr{W}}
\safemath{\colX}{\mathscr{X}}
\safemath{\colY}{\mathscr{Y}}
\safemath{\colZ}{\mathscr{Z}}

\safemath{\opA}{\mathbb{A}}
\safemath{\opB}{\mathbb{B}}
\safemath{\opC}{\mathbb{C}}
\safemath{\opD}{\mathbb{D}}
\safemath{\opE}{\mathbb{E}}
\safemath{\opF}{\mathbb{F}}
\safemath{\opG}{\mathbb{G}}
\safemath{\opH}{\mathbb{H}}
\safemath{\opI}{\mathbb{I}}
\safemath{\opJ}{\mathbb{J}}
\safemath{\opK}{\mathbb{K}}
\safemath{\opL}{\mathbb{L}}
\safemath{\opM}{\mathbb{M}}
\safemath{\opN}{\mathbb{N}}
\safemath{\opO}{\mathbb{O}}
\safemath{\opP}{\mathbb{P}}
\safemath{\opQ}{\mathbb{Q}}
\safemath{\opR}{\mathbb{R}}
\safemath{\opS}{\mathbb{S}}
\safemath{\opT}{\mathbb{T}}
\safemath{\opU}{\mathbb{U}}
\safemath{\opV}{\mathbb{V}}
\safemath{\opW}{\mathbb{W}}
\safemath{\opX}{\mathbb{X}}
\safemath{\opY}{\mathbb{Y}}
\safemath{\opZ}{\mathbb{Z}}
\safemath{\opZero}{\mathbb{O}}
\safemath{\identityop}{\opI}

\safemath{\veca}{\bma}
\safemath{\vecb}{\bmb}
\safemath{\vecc}{\bmc}
\safemath{\vecd}{\bmd}
\safemath{\vece}{\bme}
\safemath{\vecf}{\bmf}
\safemath{\vecg}{\bmg}
\safemath{\vech}{\bmh}
\safemath{\veci}{\bmi}
\safemath{\vecj}{\bmj}
\safemath{\veck}{\bmk}
\safemath{\vecl}{\bml}
\safemath{\vecm}{\bmm}
\safemath{\vecn}{\bmn}
\safemath{\veco}{\bmo}
\safemath{\vecp}{\bmmp}
\safemath{\vecq}{\bmq}
\safemath{\vecr}{\bmr}
\safemath{\vecs}{\bms}
\safemath{\vect}{\bmt}
\safemath{\vecu}{\bmu}
\safemath{\vecv}{\bmv}
\safemath{\vecw}{\bmw}
\safemath{\vecx}{\bmx}
\safemath{\vecy}{\bmy}
\safemath{\vecz}{\bmz}

\safemath{\veczero}{\bmzero}
\safemath{\vecone}{\bmone}
\safemath{\vecxi}{\bmxi}
\safemath{\veclambda}{\bmlambda}
\safemath{\vecmu}{\bmmu}
\safemath{\vecnu}{\bmnu}

\safemath{\vecomega}{\bmomega}
\safemath{\vectheta}{\bmtheta}
\safemath{\vecphi}{\bmphi}
\safemath{\vecpi}{\bmpi}
\safemath{\vecalpha}{\bmalpha}

\safemath{\matA}{\bA}
\safemath{\matB}{\bB}
\safemath{\matC}{\bC}
\safemath{\matD}{\bD}
\safemath{\matE}{\bE}
\safemath{\matF}{\bF}
\safemath{\matG}{\bG}
\safemath{\matH}{\bH}
\safemath{\matI}{\bI}
\safemath{\matJ}{\bJ}
\safemath{\matK}{\bK}
\safemath{\matL}{\bL}
\safemath{\matM}{\bM}
\safemath{\matN}{\bN}
\safemath{\matO}{\bO}
\safemath{\matP}{\bP}
\safemath{\matQ}{\bQ}
\safemath{\matR}{\bR}
\safemath{\matS}{\bS}
\safemath{\matT}{\bT}
\safemath{\matU}{\bU}
\safemath{\matV}{\bV}
\safemath{\matW}{\bW}
\safemath{\matX}{\bX}
\safemath{\matY}{\bY}
\safemath{\matZ}{\bZ}
\safemath{\matzero}{\bmzero}

\safemath{\matDelta}{\bDelta}
\safemath{\matLambda}{\bLambda}
\safemath{\matPhi}{\bPhi}
\safemath{\matSigma}{\bSigma}
\safemath{\matOmega}{\bOmega}
\safemath{\matTheta}{\bTheta}

\safemath{\matidentity}{\matI}
\safemath{\matone}{\matO}

\safemath{\rnda}{A}
\safemath{\rndb}{B}
\safemath{\rndc}{C}
\safemath{\rndd}{D}
\safemath{\rnde}{E}
\safemath{\rndf}{F}
\safemath{\rndg}{G}
\safemath{\rndh}{H}
\safemath{\rndi}{I}
\safemath{\rndj}{J}
\safemath{\rndk}{K}
\safemath{\rndl}{L}
\safemath{\rndm}{M}
\safemath{\rndn}{N}
\safemath{\rndo}{O}
\safemath{\rndp}{P}
\safemath{\rndq}{Q}
\safemath{\rndr}{R}
\safemath{\rnds}{S}
\safemath{\rndt}{T}
\safemath{\rndu}{U}
\safemath{\rndv}{V}
\safemath{\rndw}{W}
\safemath{\rndx}{X}
\safemath{\rndy}{Y}
\safemath{\rndz}{Z}

\safemath{\rveca}{\bimA}
\safemath{\rvecb}{\bimB}
\safemath{\rvecc}{\bimC}
\safemath{\rvecd}{\bimD}
\safemath{\rvece}{\bimE}
\safemath{\rvecf}{\bimF}
\safemath{\rvecg}{\bimG}
\safemath{\rvech}{\bimH}
\safemath{\rveci}{\bimI}
\safemath{\rvecj}{\bimJ}
\safemath{\rveck}{\bimK}
\safemath{\rvecl}{\bimL}
\safemath{\rvecm}{\bimM}
\safemath{\rvecn}{\bimN}
\safemath{\rveco}{\bomO}
\safemath{\rvecp}{\bimP}
\safemath{\rvecq}{\bimQ}
\safemath{\rvecr}{\bimR}
\safemath{\rvecs}{\bimS}
\safemath{\rvect}{\bimT}
\safemath{\rvecu}{\bimU}
\safemath{\rvecv}{\bimV}
\safemath{\rvecw}{\bimW}
\safemath{\rvecx}{\bimX}
\safemath{\rvecy}{\bimY}
\safemath{\rvecz}{\bimZ}

\safemath{\rvecxi}{\bmxi}
\safemath{\rveclambda}{\bmlambda}
\safemath{\rvecmu}{\bmmu}
\safemath{\rvectheta}{\bmtheta}
\safemath{\rvecphi}{\bmphi}

\safemath{\rmatA}{\bimA}
\safemath{\rmatB}{\bimB}
\safemath{\rmatC}{\bimC}
\safemath{\rmatD}{\bimD}
\safemath{\rmatE}{\bimE}
\safemath{\rmatF}{\bimF}
\safemath{\rmatG}{\bimG}
\safemath{\rmatH}{\bimH}
\safemath{\rmatI}{\bimI}
\safemath{\rmatJ}{\bimJ}
\safemath{\rmatK}{\bimK}
\safemath{\rmatL}{\bimL}
\safemath{\rmatM}{\bimM}
\safemath{\rmatN}{\bimN}
\safemath{\rmatO}{\bimO}
\safemath{\rmatP}{\bimP}
\safemath{\rmatQ}{\bimQ}
\safemath{\rmatR}{\bimR}
\safemath{\rmatS}{\bimS}
\safemath{\rmatT}{\bimT}
\safemath{\rmatU}{\bimU}
\safemath{\rmatV}{\bimV}
\safemath{\rmatW}{\bimW}
\safemath{\rmatX}{\bimX}
\safemath{\rmatY}{\bimY}
\safemath{\rmatZ}{\bimZ}

\safemath{\rmatDelta}{\bimDelta}
\safemath{\rmatLambda}{\bimLambda}
\safemath{\rmatPhi}{\bimPhi}
\safemath{\rmatSigma}{\bimSigma}
\safemath{\rmatOmega}{\bimOmega}
\safemath{\rmatTheta}{\bimTheta}

\usepackage{amssymb}
\usepackage{amsfonts}
\usepackage{mathrsfs}
\usepackage{xspace}
\usepackage{bm}

\newenvironment{textbmatrix}{	\setlength{\arraycolsep}{2.5pt}%
								\big[\begin{matrix}}{\end{matrix}\big]%
								\raisebox{0.08ex}{\vphantom{M}}}

\def\be{\begin{equation}}
\def\ee{\end{equation}}
\def\een{\nonumber \end{equation}}
\def\mat{\begin{bmatrix}}
\def\emat{\end{bmatrix}}
\def\btm{\begin{textbmatrix}}
\def\etm{\end{textbmatrix}}

\def\ba#1\ea{\begin{align}#1\end{align}}
\def\bas#1\eas{\begin{align*}#1\end{align*}}
\def\bs#1\es{\begin{split}#1\end{split}} 
\def\bg#1\eg{\begin{gather}#1\end{gather}} 
\def\bi#1\ei{\begin{itemize}#1\end{itemize}}

\DeclareMathOperator{\Prob}{\opP}			%
\DeclareMathOperator{\Exop}{\opE}			%

\safemath{\dirac}{\delta}					%
\safemath{\krond}{\dirac}					%
\safemath{\upto}{\uparrow}
\safemath{\downto}{\downarrow}
\safemath{\iu}{j}							%
\safemath{\ev}{\lambda}						%
\safemath{\hilseqspace}{l^{2}}				%
\newcommand{\banachfunspace}[1]{\setL^{#1}}	%
\safemath{\hilfunspace}{\banachfunspace{2}}	%

\safemath{\SNR}{\text{\sc snr}} 				%
\safemath{\No}{N_0}							%
\safemath{\Es}{E_s}							%
\safemath{\Eb}{E_b}							%
\safemath{\EbNo}{\frac{\Eb}{\No}}
\safemath{\EsNo}{\frac{\Es}{\No}}

\DeclareMathOperator{\CHop}{\ensuremath{\opH}} %
\safemath{\tvir}{\rndh_{\CHop}}				%
\safemath{\tvtf}{\rndl_{\CHop}}				%
\safemath{\spf}{\rnds_{\CHop}}				%
\safemath{\bff}{H_{\CHop}}					%

\safemath{\ircf}{r_{h}}						%
\safemath{\tftvcf}{r_{s}}					%
\safemath{\tfcf}{r_{l}}						%
\safemath{\bfcf}{r_{H}}						%

\safemath{\tcorr}{c_h}						%
\safemath{\scf}{c_{s}}						%
\safemath{\tfcorr}{c_{l}}					%
\safemath{\fcorr}{c_{H}}						%

\safemath{\mi}{I}							%
\safemath{\capacity}{C}						%

\safemath{\normal}{\mathcal{N}}			%
\safemath{\jpg}{\mathcal{CN}}			%
\safemath{\mchain}{\leftrightarrow}		%

\safemath{\dB}{\,\mathrm{dB}}
\safemath{\dBm}{\,\mathrm{dBm}}
\safemath{\Hz}{\,\mathrm{Hz}}
\safemath{\kHz}{\,\mathrm{kHz}}
\safemath{\MHz}{\,\mathrm{MHz}}
\safemath{\GHz}{\,\mathrm{GHz}}
\safemath{\s}{\,\mathrm{s}}
\safemath{\ms}{\,\mathrm{ms}}
\safemath{\mus}{\,\mathrm{\mu s}}
\safemath{\ns}{\,\mathrm{ns}}
\safemath{\meter}{\,\mathrm{m}}
\safemath{\mm}{\,\mathrm{mm}}
\safemath{\cm}{\,\mathrm{cm}}
\safemath{\m}{\,\mathrm{m}}
\safemath{\W}{\,\mathrm{W}}
\safemath{\J}{\,\mathrm{J}}
\safemath{\K}{\,\mathrm{K}}
\safemath{\bit}{\,\mathrm{bit}}

\safemath{\define}{=}			%

\safemath{\equivalent}{\sim}
\safemath{\distas}{\sim}					%
\newcommand{\AND}{\,\mathbf{and}\,}		%
\safemath{\sdiff}{\Delta}				%

\safemath{\reals}{\mathbb{R}}
\safemath{\positivereals}{\reals_{+}}
\safemath{\integers}{\mathbb{Z}}
\safemath{\posint}{\integers_{+}}
\safemath{\naturals}{\mathbb{N}}
\safemath{\posnaturals}{\naturals_{+}}
\safemath{\complexset}{\mathbb{C}}
\safemath{\rationals}{\mathbb{Q}}

\safemath{\vecF}{\mathbf{F}}

\ShortHeadings{VI Approach to Independent Learning in Static MFGs}{VI Approach to Independent Learning in Static MFGs}
\title{
A Variational Inequality Approach to Independent Learning in Static Mean-Field Games
}

\author{\name Batuhan Yardim \email alibatuhan.yardim@inf.ethz.ch \\
     \addr Department of Computer Science\\
          ETH Z\"urich\\
          Z\"urich, Switzerland\\
     \AND
     \name Semih Cayci   \email cayci@mathc.rwth-aachen.de \\
     \addr Department of Mathematics\\
          RWTH Aachen University \\
          Aachen, Germany\\
     \AND
     \name Niao He   \email niao.he@inf.ethz.ch \\
     \addr Department of Computer Science\\
          ETH Z\"urich\\
          Z\"urich, Switzerland
}

\begin{document}

\maketitle

\begin{abstract}%
Competitive games involving thousands or even millions of players are prevalent in real-world contexts, such as transportation, communications, and computer networks. However, learning in these large-scale multi-agent environments presents a grand challenge, often referred to as the "curse of many agents". In this paper, we formalize and analyze the Static Mean-Field Game (SMFG) under both full and bandit feedback, offering a generic framework for modeling large population interactions while enabling independent learning.

We first establish close connections between SMFG and variational inequality (VI),  showing that SMFG can be framed as a VI problem in the infinite agent limit. Building on the VI perspective, we propose independent learning and exploration algorithms that efficiently converge to approximate Nash equilibria, when dealing with a finite number of agents.   Theoretically, we provide explicit finite sample complexity guarantees for independent learning across various feedback models in repeated play scenarios, assuming (strongly-)monotone payoffs. Numerically, we validate our results through both simulations and real-world applications in city traffic and network access management. 

\end{abstract}


\section{Introduction}

Multi-agent reinforcement learning (MARL) has been an active area of research, with a vast range of successful applications in games such as Chess, Shogi \citep{silver2017mastering}, Go \citep{silver2018general}, Stratego \citep{perolat2022mastering}, as well as real-world applications for instance in robotics \citep{matignon2007hysteretic}.
However, the applications of MARL to games of a much larger scale involving thousands or millions of agents remain a theoretical and practical challenge \citep{wang2020breaking}.

Despite this limitation, competitive games with many players are ubiquitous and typically high-stakes.
In many real-world games, extremely large-scale competitive multi-agency is the rule rather than the exception: for instance, when commuting every morning between cities using infrastructure shared with millions of other commuters,
when periodically accessing an internet resource by querying servers used simultaneously by many users,
or when sending information over common communication channels.
Crucially, in such settings, learning can only feasibly occur in a \emph{decentralized manner} due to the massive scale of the problem. 
Each agent can only utilize their local observations (often in the form of partial/bandit feedback) to maximize their selfish utilities.
Another common feature in such games is \emph{statelessness}.
For instance, the capacity of an intercity highway or an internet server is only a function of its (and other highways'/servers') immediate load and not a function of time\footnote{Unless the timescale is years/decades, in which long-term degradation effects will become significant.}.
In our setting, we are interested in \emph{independent learning (IL) using repeated play} in such games, that is, algorithms where each agent learns from their own (noisy) interactions without observing others or a centralized coordinator.
IL, despite being theoretically challenging, is natural for such games as centralized control can be an unrealistic assumption for large populations.

\subsection{Game Formulation}\label{sec:game_initial_formulation}
We formalize the static mean-field game (SMFG), the main mathematical object of interest in this work.
The SMFG problem with repeated plays consists of the following.
\begin{itemize}
 \setlength\itemsep{0em}
    \item Finite set of players $\setN = \{ 1, \ldots, N\}$ with $|\setN| = N \in \mathbb{N}_{ > 0}$,
    \item Set of finitely many actions $\setA$, with $|\setA| = K \in \mathbb{N}_{ > 0}$, 
    \item A payoff function $\vecF: \Delta_\setA \rightarrow [0,1]^K$, which maps the occupancy measure of the population over actions to a corresponding payoff in $[0,1]$ for each action. 
\end{itemize}
Here $\Delta_\setA$ denotes the probability simplex over a finite set $\setA$. For $\vecmu\in\Delta_\setA$, $a\in\setA$, we denote the entry of $\vecF(\vecmu)\in\mathbb{R}^{\setA}$  corresponding to action $a$ as  $\vecF(\vecmu, a) = \vecF(\vecmu)(a) \in\mathbb{R}$.

SMFG is played across multiple rounds where players are allowed to alter their strategies in between observations, where at each round $t\in \mathbb{N}_{\geq 0}$,
\begin{enumerate} 
 \setlength\itemsep{0em}
    \item Each player $i \in \setN$ picks an action $a_t^i \in \setA$,
    \item The empirical occupancy measure over actions is set to be $\widehat{\vecmu}_t := \frac{1}{N} \sum_{i=1}^N \vece_{a_t^i}$ 
    \item Depending on the feedback model, each agent $i \in \setN$ observes noisy rewards:
    
    \textbf{Full feedback: } The (noisy) payoffs for \emph{each} action $\vecr^i_t := \vecF(\widehat{\vecmu}_t) + \vecn_t^i \in \mathbb{R}^{K}$ or,
    
    \textbf{Bandit feedback: } The (noisy) payoff for their chosen action $r^i_t := \vecr^i_t(a_t^i) \in \mathbb{R}$.
    \item Each agent receives the payoff $r^i_t$.
\end{enumerate}
Here $\vece_a \in \mathbb{R}^\setA$ is the standard unit vector with coordinate $a$ set to 1, and we assume for each action $a\in\setA$ the noise  $\vecn_t^i(a)$ are independent for all $i \in \setN, t \geq 0$ and have zero mean and variance upper bounded by $\sigma^2$.
Intuitively, SMFG models games where the payoff obtained from choosing an action depends on how the population is statistically distributed over actions, without distinguishing between particular players.
Hence, each agent observes a noisy version of the payoff of their action (or the payoff of all actions with full feedback) under the current empirical occupancy over actions.

We motivate the relevance of the SMFG framework with three examples from real-world applications.
\begin{enumerate}
    \item \textbf{Network resource management.}
Assume there are $K$ resources (e.g., servers or computational nodes) available on a computer network shared by a large number of (say, $N \gg K$) users.
These resources might have varying capabilities, hence the delay in accessing each will be a non-linear function of the overall load.
At each time step, each user tries to access a resource and experiences a delay that increases with the number of users trying to access the same node.
The Tor network application is a particularly good example of SMFGs: the network has 2-3 million users and functions by each user choosing an entry node to use in a decentralized fashion \citep{tormainpage}.

    \item \textbf{Repeated commuting with large populations.}
Every morning, $N$ commuters from city A to city B choose among $K$ routes to drive to their target (typically $N \gg K$), observing only how long it takes them to commute.
The distribution of choices among the population affects how much time each person spends traveling.
The commute times in modern road infrastructure can be a very complex function of overall load \citep{hoogendoorn2013traffic} due to non-trivial feedback loops and adaptive systems such as load-dependent traffic lights.
On the other hand, city traffic is a prime example of finite resources (with potentially complex payoff functions) shared by thousands or millions of inhabitants, especially during heavy commute hours \citep{ambuhl2017empirical}. 
    The empirical study of macroscopic flow diagrams \citep{geroliminis2008existence, ambuhl2017empirical} suggests that at least in aggregate, city traffic flow is explained by its momentary occupancy.
    This empirical observation motivates modeling periodic commutes as a stateless game.
\item \textbf{Multi-player multi-armed bandits with soft collisions.}
Multiplayer MAB have already been studied in the special case where collisions (i.e., multiple players choosing the same arm) result in zero returns.
In many real-world applications, arms used by multiple players yield diminished (but non-zero) utilities when occupied by multiple players.
For instance, in many radio communications applications (Bluetooth, Wi-Fi), common frequencies are automatically used via \emph{time-sharing}, yielding delayed but successful communication when collisions occur.

\end{enumerate}

\subsection{Learning Objectives}
We consider competitive games, that is, each agent aims to maximize their personal expected payoff without regard to social welfare.
We allow agents to play randomized actions (mixed strategies), where each agent $i$ randomly chooses their actions at time $t$ with respect to their mixed strategy $\vecpi_t^i \in \Delta_\setA$.
The primary solution concept for such a game is the Nash equilibrium (NE).

\begin{definition}[Expected payoff, exploitability, Nash equilibrium]\label{def:main_def_nplayer}
For an $N$-tuple of mixed strategies $(\vecpi^1, \ldots, \vecpi^N) \in \Delta_\setA^N$, we define the expected payoff $V^i$ of an agent $i \in \setN$ as
\begin{align*}
    V^i(\vecpi^1, \ldots, \vecpi^N) := \Exop \left[ \vecF\left( \frac{1}{N} \sum_{j=1}^N \vece_{a^j}\right)(a^i) \middle|
a^j \sim \vecpi^j, \forall j = 1, \ldots, N
\right].
\end{align*}
We define the \emph{exploitability} of agent $i$ for the tuple as $\setE^i_{\text{exp}}(\{\vecpi^j\}_{j=1}^N) := \max_{\vecpi'\in \Delta_\setA} V^i(\vecpi', \vecpi^{-i}) - V^i(\vecpi^1, \ldots, \vecpi^N)$. An $N$-tuple $(\vecpi^1, \ldots, \vecpi^N) \in \Delta_\setA^N$ is called a NE if for all $i$, $\setE^i_{\text{exp}}(\{\vecpi^j\}_{j=1}^N)=0$, namely, 
 $V^i(\vecpi^1, \ldots, \vecpi^N) = \max_{\vecpi\in \Delta_\setA} V^i(\vecpi, \vecpi^{-i}),$
and  it is called a $\delta$-NE ($\delta>0$) if 
for all $i$, $\setE^i_{\text{exp}}(\{\vecpi^j\}_{j=1}^N)\leq\delta$, namely, $V^i(\vecpi^1, \ldots, \vecpi^N) \geq \max_{\vecpi\in \Delta_\setA} V^i(\vecpi, \vecpi^{-i}) - \delta$.

\end{definition}

Intuitively, under a mixed strategy profile $(\vecpi^1, \ldots, \vecpi^N)$ that is a Nash equilibrium, no agent has the incentive to deviate from their mixed strategy as in expectation each agent plays optimally with respect to the rest.
A Nash equilibrium of the SMFG can be shown to always exist in this setting without any further assumptions \citep{nash1950non}, as the SMFG is also an $N$-player normal form game when $\setA$ is finite.

\textbf{Goal.}
With the SMFG problem definition and the solution concept introduced, we state our objective: in both feedback models, we would like to find \emph{sample efficient} algorithms\footnote{We rigorously formalize the notion of an algorithm in the different feedback models in the appendix (Section~\ref{section:alg_formalization}).} that learn an \emph{approximate NE} (in the sense of low exploitability) \emph{independently} from repeated plays by $N$ agents when $N$ is \emph{large}.

We will solve the SMFG in the regime when $N$ is very large by leveraging its connection to the corresponding problem in the infinite-player limit, i.e., $N\rightarrow\infty$. 
For this purpose, we introduce the following MF-NE concept.
\begin{definition}[MF-NE]
A mean-field Nash equilibrium (MF-NE) $\vecpi^* \in \Delta_\setA$ associated with payoff operator $\vecF$ is a probability distribution over actions such that $\sum_a \vecpi^*(a) \vecF(\vecpi^*, a) = \max_{\vecpi \in \Delta_\setA} \sum_a \vecpi(a) \vecF(\vecpi^*, a)$.
If it holds that $\sum_a \vecpi^*(a) \vecF(\vecpi^*, a) \geq \max_{\vecpi \in \Delta_\setA} \sum_a \vecpi(a) \vecF(\vecpi^*, a) - \delta$ for some $\delta > 0$, we call $\vecpi^*$ a $\delta$-MF-NE.
\end{definition}
Intuitively, the mean-field limit simplifies the SMFG problem by assuming each agent follows the same policy and replacing the notion of $N$ independent agents with a single agent playing against a continuum of infinitely many identical agents.
Notably, finding an MF-NE is equivalent to solving the following variational inequality (VI) problem corresponding to the operator $\vecF$ and domain $\Delta_\setA$:
\begin{align}\label{eq:mfg_vi_statement}
    \text{Find } \vecpi^* \in \Delta_\setA \text{ s.t. } \vecF(\vecpi^*)^\top (\vecpi^* - \vecpi) \geq 0, \forall \vecpi\in \Delta_\setA. \tag{MF-VI}
\end{align}
The $\delta$-MF-NE is related to the strong gap function of \eqref{eq:mfg_vi_statement}, as $\vecpi^*$ is a $\delta$-MF-NE if and only if $\min _{\vecpi \in \Delta_\setA}\vecF(\vecpi^*)^\top (\vecpi^* - \vecpi) \geq -\delta$.
While the MF-VI is not a concrete game and MF-NE is not strictly speaking a Nash equilibrium of any game, they will be crucial tools in constructing independent learning (IL) algorithms for SMFG.

\subsection{Summary of Contributions}
Based on the VI perspective, we develop algorithms that efficiently converge to approximate NE solutions for the SMFG problem, providing finite-sample, finite-agent, and IL guarantees under different feedback models.
Our key technical contributions are as follows:
\begin{enumerate}
    \item \textbf{VI formulation and approximation.} We formulate the infinite-agent mean-field game limit as a VI.
    We make connections between MF-NE, VIs with regularization, and NE of the SMFG, providing explicit bounds on the bias introduced due to approximating the $N$-agent game with an infinite player game.
    In particular, we show that a solution of the \eqref{eq:mfg_vi_statement} is a $\mathcal{O}(\sfrac{1}{\sqrt{N}})$-NE of the $N$-player game, thereby justifying the usefulness of the approximation \eqref{eq:mfg_vi_statement} when $N$ is large.
    \item \textbf{Independent learning with $N$ agents and full feedback.} Leveraging techniques from optimization and VIs, we analyze independent learning in the $N$-agent setting and prove finite sample bounds with regularized learning. Our work highlights how VI and operator theory can provide insights into IL with full or partial feedback.
    In each feedback model, we propose a $\tau$-Tikhonov regularization scheme that stabilizes the divergence between agent policies without communication, which introduces a $\mathcal{O}(\tau)$ bias.
   When agents have full (though noisy) feedback of reward payoffs, we show that after $T$ rounds of play, our algorithm produces policies with expected exploitability upper bounded by $\mathcal{O} (\frac{\tau^{-2}}{\sqrt{T}}+ \frac{\tau^{-1}}{\sqrt{N}} + \tau)$ with monotone payoffs. 
    For $\lambda$-strongly monotone payoffs, we improve this to $\mathcal{O} (\frac{\tau^{-\sfrac{3}{2}} \lambda^{-\sfrac{1}{2}}}{\sqrt{T}} + \frac{\tau^{-\sfrac{1}{2}} \lambda^{-\sfrac{1}{2}}}{\sqrt{N}} + \tau)$.
    \item \textbf{Learning guarantees with bandit feedback.} We further analyze the learning algorithm under a bandit feedback model and prove finite sample guarantees.
    In this case, we propose a probabilistic exploration scheme that enables agents to obtain low-variance estimates of the payoff operator $\vecF$ without centralized coordination.
    In this scenario, we show that after $T$ rounds of play in expectation the exploitability in the \emph{unregularized} SMFG is bounded by
    $\widetilde{\mathcal{O}}(\frac{N^{\sfrac{3}{4}}}{\sqrt{T}} + \frac{1}{\sqrt[4]{N}})$ for monotone payoffs and by
    $\widetilde{\mathcal{O}}(\frac{N^{\sfrac{3}{4}} \lambda^{-\sfrac{1}{2}}}{\sqrt{T}} + \frac{\lambda^{-\sfrac{1}{2}}}{\sqrt[3]{N}})$ for $\lambda$-strongly monotone payoffs.
    \item \textbf{Empirical validation.} We verify our theoretical results through synthetic and real-world experiments such as the beach bar problem, linear payoffs, city traffic management, and network access within the Tor network.
    These experiments demonstrate the effectiveness of the SMFG formulation and our IL algorithms for addressing games with large populations. 
\end{enumerate}

\subsection{Related Work and Comparisons}
This work is situated at the intersection of multiple areas of research, and we discuss each.

\textbf{Mean-field games.}
Mean-field games (MFG), originally proposed independently by \citet{lasry2007mean} and \citet{huang2006large}, have been an active research area in multi-agent RL literature.
MFG is a useful theoretical tool for analyzing a specific class of MARL problems consisting of a large number of players with symmetric (but competitive) interests by formulating the infinite-agent limit.
The MFG formalism has been used successfully to model large-scale problems such as
auctions \citep{pinasco2023learning},
electricity markets \citep{gomes2021mean},
epidemics \citep{aurell2022finite, huang2022game}, 
and carbon markets \citep{carmona2022mean}.

In this work, we analyze a learning problem in a static (or stateless) finite-action MFG.
Closer to our setting, discrete-time dynamic MFGs have been analyzed under various assumptions, reward models, and solution concepts, such as Lasry-Lions games \citep{perrin2020fictitious,perolat2022scaling}, stationary MFG \citep{anahtarci2022q,xie2021learning,zaman2023oracle, yardim2023policy}, linear quadratic MFG \citep{guo2022entropy, fu2019actor}, and MFGs on graphs \citep{yang2017learning,gu2021mean}.
However, finite-sample learning guarantees for mean-field games only exist under specific assumptions.
For Lasry-Lions games, convergence guarantees exist with an infinite agent oracle \citep{perrin2020fictitious, zhang2024learning}, while IL guarantees with finite agents in this setting do not exist to the best of our knowledge.
Recently, theoretical learning guarantees for centralized learning on monotone MFGs with graphon structure \citep{zhang2024learning} and potential MFGs \citep{geist2021concave} have been shown.
IL in stationary MFGs has been studied either under a subjective equilibrium solution concept \citep{yongacoglu2022independent} or with strong regularization bias and  poor sample complexity \citep{yardim2023policy}.

\textbf{Variational inequalities (VI).}
VIs and relevant algorithms have been an active area of research in classical and recent optimization literature; see e.g., \citet{nemirovski2004prox, facchinei2003finite, nesterov2007dual,lin2022perseus, kotsalis2022simple}, just to list a few.
VIs have been extensively analyzed in the case of (strongly) monotone operators.
Classical algorithms in monotone VIs include the projected gradient method \citep{facchinei2003finite}, the extragradient method \citep{korpelevich1976extragradient}, and proximal point algorithms \citep{rockafellar1976monotone}.
Several works provide extensions to more general operator classes such as generalized monotone operators \citep{kotsalis2022simple} or pseudo-monotone operators \citep{karamardian1976complementarity}.
A recent survey for methods for solving VIs can be found in \citet{beznosikov2023smooth}.
In this work, we use VIs to formalize the infinite-player limit and use algorithmic tools from VI literature to analyze the more challenging IL setting.
Our algorithms are adaptations of the projected gradient method to the decentralized learning under partial feedback.
However, in the SMFG setting, it is not possible to have oracle access to the VI operator (or even unbiased samples of it), posing an algorithmic and theoretical design challenge.
Strictly speaking, the VI given by \eqref{eq:mfg_vi_statement} exists only as an abstract approximation in our work and not as an oracle model.

\textbf{Decentralized VI algorithms.}
Another setting that is related but orthogonal to our problem is the distributed/decentralized VI problem studied for example by \citet{srivastava2011distributed, mukherjee2020decentralized, kovalev2022optimal}.
In this setting, the VI operator is assumed to decompose into $M$ components $\vecF := \sum_{m=1}^M \vecF_m$ each of which can be evaluated locally by $M$ workers.
Workers are permitted to occasionally transmit information along a communication graph: the communication complexity of algorithms as well as the computational complexity is of interest.
Our setting fundamentally differs from the decentralized VI problem. 
Firstly, the VI operator in our case exists only implicitly as an approximation of the finite-player game dynamics: strictly speaking, the problem becomes a VI only at the infinite-player limit.
Furthermore, the VI operator $\vecF$ in our setting does not admit components that can be independently evaluated by any player due to the game interactions: our stochastic oracle is much more restrictive.
Finally, the independent players cannot communicate at all throughout the repeated plays, introducing the algorithmic challenge of fully decentralized learning.

\textbf{Multi-player MAB (MMAB).}
Our model is related to the MMAB problem.
The standard reward model for MMAB is collisions \citep{anandkumar2011distributed}, where agents receive a reward of 0 if more than one pulls the same arm.
Results have focused on the so-called no-communications setting, proving regret guarantees for the cooperative setting without a coordinator \citep{avner2014concurrent,rosenski2016multi, bubeck2021cooperative}.
Our SMFG model can capture much more general reward models than binary collisions.
Moreover, in our competitive setting, our primary metric is exploitability or proximity to NE rather than regret as typically employed in the collaborative setting of MMAB (similar to the NE for MMAB result established by \citet{lugosi2022multiplayer}).

\textbf{Other work.}
A setting that has a similar set of keywords is mean-field bandits \citep{gummadi2013mean,wang2021mean}, as mentioned in the survey by \citet{lauriere2022learning}.
However, these models do not analyze a \emph{competitive (Nash)} equilibrium, but rather a population steady-state reached by an infinite population under prescribed unknown agent behavior.
In these works, optimality or exploitability is not a concern, hence a direct comparison with our work is not possible (differences detailed in appendix, Section~\ref{sec:detailed_comparison}).
On the other hand, our model is equivalent to decentralized learning on static MFGs \citep{lauriere2022learning} with finitely many agents and partial feedback.
The problem of IL is fundamental in algorithmic game theory and has been separately investigated in zero-sum Markov games \citep{daskalakis2020independent, sayin2021decentralized, ozdaglar2021independent} and potential games \citep{ding2022independent, heliou2017learning,alatur2024independent}.
Another related literature is population games and evolutionary dynamics \citep{sandholm2015population, quijano2017role}, where competitive populations are analyzed with differential equations.
While solution concepts overlap, we are interested in IL with repeated play under partial feedback, not the continuous-time dynamic system.

\subsection{Organization and Notations}
The paper is structured as follows.  In Section~\ref{sec:assumptions_examples}, we introduce our assumptions on the SMFG operator and discuss their theoretical and practical relevance.
In Section~\ref{sec:theoretical_tool}, we theoretically analyze the SMFG at the limit $N\rightarrow\infty$ and make connections to variational inequalities in optimization literature.
In Sections~\ref{sec:expert_feedback_results} and \ref{sec:bandit_feedback_results}, we formulate our algorithms and present the finite-time analysis of \emph{independent learning} in full and bandit feedback cases.
Finally, in Section~\ref{sec:experiments_main}, we experimentally verify our theory in numerical examples and two real-world use cases in city traffic and access to the Tor network.

\textbf{Notation.}
For any $M\in\mathbb{N}_{> 0}$, let $[M] := \{ 1, \ldots, M\}$. 
$\Delta_\setA$ denotes the probability simplex over a finite set $\setA$, and $\Delta_{\setA, N} \subset \Delta_\setA$ is the set of possible empirical measures with $N$ samples, that is, $\Delta_{\setA, N} := \{ \vecu = \{u_i\}_i \in\Delta_\setA | N u_i \in \mathbb{N}_{\geq 0} \}$.
For a set $\setA$ and a map $\vecF: \setX\rightarrow\mathbb{R}^{\setA}$ defined on an arbitrary set $\setX$, for $x\in\setX$, $a\in\setA$, we denote the entry of $\vecF(x)\in\mathbb{R}^{\setA}$ corresponding to $a$ as $\vecF(x, a) \in\mathbb{R}$.
For a vector $\vecu\in\mathbb{R}^{\setA}$, $\vecu(a) \in\mathbb{R}$ denotes its entry corresponding to $a\in\setA$.
$\vecone_K \in \mathbb{R}^K$ denotes a vector with all entries $1$.
For $N$-tuple $\vecv \in \setX^N$ and $u, v \in \setX$, $(v, \vecv^{-i})$ is the $N$-tuple with the $i$-th entry of $\vecv$ replaced by $v$ and $(u,v, \vecv^{-i, -j})$ is the $N$-tuple with the $i$-th and $j$-th entries of $\vecv$ replaced by $u,v$ respectively.
$\vece_a \in \mathbb{R}^\setA$ is the standard unit vector with coordinate $a$ set to 1.
$\mathbb{M}^{D_1,D_2}$ denotes the set of $D_1 \times D_2$ matrices, $\mathbb{S}_{++}^D$ denotes the set of positive definite $D\times D$ symmetric matrices. 
For a compact convex set $K \subset \mathbb{R}^H$, $\Pi_K: \mathbb{R}^H \rightarrow K$ denotes the projection operator.
For an event (i.e. measurable set) $E$ in a generic probability space $(\Omega, \setF, \mathbb{P})$, we denote the its complement $\Omega \setminus E$ as $\overline{E}$.

\section{Assumptions on Payoffs and Examples}\label{sec:assumptions_examples}

In general, finding a NE in normal form games is known to 
be challenging and possibly computationally intractable~\citep{papadimitriou1994complexity, daskalakis2009complexity}. This difficulty persists even in the mean-field regime \citep{yardim2024mean}.
Therefore, we will analyze SMFG under certain structural assumptions on $\vecF$ that are closely aligned with many real-world applications.

\subsection{Assumption on Lipschitz continuity}
Our first assumption is the Lipschitz continuity of the payoff.
In practice, action payoffs should not change drastically with small variations in population behavior: Assumption~\ref{ass:lipschitz} formalizes this intuition.

\begin{assumption}[Lipschitz continuous payoffs]\label{ass:lipschitz}
The payoff map $\vecF: \Delta_\setA \rightarrow [0,1]^K$ is called Lipschitz continuous with parameter $L > 0$ if for any $\vecmu, \vecmu' \in \Delta_\setA$, $\| \vecF(\vecmu) - \vecF(\vecmu')\|_2 \leq L \| \vecmu - \vecmu' \|_2$.
\end{assumption}

As a direct consequence of the Lipschitz continuous payoffs, we have the following technical lemmas, which will be used in the subsequent analysis. 
\begin{lemma}\label{lemma:technical_bound_1N}
For any policy profile $(\vecpi^1, \ldots, \vecpi^N) \in \Delta_\setA^N$, it holds that
 \begin{align*}
| V^i(\vecpi^1, \ldots, \vecpi^N) - \vecpi^{i, \top} \Exop \left[ \vecF(\widehat{\vecmu}) \middle| \substack{
a^j \sim \vecpi^j, \forall j \\
\widehat{\vecmu} := \frac{1}{N} \sum_{j=1}^N \vece_{a^j}} \right] | 
\leq 
\frac{L\sqrt{2}}{N}.
 \end{align*}
\end{lemma}

In addition, we can also obtain the Lipschitz continuity of $V^i$ and $\setE^i_{\text{exp}}$, as stated below. 

\begin{lemma}[$V^i, \setE^i_{\text{exp}}$ are Lipschitz]\label{lemma:phi_lipschitz}
For any $i\in \setN$, the value functions $V^i:\Delta_\setA^N \rightarrow \mathbb{R}$ and exploitability functions $\setE^i_{\text{exp}}: \Delta_\setA^N \rightarrow \mathbb{R}$ are Lipschitz continuous, that is, for any $(\vecpi^1, \ldots, \vecpi^N) \in \Delta_\setA^N, \vecpi,\bar{\vecpi} \in \Delta_\setA$,
\begin{align*}
| V^i(\vecpi, \vecpi^{-j}) - V^i(\bar{\vecpi}, \vecpi^{-j}) | \leq &L_{j,i} \| \vecpi - \bar{\vecpi}\|_2, \\
    | \setE^i_{\text{exp}}(\vecpi, \vecpi^{-j}) - \setE^i_{\text{exp}}(\bar{\vecpi}, \vecpi^{-j}) | \leq &\bar{L}_{j,i} \| \vecpi - \bar{\vecpi}\|_2,
\end{align*}
with the Lipschitz moduli given by
\begin{align*}
    L_{i,j} = \begin{cases}
\sqrt{K}, \text{if } i = j, \\
\frac{2L\sqrt{2K}}{N}, \text{if } i\neq j,
\end{cases}
\qquad 
\bar{L}_{i,j} = \begin{cases}
\sqrt{K}, \text{if } i = j, \\
\frac{4L\sqrt{2K}}{N}, \text{if } i\neq j.
\end{cases}
\end{align*}
\end{lemma}
We defer the proofs to Sections~\ref{app:technical_lemma} and \ref{sec:extended_proof_lema_lipschitz_phi} in the appendix.

\subsection{Assumption on Monotonicity}
The next fundamental assumption, monotonicity, is standard in variational inequality literature \citep{facchinei2003finite} and is similar in form to Lasry-Lions conditions in mean-field games literature (initially studied by \cite{lasry2007mean}, later analyzed with reinforcement learning in MFGs by \cite{perrin2020fictitious, perolat2022scaling} and many others).
\begin{assumption}[Monotone/Strongly monotone payoff]
\label{ass:monotone}
The vector map $\vecF: \Delta_\setA \rightarrow [0,1]^K$ is called monotone if for some $\lambda \geq 0$, for all $\forall \vecmu_1, \vecmu_2 \in \Delta_\setA$, it holds that
\begin{align*}
     \left(\vecF(\vecmu_1) - \vecF(\vecmu_2)\right)^\top (\vecmu_1 - \vecmu_2) &= \sum_{a\in\setA} (\vecmu_1(a) - \vecmu_2(a) ) ( \vecF(\vecmu_1, a) - \vecF(\vecmu_2, a) ) \\
     &\leq - \lambda \| \vecmu_1 - \vecmu_2 \|_2^2.
\end{align*}
Furthermore, if the above holds for some $\lambda > 0$, $\vecF$ is called $\lambda$-strongly monotone.
\end{assumption}
Monotone payoffs, as in the case of Lasry-Lions conditions, can be intuitively thought of as modeling problems where the payoff of an action on average decreases as the occupancy increases.
In comparison to \cite{lasry2007mean, perrin2020fictitious}, since the game is stateless in our case, the monotonicity condition degenerates to ``decreasing payoffs with increasing \emph{action occupancy}''.

While Assumption~\ref{ass:monotone} is abstract, one large class of concrete monotone payoff functions are given by games where actions have payoffs non-increasing with their occupancy and there is no interaction between actions, demonstrated below by Example~\ref{ex:decreasing}.

\begin{example}[Non-increasing payoffs]\label{ex:decreasing}
Let $F_a:[0,1] \rightarrow [0,1]$ for $a\in\setA$ be Lipschitz continuous and non-increasing functions, define $\vecF(\vecmu) := \sum_a F_a(\vecmu(a)) \vece_a$.
Then $\vecF$ is monotone and Lipschitz.
If $F_a$ is also \emph{strictly} decreasing and there exists $\lambda > 0$ such that $|F_a(x) - F_a(x')| \geq \lambda |x-x'|$ for all $x,x'\in[0,1], a \in \setA$, then $\vecF$ is $\lambda$-strongly monotone.
\end{example}
Applications where such an $F_a$ exists are common and intuitive: non-increasing $F_a$ models congestion effects on common resources.
Even theoretically, the existence of a non-increasing $F_a$ in wireless communications can be motivated by well-known results on the capacity of multiple access channels in information theory \citep{el1980multiple, skwirzynski1981new}.
Likewise, the conceptual models of traffic flow (e.g., see the macroscopic flow diagram \citep{loder2019understanding}) imply traffic \emph{speed} is monotonically decreasing with occupancy in aggregate.
As a special case of Example~\ref{ex:decreasing}, the classical multi-player multi-armed bandits setting with collisions yields a monotone Lipschitz continuous payoff.
We state this explicitly in the following example. 

\begin{example}[Multi-player MAB with Collisions]
For the $N$-player game, for each action $a\in\setA$, take the functions $F_{col}^a: [0,1] \rightarrow [0,1]$ such that
\begin{align*}
    F_{col}^a(x) = \begin{cases}
        \alpha^a , \text{if } x \leq \frac{1}{N}, \\
        \alpha^a N (\frac{2}{N} - x), \text{if } \frac{1}{N} \leq x \leq \frac{2}{N}, \\
        0, \text{if } x \geq \frac{2}{N}.
    \end{cases}
\end{align*}
where $\alpha^a \in [0,1]$ is the expected payoff of the action $a$ when no collision occurs.
Take the payoff map $\vecF(\vecmu) := \sum_{a\in\setA} F_{col}^a(\mu_a) \vece_a$.
The payoff map $\vecF$ is Lipschitz continuous and monotone and corresponds to the classical multi-player MAB with collisions. 

Due to this fact, one interpretation of our SMFG setting is that we allow soft collisions for action payoffs, as actions may yield non-zero but diminished payoffs when chosen by multiple players.
In fact, a model with hard collisions would not be feasible (or realistic) in applications where $N \gg K$ as any policy supported by all actions would lead to collisions with probability approaching $1$.
On the other hand, a general non-increasing function $F_{col}^a$ as in Example~\ref{ex:decreasing} can model more realistic use cases where collisions happen almost surely due to a large number of players/users.
\end{example}

Before concluding the discussion of monotonicity of payoffs, we also compare our payoff assumptions with two other common game-theoretical settings -- potential games\footnote{Not to be confused with $\vecF$ that admits a potential $\Phi$ such that $\vecF = \nabla \Phi$, also called a gradient field.} \citep{monderer1996potential} and their subset,  congestion games \citep{rosenthal1973class}, for which  algorithms with IL guarantees already exist \citep{leonardos2021global}. 
We show that SMFGs with monotone payoffs are not special cases of either of these settings. Specifically, we provide a counter-example of SMFG that does not admit a game potential.

\begin{remark}[SMFG is not a potential game]
The SMFG is a \emph{potential game} if there exists a map $P: \setA^N \rightarrow \mathbb{R}$ such that for all $i\in\setN, \veca = (a_1, \ldots, a_N) \in \setA^N, a\in\setA$, it holds that
\begin{align}\label{eq:potgame}
    V^{i}(a^i, \veca^{-i}) - V^{i}(a, \veca^{-i}) = P(a^i, \veca^{-i}) - P(a, \veca^{-i}),
\end{align}
where $V^{i}(\veca)$ denotes the expected reward of player $i$ when each player $j \in \setN$ (deterministically) plays action $a^j$.
Note that $V^{i}(\veca) = \vecF(\widehat{\vecmu}(\veca))(a^i)$ where $\widehat{\vecmu}(\veca) \in \Delta_\setA$ is the empirical distribution of actions over actions induced by action profile $\veca$.

We provide an example of an SMFG where no such $P$ exists.
Let $\matS \in \mathbb{S}_{++}^{D}$ be a symmetric positive definite matrix, $\vecb\in\mathbb{R}^D$ be an arbitrary vector, and $\matX \in \mathbb{M}^{D, D}$ be a anti-symmetric matrix such that $\matX = -\matX^\top$.
Take the payoff operator
\begin{align*}
    \vecF(\vecmu) = (-\matS + \matX) \vecmu + \vecb,
\end{align*}
which can be trivially shown to be monotone.
In particular, take $N=3, K=3, \setA = \{ a_1, a_2, a_3\}$, and define the reward $\vecF$ as
\begin{align*}
    \vecF \begin{pmatrix}
\mu_1 \\
\mu_2 \\
\mu_3
\end{pmatrix} = \begin{pmatrix}
-\mu_1 - \mu_2 \\
\mu_1 \\
-\mu_3
\end{pmatrix},
\end{align*}
which is monotone.
Assume that a potential $P$ exists and satisfies Equation~(\ref{eq:potgame}).
Then, it would follow that:
\begin{align*}
    V^1(a_2, a_1, a_1) - V^1(a_3, a_1, a_1) &= P(a_2, a_1, a_1) - P(a_3, a_1, a_1), \\
    V^2(a_3, a_1, a_1) - V^2(a_3, a_2, a_1) &= P(a_3, a_1, a_1) - P(a_3, a_2, a_1), \\
    V^1(a_3, a_2, a_1) - V^1(a_2, a_2, a_1) &= P(a_3, a_2, a_1) - P(a_2, a_2, a_1), \\
    V^2(a_2, a_2, a_1) - V^2(a_2, a_1, a_1) &= P(a_2, a_2, a_1) - P(a_2, a_1, a_1) .
\end{align*}
Hence, adding the inequalities above, we have that
\begin{align*}
    0 = & V^1(a_2, a_1, a_1) - V^1(a_3, a_1, a_1)
        + V^2(a_3, a_1, a_1) - V^2(a_3, a_2, a_1) \\
        & + V^1(a_3, a_2, a_1) - V^1(a_2, a_2, a_1)  
        + V^2(a_2, a_2, a_1) - V^2(a_2, a_1, a_1) \\
     = & \vecF((\sfrac{2}{3}, \sfrac{1}{3}, 0), a_2) - \vecF((\sfrac{2}{3}, 0, \sfrac{1}{3}), a_3) +
     \vecF((\sfrac{2}{3}, 0, \sfrac{1}{3}), a_1) - \vecF((\sfrac{1}{3}, \sfrac{1}{3}, \sfrac{1}{3}), a_2) \\
     & + \vecF((\sfrac{1}{3}, \sfrac{1}{3}, \sfrac{1}{3}), a_3) - \vecF((\sfrac{1}{3}, \sfrac{2}{3}, 0), a_2)
     + \vecF((\sfrac{1}{3},\sfrac{2}{3},0), a_2) - \vecF((\sfrac{2}{3},\sfrac{1}{3},0), a_1) \\
     = &  \sfrac{2}{3} -  (-\sfrac{1}{3}) 
     + (-\sfrac{2}{3}) - \sfrac{1}{3} 
     + (-\sfrac{1}{3}) - \sfrac{1}{3}
     + \sfrac{1}{3} - (-1) 
     \neq  0,
\end{align*}
leading to a contradiction.
As a result,  no such potential $P$ could exist, and this (monotone) SMFG cannot be a potential game.
\end{remark}

\subsection{Assumption on Stochastic Noise}
Finally, our results require the standard assumption of independent noise with bounded variance. 
We formalize this below in Assumption~\ref{ass:noise}. 
\begin{assumption}[Independent, bounded variance noise]
\label{ass:noise}
We assume that the payoff noise is zero mean, has entrywise variance upper bounded by $\sigma^2$, and is independent.
That is, we assume the following hold:
\begin{enumerate}
    \item For all $t \geq 0, i \in \setN$ and $a\in \setA$, it holds that $\Exop[ \vecn_t^i(a) ] = 0$ and $\Exop[ \vecn_t^i(a)^2 ] \leq \sigma^2$, 
    \item $\left\{\vecn_t^i(a)\right\}_{i\in  \setN, t \geq 0, a \in \setA}$ are independent random variables,
    \item For any $t \geq 0, i\in \setN$, the random vector $\vecn_t^i$ is independent from $\{a_{t'}^j\}_{t' \leq t, j \in \setN}$ (i.e., actions taken up to step $t$ by all players).
\end{enumerate}
\end{assumption}

The assumption of independent noise with bounded variance is common in bandits literature \citep{anandkumar2011distributed, avner2014concurrent, bubeck2021cooperative}.
We emphasize that the noise vectors $\vecn_t^i$ have entry-wise bounded variance by $\sigma^2$ in Assumption~\ref{ass:noise}.
In optimization or variational inequality settings, the assumption of independent sources of noise in stochastic oracles is also standard, although typically a bound on $\Exop[\| \vecn_t^i \|_2^2]$ is assumed \citep{  juditsky2011solving}.

\section{The VI Approximation as $N\rightarrow\infty$}\label{sec:theoretical_tool}

Our first set of theoretical results presented in this chapter makes the connection between the NE and a solution of \eqref{eq:mfg_vi_statement} explicit.
We will show that solutions of \eqref{eq:mfg_vi_statement} form good approximations of the true NE in the $N$-player game if $N$ is large.

\begin{remark}[Existence and Uniqueness of MF-NE]
\label{remark:vi_existence}
Let $\vecF:\Delta_\setA \rightarrow [0,1]^K$ be a continuous function.
Then $\vecF$ has at least one MF-NE $\vecpi^*$, and the set of MF-NE is compact.
Furthermore, if $\vecF$ is also $\lambda$-strongly monotone for some $\lambda > 0$, then the MF-NE is unique.
This can be seen as follows.
The MF-NE corresponds to solutions of the VI: $\forall \vecpi \in \Delta_\setA, \vecF(\vecpi^*)^\top (\vecpi^* - \vecpi) \geq 0$.
The domain set $\Delta_\setA$ is compact and convex, and the assumption that $\vecF$ is continuous yields the existence of a solution using Corollary~2.2.5 of \cite{facchinei2003finite}.
For uniqueness in the case of strong monotonicity, see Theorem~2.3.3 of \cite{facchinei2003finite}.
\end{remark}

The following theorem shows that the solution of \eqref{eq:mfg_vi_statement}, when deployed by all players, is a $\mathcal{O}\left(\sfrac{1}{\sqrt{N}}\right)$ solution of the $N$-player game.
Therefore, the MF-NE solution will be an arbitrarily good approximation of the true NE when $N\rightarrow\infty$, and the bias introduced by studying the $N$-player game can be explicitly quantified.

\begin{theorem}\label{theorem:mfg_ne}
    Let $\vecF$ be $L$-Lipschitz, $\delta\geq 0$ arbitrary, and let $\vecpi^*$ be a $\delta$-MF-NE.
Then, the strategy profile $(\vecpi^*, \ldots, \vecpi^*) \in \Delta_\setA^N$ is a $\mathcal{O}\left(\delta + \frac{L}{\sqrt{N}}\right)$-NE of the $N$-player SMFG. 
\end{theorem}

\begin{proof}
Firstly, define the independent random variables $a^j \sim \vecpi^*$ for all $j\in\setN$ for a $\delta$-MF-NE $\vecpi^* \in \Delta_\setA$.
Define the random variable $\widehat{\vecmu} := \sfrac{1}{N} \sum_{j=1}^N \vece_{a^j}$, which is the empirical distribution of players over actions in a single round of an SMFG.
The proof will proceed by formally proving that if $N$ is large enough, then $\Exop[\vecF(\widehat{\vecmu})] \approx \vecF(\vecpi^*)$ and $a^j$ is approximately independent from $\vecF(\widehat{\vecmu})$.

It is straightforward that $\Exop \left[ \widehat{\vecmu}  \right] = \vecpi^*$.
Furthermore, by independence of the random vectors $\vece_{a^j}$, we have
\begin{align*}
    \Exop \left[ \left\| \widehat{\vecmu} - \vecpi^* \right\|_2  \right] \leq
    &\sqrt{\Exop \Big[ \Big\| \frac{1}{N} \sum_{j=1}^N \vece_{a^j} - \vecpi^* \Big\|_2^2  \Big]} 
    \leq  \sqrt{\frac{1}{N^2} \sum_{j=1}^N \Exop \left[  \left\| \vece_{a^j} - \vecpi^* \right\|_2^2 \right]} \leq \frac{2}{\sqrt{N}}.
\end{align*}
Hence, as $\vecF$ is $L$-Lipschitz, we have that
\begin{align}\label{eq:theorem1:ineq2}
\|\Exop[\vecF(\widehat{\vecmu})|a_j\sim \vecpi^*] - \vecF(\vecpi^*)\|_2 \leq \Exop[\|\vecF(\widehat{\vecmu}) - \vecF(\vecpi^*) \|_2] \leq \frac{2L}{\sqrt{N}}.
\end{align}

Now let $i\in \setN$ be arbitrary, and let $\vecpi' \in \Delta_\setA$ be any distribution over actions that satisfies $V^i(\vecpi', \vecpi^{*, -i}) = \max_{\vecpi} V^i(\vecpi, \vecpi^{*, -i})$.
We also define the quantities
\begin{align*}
    \overline{\vecF}_1 = \Exop\left[\vecF(\widehat{\vecmu}) \middle| a^j \sim \vecpi^*, \forall j\in\setN \right], \qquad
    \overline{\vecF}_2 = \Exop\left[\vecF(\widehat{\vecmu}) \middle| a^j \sim \vecpi^* \text{ for } \forall i \neq j, \quad a^i \sim \vecpi' \right].
\end{align*}
We will bound $V^i(\vecpi', \vecpi^{*, -i}) - V^i(\vecpi^*, \vecpi^{*, -i})$.
Combining Lemma~\ref{lemma:technical_bound_1N} and the inequality \eqref{eq:theorem1:ineq2}, we observe
\begin{align*}
    |V^i(\vecpi', \vecpi^{*, -i}) - \vecpi'^\top\vecF(\vecpi^*)| \leq & |V^i(\vecpi', \vecpi^{*, -i}) -  \vecpi'^\top \overline{\vecF}_2 |
    + \Big|\vecpi'^\top \overline{\vecF}_2 - \vecpi'^\top\vecF\Big(\frac{N-1}{N}\vecpi^* + \frac{1}{N} \vecpi'\Big)\Big| \\
    &+ \Big|\vecpi'^\top\vecF\Big(\frac{N-1}{N}\vecpi^* + \frac{1}{N} \vecpi'\Big) - \vecpi'^\top\vecF(\vecpi^*)\Big| \\
    \leq & \frac{L\sqrt{2}}{N} + \frac{2L}{\sqrt{N}} + \frac{2L}{N},
\end{align*}
since $\vecF$ is $L$-Lipschitz.
Likewise, using Lemma~\ref{lemma:technical_bound_1N} once again, we have
\begin{align*}
    |V^i(\vecpi^*, \vecpi^{*, -i}) - \vecpi^{*,\top}\vecF(\vecpi^*)| &\leq |V^i(\vecpi^*, \vecpi^{*, -i}) - \vecpi^{*,\top}\overline{\vecF}_1| + |\vecpi^{*,\top} \overline{\vecF}_1 - \vecpi^{*,\top}\vecF(\vecpi^*)|  \\
    &\leq \frac{L\sqrt{2}}{N} + \frac{2L}{\sqrt{N}}.
\end{align*}
Finally, using the definition of a $\delta$-MF-NE, it holds that
\begin{align*}
V^i(\vecpi', \vecpi^{*, -i}) - V^i(\vecpi^*, \vecpi^{*, -i}) \leq &\vecF(\vecpi^*)^\top (\vecpi' - \vecpi^*) +|V^i(\vecpi^*, \vecpi^{*, -i}) -  \vecpi^{*,\top}\vecF(\vecpi^*)| \\
    & + |V^i(\vecpi', \vecpi^{*, -i}) - \vecpi'^\top\vecF(\vecpi^*)| \\
\leq &\delta + \frac{L(2\sqrt{2} + 4)}{N} + \frac{4L}{\sqrt{N}}.
\end{align*}
\end{proof}

Recall our goal in the context of the $N$-player SMFG is to find policies $\{\vecpi^j\}_{j=1}^N$ with low exploitability $\setE_{\text{exp}}^i$ for all $i$. Theorem~\ref{theorem:mfg_ne} considers that agents adopt the same policy $\vecpi^*$ from solving the VI corresponding to operator $\vecF$ to obtain a low-exploitability approximation. 
We will generalize this result to explicitly bound $\setE_{\text{exp}}^i$ when agent policies can also deviate and when agents can employ regularization.
In our algorithms, regularizing the MF-VI problem will play a crucial role in the IL setting, as it will prevent the policies of agents from diverging when there is no centralized controller synchronizing the policies of agents.
For this reason, our algorithms in the later sections will introduce extraneous regularization to \eqref{eq:mfg_vi_statement} and instead solve the following $\tau$-Tikhonov regularized VI problem:
\begin{align}\label{eq:mfg_rvi_statement}
    \text{Find } \vecpi^* \in \Delta_\setA \text{ s.t. } (\vecF - \tau \matI)(\vecpi^*)^\top (\vecpi^* - \vecpi) \geq 0, \forall \vecpi\in \Delta_\setA. \tag{MF-RVI}
\end{align}

The following theorem quantifies the additional exploitability incurred in the $N$-player game due to (1) extraneous regularization, which is useful for algorithm design, and (2) deviations in agent policies from the MF-NE, potentially due to stochasticity in learning.
Theorem~\ref{theorem:mfgrvi_and_explotability} will be a more useful result later in a learning setting since the learned policies $(\vecpi^1, \ldots, \vecpi^N)$ will only approximate the solution of \eqref{eq:mfg_rvi_statement}.

\begin{theorem}
\label{theorem:mfgrvi_and_explotability}
Let $\vecF$ be monotone, $L$-Lipschitz.
Let $\vecpi_{\tau}^* \in \Delta_\setA$ be the (unique) MF-NE of the regularized map $\vecF - \tau \matI$.
Let $\vecpi^1, \ldots, \vecpi^N \in \Delta_\setA$ be such that $\|\vecpi^i - \vecpi_\tau^*\|_2 \leq \delta$ for all $i$, then it holds that $\setE^i_{\text{exp}}(\{\vecpi^j\}_{j=1}^N) = \mathcal{O}(\tau + \delta + \sfrac{1}{\sqrt{N}})$ for all $i\in\setN$.
\end{theorem}
\begin{proof}
By the Lipschitz continuity of exploitability (Lemma~\ref{lemma:phi_lipschitz}), we have
\begin{align}
    \setE^i_{\text{exp}}(\{ \vecpi^j \}_{j=1}^N) \leq &\setE^i_{\text{exp}}(\{ \vecpi_\tau^* \}_{j=1}^N) + \sqrt{K} \| \vecpi^i - \vecpi_\tau^* \|_2 + \sum_{j\neq i} \frac{4L\sqrt{2K}}{N} \| \vecpi^j - \vecpi_\tau^*\|_2 \notag \\
        \leq & \setE^i_{\text{exp}}(\{ \vecpi_\tau^* \}_{j=1}^N) + \delta \sqrt{K} + 4L\sqrt{2K} \delta. \label{eq:theorem:rviexpbound}
\end{align}
Since $\vecpi_\tau^*$ is the unique MF-NE of the operator $\vecF - \tau \matI$, it holds by definition that
\begin{align*}
    (\vecF - \tau \matI)(\vecpi_{\tau})^\top \vecpi_{\tau} &\geq (\vecF - \tau \matI)(\vecpi_{\tau})^\top \vecpi.
\end{align*}
Organizing both sides, we have
\begin{align*}
    \vecF(\vecpi_{\tau})^\top\vecpi_{\tau} &\geq \vecF(\vecpi_{\tau})^\top \vecpi + \tau \vecpi_{\tau}^\top (\vecpi_{\tau} - \vecpi) \geq \vecF(\vecpi_{\tau})^\top \vecpi - 2\tau,
\end{align*}
as $|\vecpi_{\tau}^\top (\vecpi_{\tau} - \vecpi)| \leq \|\vecpi_{\tau}\|_2 \|\vecpi_{\tau} - \vecpi\|_2 \leq 2$.
Then, $\vecpi^*_{\tau}$ is a $2\tau$-MF-NE for the operator $\vecF$, and by Theorem~\ref{theorem:mfg_ne}, $\setE^i_{\text{exp}}(\{ \vecpi_\tau^* \}_{j=1}^N) \leq \mathcal{O}(\tau + \sfrac{1}{\sqrt{N}})$.
Placing this in~(\ref{eq:theorem:rviexpbound}) proves the theorem.
\end{proof}

To summarize, this section presented key approximation results linking the solutions of \eqref{eq:mfg_vi_statement} and \eqref{eq:mfg_rvi_statement} to the $N$-player exploitability in the SMFG.
The next sections will be devoted to designing sample-efficient IL algorithms.

\section{Convergence in the Full Feedback Case}\label{sec:expert_feedback_results}

We first present an IL algorithm for the full feedback setting, as a first step towards analyzing the more interesting bandit feedback setting.
In this setting, while there is no centralized controller, independent noisy reports of all action payoffs are available to each agent after each round.

\textbf{Is it possible to simply solve MF-RVI in our IL setting?}
Before we present our results, we note the following:
Past works in MFG have already proved approximation results of $N$-agent games by MFG albeit in different settings \citep{saldi2019approximate, yardim2024mean}, but these results do not consider when \emph{learning itself} is carried out with $N$ agents. 
If $N$ agents can not communicate, it is theoretically challenging to approximate the MF-RVI and to tackle bandit feedback.
Most importantly, the IL algorithms formalized in Section~\ref{section:alg_formalization} can not query an operator oracle or maintain a common iterate throughout repeated plays.
Therefore, the approximation properties of \eqref{eq:mfg_vi_statement} do not immediately imply the MF-NE can be learned using VI algorithms.
In this section and the next, we prove the more challenging result of convergence with IL, first under full feedback and later under partial (bandit) feedback.

Our analysis builds up on Tikhonov regularized projected ascent (TRPA).
The TRPA operator is defined as
\begin{align}
    \Gamma^{\eta, \tau}(\vecpi) := \Pi_{\Delta_\setA} ( \vecpi + \eta (\vecF - \tau \matI)(\vecpi) ) = \Pi_{\Delta_\setA} ( (1-\eta\tau) \vecpi + \eta \vecF(\vecpi) ), \tag{TRPA}
\end{align}
for a learning rate $\eta > 0$ and regularization $\tau > 0$.
Intuitively, $\Gamma^{\eta, \tau}$ uses $\vecF$ evaluated at $\vecpi$ to modify action probabilities in the direction of the greatest payoff, incorporating an $\ell_2$ regularizer of $\tau$.
Furthermore, the unique MF-NE $\vecpi^*$ of \eqref{eq:mfg_rvi_statement} is also a fixed point of $\Gamma^{\eta, \tau}$.
The analysis of TRPA is standard and known to converge for monotone $\vecF$ \citep{facchinei2003finite, nemirovski2004prox}, when (stochastic) oracle access to $\vecF$ is assumed.
Naturally, the main complication in applying the method above will be the fact that in the IL setting, agents can not evaluate the operator $\vecF$ arbitrarily, but rather can only observe (a noisy) estimate of $\vecF$ as a function of the empirical population distribution and not of their policy $\vecpi$.
In the full feedback setting, we analyze the following dynamics:
\begin{align}
     \vecpi_0^i := \operatorname{Unif} (\setA) = \frac{1}{K}\vecone_K, \hspace{1em} \vecpi^i_{t+1} =\Pi_{\Delta_\setA} ( (1 - \tau \eta_t) \vecpi_t^i + \eta_t \vecr_t^i ), \tag{TRPA-Full}
\end{align}
for a time varying learning rate $\eta_t$, for each agent $i \in \setN$.
The extraneous $\ell_2$-regularization incorporated in each agent running TRPA-Full is critical for the analysis and convergence in IL, as it allows explicit synchronization of policies of agents without communication.
We state the TRPA-Full algorithm in Algorithm~\ref{alg:full} for reference.

We state the following standard result regarding the TRPA operator without proof, as it will be used later.

\begin{lemma}[cf. Theorem 12.1.2 of \cite{facchinei2003finite}]\label{lemma:contraction_pg}
Assume $\vecF$ is $\lambda\geq 0$-monotone and $L$-Lipschitz.
Then $\Gamma^{\eta,\tau}$ is Lipschitz with constant $\sqrt{1 - 2 (\lambda + \tau) \eta + \eta^2 (L+\tau)^2}$ with respect to the $\ell_2$-norm.
\end{lemma}

\begin{algorithm}
    \caption{TRPA-Full: IL with full feedback algorithm for each agent $i \in \setN$.}\label{alg:full}
    \begin{algorithmic}
    \Require Number of actions $K$, regularization $\tau > 0$, learning rate $\{\eta_t\}_{t=0}^T$, rounds $T > 0$.
    \State $\vecpi^i_0 \leftarrow \frac{1}{K} \vecone$
    \For{$t = 0, \ldots, T-1$}
    \State \text{Play action with current policy $a^i_{t}\sim \vecpi^i_t$}.
    \State \text{Observe payoff $\vecr^i_{t}$}
    \State $\vecpi^i_{t+1} = \Pi_{\Delta_\setA} ( (1 - \tau \eta_t) \vecpi^i_t + \eta_t \vecr^i_t )$
    \EndFor
    \State Return $\vecpi^i_T$
    \end{algorithmic}
\end{algorithm}

For abuse of notation, let $\vecpi^* \in \Delta_\setA$ be the unique solution of \eqref{eq:mfg_rvi_statement} for the regularization $\tau > 0$.
Also define the sigma algebra $\mathcal{F}_{t} := \mathcal{F}(\{ \vecpi_{t'}^i \}_{t'=0, \ldots, t}^{i=1, \ldots, N})$.
We maintain the definitions of the core random variables of the SMFG dynamics introduced in Section~\ref{sec:game_initial_formulation},
\begin{align*}
    \widehat{\vecmu}_t := \frac{1}{N} \sum_{i=1}^N \vece_{a_t^i}, \quad \vecr^i_t := \vecF(\widehat{\vecmu}_t) + \vecn_t^i.
\end{align*}
Under TRPA-Full dynamics, we also define the following random variables that assist our analysis.
\begin{align*}
    \bar{\vecmu}_t &:= \frac{1}{N} \sum_{i=1}^N \vecpi_t^i, \quad e_t^i := \|\vecpi^i_t - \bar{\vecmu}_t \|_2^2, \quad
    u_t^i := \Exop\left[\| \vecpi_t^i - \vecpi^* \|_2^2\right].
\end{align*}
We call $\bar{\vecmu}_t$ the mean policy, $e_t^i$ the mean policy deviation, and $u_t^i$ the expected $\ell_2$-deviation from the regularized MF-NE.
Our goal is to bound the sequence or error terms $u_t^i$; however, the process is complicated by the fact that in general the policy deviations of agents $e_t^i$ are nonzero.
Our strategy is as follows: 
(1) derive a  recursion for $u_t^i$ incorporating the terms $e_t^i$, 
(2) bound the terms $e_t^i$, showing the deviation of the policies of the agents goes to zero in expectation, and
(3) solve the recursion to obtain the convergence rate.

The following lemma captures the first step and provides a recurrence for the evolution of $u_t^i$ under TRPA-Full.

\begin{lemma}[Error recurrence under full feedback]\label{lemma:full_error_recurrence}
    Under TRPA-Full with learning rates $\eta_t$, it holds for $L$-Lipschitz and $\lambda$-strongly monotone $\vecF$ that
    \begin{align*}
    \Exop\left[\| \vecpi_{t+1}^i - \vecpi^* \|_2^2\right] \leq &3\eta_t^2 K(1 + \sigma^2) + 2\eta_t^2(L+\tau)^2 + \frac{4\eta_t L^2 \lambda^{-1}}{ N } \\
        & + 2\eta_t L^2 \lambda^{-1} \Exop\left[e_t^i\right] + \left(1 - 2 \eta_t(\sfrac{\lambda}{2} + \tau)\right) \Exop\left[\| \vecpi_t^i - \vecpi^* \|_2^2\right],
\end{align*}
and for $L$-Lipschitz and monotone $\vecF$ that
\begin{align*}
    \Exop\left[\| \vecpi_{t+1}^i - \vecpi^* \|_2^2\right] \leq &3\eta_t^2 K(1 + \sigma^2) + 2\eta_t^2(L+\tau)^2 + \frac{4\tau^{-1} \eta_t L^2 \delta^{-1}}{ N} \\
        & + \tau^{-1}\eta_t L^2\delta^{-1} \Exop\left[e_t^i\right] + \left(1 - 2\tau \eta_t (1-\delta)\right) \Exop\left[\| \vecpi_t^i - \vecpi^* \|_2^2\right],
\end{align*}
for arbitrary $\delta \in (0,1)$.
\end{lemma}
\begin{proof}
We analyze for any $i\in[N]$ the error term $\| \vecpi_t^i - \vecpi^*\|_2^2$.
Denote $\alpha_t := (1 - \tau \eta_t)$.
For the regularized solution $\vecpi^*$, we have the fixed point result
\begin{align*}
    \Pi_{\Delta_{\setA}} ((1 - \tau \eta_t) \vecpi^* + \eta_t \vecF(\vecpi^*)) = \Pi_{\Delta_{\setA}} (\vecpi^* + \eta_t (\vecF - \tau \matI)(\vecpi^*)) = \vecpi^*.
\end{align*}
The proof strategy is to decompose the $\ell_2$ distance of player policies to $\vecpi^*$ into 3 components using this property.
We can bound the quantity $\| \vecpi_{t+1}^i - \vecpi^*\|_2^2$ by using the non-expansiveness of $\Pi_{\Delta_{\setA}}$:
\begin{align}
    \| \vecpi_{t+1}^i - \vecpi^*\|_2^2 = &\| \Pi_{\Delta_{\setA}}(\alpha_t \vecpi^i_t + \eta_t \vecr_t^i) - \Pi_{\Delta_{\setA}} (\alpha_t \vecpi^* + \eta_t \vecF(\vecpi^*)) \|_2^2 \notag \\
        \leq &\| \alpha_t \vecpi_t^i + \eta_t \vecF(\vecpi_t^i) - \alpha_t \vecpi^* - \eta_t \vecF(\vecpi^*) + \eta_t (\vecr_t^i - \vecF(\vecpi_t^i) )\|_2^2 \notag \\
        = & \eta_t^2\| \vecr_t^i - \vecF(\vecpi_t^i) \|_2^2 + 2\eta_t (\alpha_t (\vecpi_t^i - \vecpi^*) + \eta_t (\vecF(\vecpi_t^i) - \vecF(\vecpi^*)) )^\top (\vecr_t^i - \vecF(\vecpi_t^i)) \notag \\
         & + \|\alpha_t (\vecpi_t^i - \vecpi^*) + \eta_t (\vecF(\vecpi_t^i) - \vecF(\vecpi^*))\|_2^2 \notag \\
         \leq & \underbrace{\eta_t^2\| \vecr_t^i - \vecF(\vecpi_t^i) \|_2^2 + 2\eta_t^2 (\vecF(\vecpi_t^i) - \vecF(\vecpi^*))^\top (\vecr_t^i - \vecF(\vecpi_t^i))}_{(a)} \notag \\
            &+ \underbrace{2\eta_t\alpha_t (\vecpi_t^i - \vecpi^*)^\top (\vecr_t^i - \vecF(\vecpi_t^i))}_{(b)} + \underbrace{\|\alpha_t (\vecpi_t^i - \vecpi^*) + \eta_t (\vecF(\vecpi_t^i) - \vecF(\vecpi^*))\|_2^2 }_{(c)}. \label{ineq:decomp_abc_full_recur}
\end{align}

We analyze the three marked terms separately.
For term $(a)$, using the independence assumption of the noise vectors and Young's inequality, in expectation we obtain
\begin{align*}
    \Exop[(a)] \leq &\eta_t^2 \Exop[\| \vecr_t^i - \vecF(\vecpi_t^i) \|_2^2] + \eta_t^2 \Exop[\|\vecF(\vecpi_t^i) - \vecF(\vecpi^*)\|_2^2] + \eta_t^2 \Exop[\| \vecr_t^i - \vecF(\vecpi_t^i) \|_2^2] \\
    \leq &2\eta_t^2 \Exop[\| \vecr_t^i - \vecF(\vecpi_t^i) \|_2^2] + \eta_t^2 \Exop[\|\vecF(\vecpi_t^i) - \vecF(\vecpi^*)\|_2^2] \\
    \leq & 2\eta_t^2 \Exop[\| \vecr_t^i - \vecF(\widehat{\vecmu}_t)\|_2^2 + \| \vecF(\widehat{\vecmu}_t) - \vecF(\vecpi^i_t)\|_2^2 ] + \eta_t^2 K \\
    \leq & 2\eta_t^2 \sigma^2 K + 3\eta_t^2 K \leq 3\eta_t^2 K(\sigma^2 + 1)
\end{align*}
For the term $(c)$, we obtain
\begin{align*}
    (c) = &\|\alpha_t (\vecpi_t^i - \vecpi^*) + \eta_t (\vecF(\vecpi_t^i) - \vecF(\vecpi^*))\|_2^2 \\
        = & \|(\vecpi_t^i - \vecpi^*) + \eta_t (\vecF(\vecpi_t^i) - \tau \vecpi_t^i - \vecF(\vecpi^*) + \tau \vecpi^* )\|_2^2 \\
        \leq & \left(1 - 2 (\lambda + \tau) \eta_t + (L + \tau)^2 \eta_t^2\right) \| \vecpi_t^i - \vecpi^* \|_2^2 \\
        \leq & \left(1 - 2 (\lambda + \tau) \eta_t \right) \| \vecpi_t^i - \vecpi^* \|_2^2 + 2(L + \tau)^2 \eta_t^2
\end{align*}
where the last inequality holds from the Lipschitz continuity result of Lemma~\ref{lemma:contraction_pg}.

For the term $(b)$, first taking the strongly monotone problem $\lambda > 0$ , we have that
\begin{align*}
(b) = & 2\eta_t\alpha_t (\vecpi_t^i - \vecpi^*)^\top (\vecr_t^i - \vecF(\vecpi_t^i)) \\
 = & 2\eta_t\alpha_t (\vecpi_t^i - \vecpi^*)^\top (\vecr_t^i - \vecF(\widehat{\vecmu}_t)) + 2\eta_t\alpha_t (\vecpi_t^i - \vecpi^*)^\top (\vecF(\widehat{\vecmu}_t) - \vecF(\bar{\vecmu}_t)) \\
    & + 2\eta_t\alpha_t (\vecpi_t^i - \vecpi^*)^\top (\vecF(\bar{\vecmu}_t) - \vecF(\vecpi_t^i)) \\
\leq &2\eta_t\alpha_t \left( \frac{\lambda}{4} \|\vecpi_t^i - \vecpi^* \|_2^2 + \frac{1}{\lambda} \|\vecF(\widehat{\vecmu}_t) - \vecF(\bar{\vecmu}_t)\|_2^2\right) + 2\eta_t\alpha_t \left(\frac{\lambda}{4} \|\vecpi_t^i - \vecpi^* \|_2^2 + \frac{1}{\lambda} \|\vecF(\bar{\vecmu}_t) - \vecF(\vecpi_t^i)\|_2^2 \right) \\
    &+2\eta_t\alpha_t (\vecpi_t^i - \vecpi^*)^\top (\vecr_t^i - \vecF(\widehat{\vecmu}_t)) \\
\leq & \eta_t \lambda \|\vecpi_t^i - \vecpi^* \|_2^2 + 2\eta_t\lambda^{-1}\|\vecF(\widehat{\vecmu}_t) - \vecF(\bar{\vecmu}_t)\|_2^2 + 2\eta_t\lambda^{-1} \|\vecF(\bar{\vecmu}_t) - \vecF(\vecpi_t^i)\|_2^2 \\
    &+2\eta_t\alpha_t (\vecpi_t^i - \vecpi^*)^\top (\vecr_t^i - \vecF(\widehat{\vecmu}_t)),
\end{align*}
which follows from applications of Young's inequality.
For the last three terms we observe:
\begin{align*}
    \Exop\left[2\eta_t\alpha_t (\vecpi_t^i - \vecpi^*)^\top (\vecr_t^i - \vecF(\widehat{\vecmu}_t)) | \mathcal{F}_t\right] = &0, \\
    \Exop[\|\vecF(\widehat{\vecmu}_t) - \vecF(\bar{\vecmu}_t)\|_2^2 | \mathcal{F}_{t}] \leq & L^2 \Exop\left[ \|\widehat{\vecmu}_t - \bar{\vecmu}_t\|_2^2 | \mathcal{F}_{t}\right] \\
    \leq & L^2\Exop\left[\frac{1}{N^2}\|\sum_{i}\vecpi_t^i - \sum_{i} \vece_{a_t^i}\|^2_2 | \mathcal{F}_{t}\right] \\
    = & \frac{L^2}{N^2}\sum_{i}\Exop[\|\vecpi_t^i - \vece_{a_t^i}\|^2_2 | \mathcal{F}_{t}] \leq \frac{2L^2}{N}, \\
    \|\vecF(\bar{\vecmu}_t) - \vecF(\vecpi_t^i)\|_2^2 \leq & L^2 \|\bar{\vecmu}_t - \vecpi_t^i\|_2^2 = L^2 e_t^i.
\end{align*}
The second inequality above follows from the fact that $\widehat{\vecmu}_t$ is the sum of $N$ independent random variables and has expectation $\bar{\vecmu}_t$.
Hence, putting in the bounds for $(a), (b), (c)$ and taking expectations, we obtain the inequality
\begin{align*}
    \Exop\left[\| \vecpi_{t+1}^i - \vecpi^*\|_2^2 \right] \leq & 3 \eta_t^2 K(1 + \sigma^2) + \frac{4\eta_t L^2}{\lambda N} + \frac{2\eta_t L^2}{\lambda} \Exop\left[e_t^i\right] \\
        &+\left(1 - 2 (\sfrac{\lambda}{2} + \tau) \eta_t \right) \Exop\left[\| \vecpi_t^i - \vecpi^* \|_2^2\right] + 2(L + \tau)^2 \eta_t^2.
\end{align*}

Turning back to the monotone case, if $\lambda=0$, vary the upper bound on $(b)$ as follows.
Take any arbitrary $\delta \in (0,1)$.
Then, once again applying Young's inequality, we obtain
\begin{align*}
(b) = & 2\eta_t\alpha_t (\vecpi_t^i - \vecpi^*)^\top (\vecr_t^i - \vecF(\vecpi_t^i)) \\
 = & 2\eta_t\alpha_t (\vecpi_t^i - \vecpi^*)^\top (\vecr_t^i - \vecF(\widehat{\vecmu}_t)) + 2\eta_t\alpha_t (\vecpi_t^i - \vecpi^*)^\top (\vecF(\widehat{\vecmu}_t) - \vecF(\bar{\vecmu}_t)) \\
    & + 2\eta_t\alpha_t (\vecpi_t^i - \vecpi^*)^\top (\vecF(\bar{\vecmu}_t) - \vecF(\vecpi_t^i)) \\
\leq &2\eta_t\alpha_t \left( \frac{\tau\delta}{2} \|\vecpi_t^i - \vecpi^* \|_2^2 + \frac{1}{2\tau\delta} \|\vecF(\widehat{\vecmu}_t) - \vecF(\bar{\vecmu}_t)\|_2^2\right) + 2\eta_t\alpha_t \left(\frac{\tau\delta}{2} \|\vecpi_t^i - \vecpi^* \|_2^2 + \frac{1}{2\tau\delta} \|\vecF(\bar{\vecmu}_t) - \vecF(\vecpi_t^i)\|_2^2 \right)  \\
    &+2\eta_t\alpha_t (\vecpi_t^i - \vecpi^*)^\top (\vecr_t^i - \vecF(\widehat{\vecmu}_t)) \\
\leq & 2 \eta_t \tau\delta \|\vecpi_t^i - \vecpi^* \|_2^2 + \frac{\eta_t}{\tau\delta}\|\vecF(\widehat{\vecmu}_t) - \vecF(\bar{\vecmu}_t)\|_2^2 + \frac{\eta_t}{\tau\delta} \|\vecF(\bar{\vecmu}_t) - \vecF(\vecpi_t^i)\|_2^2 \\
    &+2\eta_t\alpha_t (\vecpi_t^i - \vecpi^*)^\top (\vecr_t^i - \vecF(\widehat{\vecmu}_t)).
\end{align*}
Applying the same bounds for the terms $(a), (c)$ as before yields  the lemma.
\end{proof}

This above lemma has two key features: a dependence on expected mean policy deviation $\Exop\left[e_t^i\right]$, and a term that scales as $\mathcal{O}(\sfrac{1}{N})$.
While the $\mathcal{O}(\sfrac{1}{N})$ term can be anticipated (and asymptotically ignored when $N$ is large) due to the finite-agent mean-field bias (as shown previously in Section~\ref{sec:theoretical_tool}), the term $\Exop\left[e_t^i\right]$ must be controlled separately in the independent learning setting, where policies cannot be synchronized through explicit communication between agents.
The term $\Exop\left[e_t^i\right]$ reflects the core difference of the SMFG model from typical VI stochastic oracles.
Unlike typical VI oracle models, in SMFG the operator $\vecF$ cannot be evaluated at the current iterate $\vecpi^i_t$ of a player $i$  but only approximately at the mean $\bar{\vecmu}_t$.
This is due to decentralized learning: players can only evaluate the current payoffs at the ``mean-iterate'' given by $\vecF(\widehat{\vecmu}_t) \approx \vecF(\bar{\vecmu}_t)$ (up to some stochastic noise) that is almost surely different than their iterates $\{\vecpi^i_t\}_i$ apart from the case with degenerate/zero noise.
Furthermore, Lemma~\ref{lemma:full_error_recurrence} suggests that the algorithmic scheme must guarantee that $\Exop\left[e_t^i\right]$ decays with the rate at least $\mathcal{O}(\sfrac{1}{t})$ to obtain a non-vacuous bound on exploitability.
Taking inspiration from algorithmic stability literature \citep{ahn2022reproducibility, zhang2024optimal}, we utilize a regularization scheme to ensure the iterates of players do not diverge.
The following lemma shows that by introducing explicit regularization $\tau>0$, the expected mean policy deviation can be controlled throughout training.

\begin{lemma}[Policy variations bound]\label{lemma:policy_variations_bound_trpa_full}
    Under TRPA-Full with learning rates $\eta_t :=\frac{\tau^{-1}}{t+2}$, we have $\Exop\left[e_t^i\right] \leq \frac{14 \tau^{-2} K\sigma^2 + 14}{t+2}$.
\end{lemma}
\begin{proof}
Note that for any $i,j\in\setN$ such that $i\neq j$, using the non-expansiveness of the projection operator it holds that
\begin{align*}
    \| \vecpi^i_{t+1} - \vecpi^j_{t+1} \|_2^2 = &  \| \Pi_{\Delta_\setA}((1 - \tau \eta_t) \vecpi^i_t + \eta_t \vecr_t^i) - \Pi_{\Delta_\setA}((1 - \tau \eta_t) \vecpi^j_t + \eta_t \vecr_t^j) \|_2^2 \\
    \leq & \| (1 - \tau \eta_t) \vecpi^i_t + \eta_t \vecr_t^i - (1 - \tau \eta_t) \vecpi^j_t - \eta_t \vecr_t^j \|_2^2 \\
    \leq & \| (1 - \tau \eta_t) (\vecpi^i_t - \vecpi^j_t) + \eta_t (\vecr_t^i - \vecr_t^j) \|_2^2 \\
    = & (1 - \tau \eta_t)^2 \| \vecpi^i_t - \vecpi^j_t \|_2^2 + \eta_t^2 \|  \vecr_t^i - \vecr_t^j \|_2^2 + 2 (1 - \tau \eta_t) \eta_t (\vecpi^i_t - \vecpi^j_t) ^ \top ( \vecr_t^i - \vecr_t^j )
\end{align*}
Taking the conditional expectation on both sides, we obtain
\begin{align*}
    \Exop \left[ \| \vecpi^i_{t+1} - \vecpi^j_{t+1} \|_2^2 | \mathcal{F}_t \right] \leq & (1 - \tau \eta_t)^2 \| \vecpi^i_t - \vecpi^j_t \|_2^2 + \Exop \left[ \eta_t^2 \|  \vecr_t^i - \vecr_t^j \|_2^2 | \mathcal{F}_t \right] \\
        &+ 2 (1 - \tau \eta_t) \eta_t (\vecpi^i_t - \vecpi^j_t) ^ \top \Exop\left[\vecr_t^i - \vecr_t^j | \mathcal{F}_t \right] \\
    = & (1 - \tau \eta_t)^2 \|  \vecpi^i_t - \vecpi^j_t \|_2^2 + \eta_t^2 \Exop \left[\|\vecn^i_t - \vecn^j_t\|_2^2 | \mathcal{F}_t \right] \\
    = & (1 - \tau \eta_t)^2 \|  \vecpi^i_t - \vecpi^j_t \|_2^2 + 2\eta_t^2 K \sigma^2
\end{align*}
almost surely, since we have $\vecr_t^i := \vecF(\widehat{\vecmu}_t) + \vecn^i_t$.
Then, taking the expectation on both sides, 
\begin{align*}
    \Exop \left[ \| \vecpi^i_{t+1} - \vecpi^j_{t+1} \|_2^2 \right] \leq &(1 - \tau \eta_t)^2 \Exop\left[\|  \vecpi^i_t - \vecpi^j_t \|_2^2\right] + 2\eta_t^2 K\sigma^2 \\
    \leq & \left(1 - \frac{1}{t+2}\right)^2 \Exop\left[\|  \vecpi^i_t - \vecpi^j_t \|_2^2\right] + \left(\frac{\tau^{-1}}{t+2}\right)^2 2K\sigma^2 \\
    \leq & \left(1 - \frac{2}{t+2}\right) \Exop\left[\|  \vecpi^i_t - \vecpi^j_t \|_2^2\right] + \frac{1}{(t+2)^2} \Exop\left[\|  \vecpi^i_t - \vecpi^j_t \|_2^2\right] + \frac{2\tau^{-2}K\sigma^2}{(t+2)^2} \\
    \leq & \left(1 - \frac{2}{t+2}\right) \Exop\left[\|  \vecpi^i_t - \vecpi^j_t \|_2^2\right] + \frac{2\tau^{-2}K\sigma^2 + 2}{(t+2)^2}
\end{align*}
To bound the recurrence, we can use the recurrence lemma (Lemma~\ref{lemma:general_recurrence}, noting $\gamma=2, a = 2, u_0 = 0, c_0 = 0, c_1 = 2\tau^{-2}K\sigma^2 + 2$ in its statement):
\begin{align*}
    \Exop \left[ \| \vecpi^i_{t+1} - \vecpi^j_{t+1} \|_2^2 \right] \leq & 5\frac{2\tau^{-2}K\sigma^2 + 2}{(t+2)^2} + 3\frac{2\tau^{-2}K\sigma^2 + 2}{t+2} + \frac{2\tau^{-2}K\sigma^2 + 2}{(t+2)^2} \leq \frac{14 \tau^{-2} K\sigma^2 + 14}{t+2}.
\end{align*}
Then, the expected values of $e_t^i$ can be bounded using:
\begin{align*}
     e_t^i = &\|\vecpi^i_t - \bar{\vecmu}_t \|_2^2 
     =  \left\|\vecpi^i_t - \frac{1}{N} \sum_{j=1}^N \vecpi^j_t  \right\|_2^2 
     \leq \frac{1}{N} \sum_{j=1}^N \| \vecpi^i_t - \vecpi^j_t \|_2^2
\end{align*}
by an application of Jensen's inequality.
Then we have $\Exop\left[e_t^i\right] \leq \frac{14 \tau^{-2} K\sigma^2 + 14}{t+2}$.
\end{proof}

With an explicit bound in expectation on the mean policy deviation $e_t^i$, we can now proceed to the main recurrence for the expected error terms $u_t^i$ in order to prove our main convergence result.
We state our main convergence result for TRPA-Full dynamics in Theorem~\ref{theorem:expert_short} by solving these two recurrences for the monotone and strongly monotone cases.

\begin{theorem}[Convergence, full feedback]\label{theorem:expert_short}
Assume $\vecF$ is Lipschitz, monotone.
Assume $N$ agents run the TRPA-Full update rule for $T$ time steps with learning rates $\eta_t := \frac{\tau^{-1}}{t+2}$ and arbitrary regularization $\tau>0$.
Then it holds for any $i\in[N]$ that $\Exop\left[ \setE^i_{\text{exp}}( \{\vecpi^j_{T}\}_{j=1}^N ) \right] \leq \mathcal{O} (\frac{\tau^{-2}}{\sqrt{T}}+ \frac{\tau^{-1}}{\sqrt{N}} + \tau)$.
Furthermore, if $\vecF$ is $\lambda$-strongly monotone, then $\Exop\left[ \setE^i_{\text{exp}}( \{\vecpi^j_{T}\}_{j=1}^N ) \right] \leq \mathcal{O} (\frac{\tau^{-\sfrac{3}{2}} \lambda^{-\sfrac{1}{2}}}{\sqrt{T}} + \frac{\tau^{-\sfrac{1}{2}} \lambda^{-\sfrac{1}{2}}}{\sqrt{N}} + \tau)$.
\end{theorem}
\begin{proof}
Note that the exploitability in the main statement of the theorem can be related to $u_t^i$ as follows using Lemma~\ref{lemma:phi_lipschitz}:
\begin{align*}
    \Exop[\setE^i_{\text{exp}}(\{\vecpi^j_{t}\}_{j=1}^N)] \leq &\setE^i_{\text{exp}}(\{\vecpi^*\}_{j=1}^N) + \sqrt{K} \Exop[\| \vecpi_t^i - \vecpi^* \|_2] + \frac{4L\sqrt{2K}}{N} \sum_{j\neq i} \Exop[\| \vecpi_t^j - \vecpi^* \|_2] \\
    \leq & \setE^i_{\text{exp}}(\{\vecpi^*\}_{j=1}^N) + \sqrt{K} \sqrt{u_t^i} +  \frac{4L\sqrt{2K}}{N} \sum_{j\neq i} \sqrt{u_t^j} \\
    \leq & \setE^i_{\text{exp}}(\{\vecpi^*\}_{j=1}^N) + \frac{\max\{ \sqrt{K}, 4L\sqrt{2K} \}}{N} \sum_{j} \sqrt{u_t^j}
\end{align*}
Hence the bounds on $u_t^j$ will yield the result of the theorem by linearity of expectation, along with an invocation of Theorem~\ref{theorem:mfgrvi_and_explotability}.

Finally, we solve the recurrences for $\lambda = 0$ and $\lambda > 0$ using Lemma~\ref{lemma:full_error_recurrence}.
For the case $\lambda > 0$, if $\eta_t=\frac{\tau^{-1}}{t+2}$, Lemma~\ref{lemma:full_error_recurrence} provides the bound 
\begin{align*}
    u_{t+1}^i \leq &\frac{ 3\tau^{-2} K(1 + \sigma^2) + 2\tau^{-2}(L+\tau)^2}{(t+2)^2} + \frac{4\tau^{-1} L^2 \lambda^{-1}}{ N (t+2)} + \frac{2\tau^{-1} L^2 \lambda^{-1}}{t+2} \Exop\left[e_t^i\right] \\
        &+ \left(1 - \frac{2 \tau^{-1}(\sfrac{\lambda}{2} + \tau)}{t+2}\right) u_{t}^i. 
\end{align*}
By placing $\Exop\left[ e^i_t\right] \leq \frac{14 \tau^{-2} K\sigma^2 + 14}{t+2}$ due to Lemma~\ref{lemma:policy_variations_bound_trpa_full}, we obtain
\begin{align*}
    u_{t+1}^i \leq &\frac{ 3\tau^{-2}K(1 + \sigma^2) + 2\tau^{-2}(L+\tau)^2 +  28\tau^{-3} K \sigma^2 \lambda^{-1} L^2 + 28 \tau^{-1} L^2 \lambda^{-1}}{(t+2)^2} \\
        &+ \frac{4\tau^{-1} L^2 \lambda^{-1}}{ N (t+2)} + \left(1 - \frac{2}{t+2}\right) u_{t}^i.
\end{align*}
Invoking a generic recurrence lemma (Lemma~\ref{lemma:general_recurrence} in Appendix~\ref{app:basic_inequalities}) leads to the main statement of the theorem.

For the monotone case $\lambda = 0$, we have the recursion:
\begin{align*}
    u_{t+1}^i \leq &\frac{ 3\tau^{-2}K(1 + \sigma^2) + 2\tau^{-2}(L+\tau)^2}{(t+2)^2} + \frac{4\tau^{-2} L^2 \delta^{-1}}{ N (t+2)} + \frac{ \tau^{-2} L^2\delta^{-1}}{ t+2 } \Exop\left[e_t^i\right] \\
        & + \left(1 - \frac{2 (1-\delta)}{t+2}\right) u_{t}^i.
\end{align*}
and once again placing the upper bound on expected policy deviation due to Lemma~\ref{lemma:policy_variations_bound_trpa_full},
\begin{align*}
    u_{t+1}^i \leq &\frac{ 3\tau^{-2}K(1 + \sigma^2) + 2\tau^{-2}(L+\tau)^2 + 28 K \tau^{-4} L^2\delta^{-1}\sigma^2 +28 \tau^{-2} L^2\delta^{-1}}{(t+2)^2} \\
        &+ \frac{2\tau^{-2} L^2 \delta^{-1}}{ N (t+2)} + \left(1 - \frac{2 (1-\delta)}{t+2}\right) u_{t}^i.
\end{align*}
Another invocation of Lemma~\ref{lemma:general_recurrence} concludes the proof, choosing $\delta=\sfrac{1}{4}$.
\end{proof}

This convergence result is stated in terms of exploitability of the unregularized game, leading to an additional $\mathcal{O}(\tau)$ term.
However, in many cases, the Nash equilibrium of the regularized game itself is of interest, in which case the upper bounds should read
$\mathcal{O} (\frac{\tau^{-2}}{\sqrt{T}}+ \frac{\tau^{-1}}{\sqrt{N}})$
and
$\mathcal{O} (\frac{\tau^{-\sfrac{3}{2}} \lambda^{-\sfrac{1}{2}}}{\sqrt{T}} + \frac{\tau^{-\sfrac{1}{2}} \lambda^{-\sfrac{1}{2}}}{\sqrt{N}})$
for the monotone and strongly monotone cases respectively.

In the choice of learning rate $\eta_t$ above, no intrinsic problem parameter is assumed to be known.
Furthermore, due to (1) the regularization $\tau$ and (2) a finite population, a non-vanishing exploitability of $\mathcal{O}(\tau + \sfrac{\tau^{-1}}{\sqrt{N}})$ will be induced in terms of the NE in the monotone case.
While Theorem~\ref{theorem:mfgrvi_and_explotability} readily suggested a bias of order $\mathcal{O}(\sfrac{1}{\sqrt{N}})$ is fundamental, when learning is conducted with finitely many agents Theorem~\ref{theorem:expert_short} shows this is amplified to $\mathcal{O}(\sfrac{\tau^{-1}}{\sqrt{N}})$.
Since for finite population SMFG, there will always be a non-vanishing exploitability in terms of NE due to the mean-field approximation, in practice $\tau$ could be chosen to incorporate an acceptable bias level.
Alternatively, if the exact value of the number of players $N$ is known by each agent, one could choose $\tau$ optimally, to obtain the following corollary.

\begin{corollary}[Optimal $\tau$, full feedback]\label{corollary:expert}
Assume the conditions of Theorem~\ref{theorem:expert_short}.
For monotone $\vecF$, choosing regularization parameter $\tau = \sfrac{1}{\sqrt[4]{N}}$ yields
$\Exop\left[\setE^i_{\text{exp}}(\{\vecpi^j_T\}_{j=1}^N) \right] \leq \mathcal{O}(\frac{\sqrt{N}}{\sqrt{T}} + \frac{1}{\sqrt[4]{N}})$ for any $i$.
For $\lambda$-strongly monotone $\vecF$, choosing $\tau = \sfrac{1}{\sqrt[3]{N}}$ yields $\Exop\left[\setE^i_{\text{exp}}(\{\vecpi^j_T\}_{j=1}^N) \right] \leq \mathcal{O}(\frac{ \lambda^{-\sfrac{1}{2}} \sqrt{N}}{\sqrt{T}} + \frac{\lambda^{-\sfrac{1}{2}}}{\sqrt[3]{N}})$.
\end{corollary}

Even though TRPA-Full solves the regularized (hence strongly monotone) problem, compared to the $\mathcal{O}(\sfrac{1}{T})$ rate in classical strongly monotone VI \citep{kotsalis2022simple} or strongly convex optimization \citep{rakhlin2011making},
our worse $\mathcal{O}(\sfrac{1}{\sqrt{T}})$ time dependence is due to independent learning.
Intuitively, additional time is required to ensure the policies of independent learners are sufficiently close when ``collectively'' evaluating $\vecF$.
The additional dependence of the time-vanishing term on $\sqrt{N}$ is also a result of this fact.
Furthermore, when learning itself is performed by $N$ agents, we note that the bias as a function of $N$ decreases with $\mathcal{O}(\sfrac{1}{\sqrt[4]{N}})$ (or $\mathcal{O}(\sfrac{1}{\sqrt[3]{N}})$ for strongly monotone problems), and not with $\mathcal{O}(\sfrac{1}{\sqrt{N}})$ as Theorem~\ref{theorem:mfg_ne} might suggest.
We leave the question of whether this gap can be improved and whether knowledge of $N$ is required in Corollary~\ref{corollary:expert}, as future work.

\section{Convergence in the Bandit Feedback Case}\label{sec:bandit_feedback_results}

We now move on to the more challenging and realistic bandit feedback case, where agents can only observe the payoffs of the actions they have chosen.
Once again, we analyze the IL setting (or in bandits terminology, the ``no communications'' setting) where agents can not interact or coordinate with each other.
One of the main challenges of bandit feedback with IL in our setting is that it is difficult for each agent to identify itself (i.e., assign itself a unique number between $1,\ldots,N$) so that exploration of action payoffs can be performed in turns.
For instance, in MMAB algorithms, this is typically achieved using variants of the so-called musical chairs algorithm \citep{lugosi2022multiplayer}, which is not available in our formulation.
Instead, we adopt a \emph{probabilistic} exploration scheme where each agent probabilistically decides it is its turn to explore payoffs while the rest of the agents induce the required empirical population distribution on which $\vecF$ should be evaluated.

Our algorithm, which we call TRPA-Bandit, is straightforward and relies on exploration occurring over epochs, where policies are updated once in between epochs using the estimate of action payoffs constructed during the exploration phase.
We use the subscript $h$ to index epochs, which consist of $T_h$ repeated plays indexed by $(h,t)$ for $t=1,\ldots,T_h$.
While we formally presented TRPA-Bandit (Algorithm~\ref{alg:bandit}), the procedure informally is as follows for each agent, fixing an exploration parameter $\varepsilon \in (0,1)$ and an agent $i\in\setN$:
\begin{enumerate}
    \item At each epoch $h$, for $T_h > 0$ time steps, repeat the following:
    \begin{enumerate}
        \item With probability $\varepsilon$, sample uniformly an action $a^i_{h,t}$, observe the payoff $r^i_{h,t}$, and keep the importance sampling estimate $\widehat{\vecr}^i_h \leftarrow K r_{h,t}^i \vece_{a^i_{h,t}}$.
        \item Otherwise (with probability $1-\varepsilon$), sample action according to current policy $\vecpi^i_h$.
    \end{enumerate}
    \item Update the policy using TRPA, $\vecpi^i_{h+1} = \Pi_{\Delta_\setA} ( (1 - \tau \eta_h) \vecpi^i_h + \eta_h \widehat{\vecr}^i_h )$.
    If the agent did not explore this epoch, use $\widehat{\vecr}^i_h = \veczero$.
\end{enumerate}
Intuitively, the probabilistic sampling scheme allows some agents to build a low-variance estimate of $\vecF$, while others simply sample actions with their current policy in order to induce the empirical population distribution at which $\vecF$ should be evaluated.

\begin{algorithm}
    \caption{TRPA-Bandit: IL with bandit feedback algorithm for each agent $i\in\setN$.}\label{alg:bandit}
    \begin{algorithmic}
    \Require Number of actions $K$, regularization $\tau > 0$, exploration probability $\varepsilon > 0$, number of epochs $H$, epoch lengths $\{T_h\}_h$, learning rates $\{\eta_h\}_h$
    \State $\vecpi^i_0 \leftarrow \frac{1}{K} \vecone$
    \For{$h = 0, \ldots, H-1$}
    \State $\widehat{\vecr}^i_h \leftarrow \veczero$ %
    \For{$t = 1, \ldots, T_h$} \Comment{Exploration for $T_h$ rounds before policy update,}
    \State Sample Bernoulli r.v. $X_{h,t}^i \sim \operatorname{Ber}(\varepsilon)$.
    \If{$X_{h,t}^i=1$}
        \State \text{Play action $a^i_{h,t} \sim \operatorname{Unif}(\setA)$ uniformly at random}.
        \Comment{Explore with prob. $\varepsilon$,}
        \State \text{Observe payoff $r^i_{h,t}$}, set  $\widehat{\vecr}^i_h \leftarrow K r^i_{h,t}\vece_{a^i_{h,t}}$.
    \ElsIf{$X_{h,t}^i=0$}
        \State \text{Play action with current policy $a^i_{h,t}\sim \vecpi^i_h$}.
        \Comment{Else, play the current policy.}
    \EndIf
    \EndFor
    \State $\vecpi^i_{h+1} = \Pi_{\Delta_\setA} ( (1 - \tau \eta_h) \vecpi^i_h + \eta_h \widehat{\vecr}^i_h )$
    \Comment{After each epoch, update policy.}
    \EndFor
    \State Return $\vecpi^i_H$
    \end{algorithmic}
    \end{algorithm}

Similar to the full feedback setting, we introduce useful notation used throughout this chapter.
We define the sigma algebra $\mathcal{F}_{h} := \mathcal{F}(\{ \vecpi_{h'}^i \}_{h'=0, \ldots, h}^{i=1, \ldots, N})$.
Adapting the notation from Section~\ref{sec:game_initial_formulation} to the case with multiple epochs, we use
\begin{align*}
\widehat{\vecmu}_{h, t} := \frac{1}{N} \sum_{i=1}^N \vece_{a_{h, t}^i}, \qquad
    \vecr^i_{h, t} := \vecF(\widehat{\vecmu}_{h, t}) + \vecn_{h, t}^i,
\end{align*}
where the updated time indices $h, t$ simply refer to the $t$-th round of play in epoch $h$, and $a_{h, t}^i$ is the action played by player $i$ at epoch $h$, round $t$.
Under the dynamics of Algorithm~\ref{alg:bandit}, we define the following random variables to assist our proofs:
\begin{align*}
    \bar{\vecmu}_h &:= \frac{1}{N} \sum_{i=1}^N \vecpi_h^i, 
    \qquad
    e_t^i := \|\vecpi^i_h - \bar{\vecmu}_h \|_2^2, 
    \qquad
    u_h^i := \Exop\left[\| \vecpi_h^i - \vecpi^* \|_2^2\right],
\end{align*}
which correspond to the mean policy at epoch $h$, the policy deviation of agent $i$ from the mean at epoch $h$ and the $\ell_2$ distance from the MF-NE.
Note that since policies are updated only in between epochs, the above quantities are indexed by epochs $h$ rather than rounds $h, t$.

Our analysis follows the ideas in the case of expert feedback, the main difference being randomization due to the exploration probabilities and the errors being analyzed per epoch rather than per round.
Similar to the full feedback setting, we will proceed in several steps expressed as intermediate lemmas:
(1) we bound the added bias and variance due to the importance sampling strategy,
(2) we obtain a non-linear recursion for the expectation of the terms $e_t^i$ and possible sampling bias, 
(3) we bound the expected differences of each agent's action probabilities $e_t^i$, showing the deviation of the policies of the agents goes to zero in expectation, 
(4) we solve the recursion to obtain the convergence rate.

The next result, Lemma~\ref{lemma:exploration_bias_trpa_bandit}, provides an answer to the first step.
We show that despite the probabilistic exploration step, the estimates $\widehat{\vecr}_h^i$ in Algorithm~\ref{alg:bandit} have low bias and variance.

\begin{lemma}[Exploration bias]\label{lemma:exploration_bias_trpa_bandit}
Under the dynamics of TRPA-Bandit, it holds almost surely for each epoch $h \geq 0$ that
\begin{align*}
    \| \Exop[\widehat{\vecr}_h^i | 
 \mathcal{F}_h ] - \vecF(\varepsilon \frac{1}{K}\vecone + (1-\varepsilon) \bar{\vecmu}_h) \|_2 \leq K^{\sfrac{3}{2}} \sqrt{1 + \sigma^2} \exp\left\{ -\varepsilon T_{h}\right\} + \frac{2L}{N} + \frac{2L}{\sqrt{N}}.
\end{align*}
\end{lemma}

The full proof has been postponed to Appendix~\ref{sec:proof_lemma_bandit_exploration_bias}.
In summary, the proof strategy is to decompose and analyze the bias due to the possibility of no exploration round happening (the term $K^{\sfrac{3}{2}} \sqrt{1 + \sigma^2} \exp\left\{ -\varepsilon T_{h}\right\}$),  the impact of the exploring agent on payoffs (the term $\frac{2L}{N}$), and bias due to the finitely many agents, similar to Theorem~\ref{theorem:mfg_ne} (the term $\frac{2L}{\sqrt{N}}$).
The additional bias due to probabilistic exploration originates from the possibility that no exploration occurs in a given epoch: the probability of this event can be bounded by $\exp\left\{ -\varepsilon T_{h}\right\}$.

Lemma~\ref{lemma:exploration_bias_trpa_bandit} shows that even if the players do not have full feedback, they can obtain low-bias, low-variance estimates of $\vecF(\varepsilon \frac{1}{K}\vecone + (1-\varepsilon) \bar{\vecmu}_h) \approx \vecF(\bar{\vecmu}_h)$ when $\varepsilon$ is small.
It guarantees that even if the agents do not explore each epoch, in expectation the probabilistic exploration scheme yields a low bias if the epoch lengths $T_h$ are logarithmically large: hence, full feedback can be simulated by paying a logarithmic cost.
Therefore, in our epoched exploration scheme, the bias in ``querying'' the payoff operator $\vecF$ due to an exploring population can be controlled by tuning $\varepsilon$ and $T_h$.

We next state the error recursion in Lemma~\ref{lemma:bandit_main_recurrence}, which uses the result of Lemma~\ref{lemma:exploration_bias_trpa_bandit} to construct the main recurrence for the bandit feedback case.

\begin{lemma}[Main recurrence for TRPA-Bandit]\label{lemma:bandit_main_recurrence}
Under TRPA-Bandit dynamics, it holds for any $i \in \{1, \ldots, N \}$ and each epoch $h\geq 0$ that
\begin{align*}
    \Exop\left[\| \vecpi_{h+1}^i - \vecpi^* \|_2^2\right] \leq & 4 \eta_h^2 K^3(1 + \sigma^2) + 8\eta_h^2(L+\tau)^2 + 8 K^{\sfrac{3}{2}}\eta_h \sqrt{1+\sigma^2}  \exp\{-\varepsilon T_h\} \\
        &+128\eta_h\lambda^{-1} L^2 N^{-1} + 16\eta_h\lambda^{-1}L^2\varepsilon^2 + 2\eta_h\lambda^{-1} L^2 \Exop\left[e_h^i\right] \\
        &+\left(1 - 2 \eta_h(\sfrac{\lambda}{2} + \tau)\right) \Exop\left[\| \vecpi_h^i - \vecpi^* \|_2^2\right],
\end{align*}
for strongly monotone $\lambda >0$ payoffs, and
\begin{align*}
    \Exop\left[\| \vecpi_{h+1}^i - \vecpi^* \|_2^2\right] \leq &  4\eta_h^2 K^3(\sigma^2 + 1) + 8 \eta_h^2 (L+\tau)^2 + 8 K^{\sfrac{3}{2}} \eta_h \sqrt{1+\sigma^2} \exp\{-\varepsilon T_h\} \\
    &+64\tau^{-1} \eta_h \delta^{-1}L^2 N^{-1}+8\tau^{-1} \eta_h \delta^{-1}L^2 \varepsilon^{2} + \tau^{-1} \eta_h\delta^{-1}L^2 \Exop\left[e_h^i\right] \\  
        &+ \left(1 - 2 \tau \eta_h (1-\delta)\right) \Exop\left[\| \vecpi_h^i - \vecpi^* \|_2^2\right],
\end{align*}
for monotone payoffs for arbitrary $\delta \in (0,1)$.
\end{lemma}
Once again, the full proof has been postponed to Appendix~\ref{sec:proof_lemma_bandit_recurrence}.
The proof of Lemma~\ref{lemma:bandit_main_recurrence} follows a similar path as in the recurrence in the full feedback case (Lemma~\ref{lemma:full_error_recurrence}), with the exception that $\vecr_{h,t}^i$ has been replaced by the importance sampling estimator $\widehat{\vecr}^i_h$.
In the decomposition due to Inequality~\eqref{ineq:decomp_abc_full_recur}, the analysis of term (a) remains the same, whereas the terms (b), (c) must be further analyzed using Lemma~\ref{lemma:exploration_bias_trpa_bandit} to account for deviations between $\vecr_{h,t}^i$ and $\widehat{\vecr}^i_h$, as well as the $\varepsilon$ fraction of the population now exploring each round.

The recurrence in Lemma~\ref{lemma:bandit_main_recurrence} is similar in form to the full feedback case (Lemma~\ref{lemma:full_error_recurrence}), apart from the term $8 K^{\sfrac{3}{2}}\eta_h \sqrt{1+\sigma^2}  \exp\{-\varepsilon T_h\}$ due to the exploration scheme.
However, keeping the exploration epoch lengths $T_h$ logarithmically large can make this term small.
Furthermore, once again the recursion produces a dependence on expected mean policy deviations, $\Exop[e_h^i]$.
Hence, the expected policy deviation $\Exop[e_h^i]$ at epoch $h$ must be bounded once again at a rate $\mathcal{O}(\sfrac{1}{h})$ in order to obtain a non-vacuous upper bound on exploitability.
As in the full feedback case, we employ regularization to ensure $\Exop[e_h^i]$ is small.
The next lemma presents our upper bound.

\begin{lemma}[Policy deviation under TRPA-Bandit]\label{lemma:bandit_pol_deviation}
Under TRPA-Bandit dynamics, with learning rates $\eta_h:=\frac{\tau^{-1}}{h+2}$, arbitrary exploration rate $\varepsilon > 0$ and epoch lengths $T_h := \lceil \varepsilon^{-1} \log(h+2) \rceil$
it holds for any $i,j \in \{1, \ldots, N \}, i\neq j$ and each epoch $h\geq 0$ that
\begin{align*}
    \Exop[e_h^i] \leq \frac{24\tau^{-2} K^3 (\sigma^2 + 2) + 48\tau^{-2} + 24}{h+1} + \frac{16\tau^{-2} L ^ 2}{N^2}.
\end{align*}
\end{lemma}
The proof of Lemma~\ref{lemma:bandit_pol_deviation} follows similar ideas to Lemma~\ref{lemma:policy_variations_bound_trpa_full}, while accounting for (a) the increased variance due to importance sampling, and (b) potential further deviation between agent policies due to the $\mathcal{O}(\sfrac{1}{N})$ impact of exploration on $\widehat{\vecmu}_{h,t}$.
In particular, an additional source of policy deviation occurs when an agent does not explore in a given epoch, in which case the payoff estimator is uninformative ($\widehat{\vecr}_h^i = \veczero$) causing additional policy deviation.
The full proof has been postponed to Appendix~\ref{sec:proof_lemma_bandit_pol_dev}.

In the case of TRPA-Bandit, due to the additional variance of probabilistic exploration, the policy deviations between agents might be larger: compare the upper bounds of Lemma~\ref{lemma:bandit_pol_deviation} and Lemma~\ref{lemma:policy_variations_bound_trpa_full}.
In particular, the upper bound of Lemma~\ref{lemma:bandit_pol_deviation} contains a non-vanishing term unlike Lemma~\ref{lemma:policy_variations_bound_trpa_full}.
Nevertheless, they can still be controlled of order $\mathcal{O}(\sfrac{1}{(h+1)} + \sfrac{1}{N^2})$, where the additional $\mathcal{O}(\sfrac{1}{N^2})$ term compared to TRPA-Full vanishes very quickly when $N$ is large.

With these intermediate lemmas established, we state and prove the main convergence result for TRPA-Bandit  Theorem~\ref{theorem:bandit_short}, the main result of this work.
We provide asymptotic rates for brevity, although the proof of the theorem provides explicit bounds.

\begin{theorem}[Convergence, bandit feedback]\label{theorem:bandit_short}
Assume $\vecF$ is Lipschitz, monotone.
Assume $N$ agents run TRPA-Bandit (Algorithm~\ref{alg:bandit}) for $T$ time steps with regularization $\tau>0$ and exploration parameter $\varepsilon > 0$, and agents return policies $\{\vecpi^i\}_i$ after executing Algorithm~\ref{alg:bandit}.
Then, for any agent $i \in \setN$ that $\Exop\left[ \setE^i_{\text{exp}}( \{\vecpi^j\}_{j=1}^N ) \right] \leq \widetilde{\mathcal{O}} (\frac{\tau^{-2}\varepsilon^{-\sfrac{1}{2}}}{\sqrt{T}} + \tau^{-1}\varepsilon + \tau + \frac{\tau^{-1}}{\sqrt{N}} + \frac{\tau^{-\sfrac{3}{2}} }{N} )$.
If $\vecF$ is $\lambda$-strongly monotone, then $\Exop\left[ \setE^i_{\text{exp}}( \{\vecpi^j\}_{j=1}^N ) \right] \leq \widetilde{\mathcal{O}} (\frac{\tau^{-\sfrac{3}{2}} \lambda^{-\sfrac{1}{2}} \varepsilon^{-\sfrac{1}{2}}}{\sqrt{T}} + \tau^{-\sfrac{1}{2}} \lambda^{-\sfrac{1}{2}} \varepsilon + \tau + \frac{\tau^{-\sfrac{1}{2}} \lambda^{-\sfrac{1}{2}}}{\sqrt{N}} + \frac{\tau^{-1} \lambda^{-\sfrac{1}{2}} }{N} )$.
\end{theorem}
 The proof follows from a straightforward combination of Lemmas~\ref{lemma:exploration_bias_trpa_bandit}, \ref{lemma:bandit_main_recurrence}, and \ref{lemma:bandit_pol_deviation} as in the full feedback case.
The exact bounds and proof are in the appendix, Section~\ref{sec:bandit_theorem_full}.

Once again, the values of $\tau$ and exploration probability $\varepsilon$ can be chosen to incorporate tolerable exploitability.
In the case where the number of participants $N$ in the game is known, the following corollary indicates the asymptotically optimal choices for the hyperparameters.

\begin{corollary}[Optimal $\varepsilon, \tau$, bandit feedback]\label{corollary:bandit}
Assume the conditions of Theorem~\ref{theorem:bandit_short} for $N$ agents running TRPA-Bandit.
For monotone $\vecF$, choosing $\tau = \sfrac{1}{\sqrt[4]{N}}$ and $\varepsilon = \sfrac{1}{\sqrt{N}}$ yields
$\Exop\left[\setE^i_{\text{exp}}(\{\vecpi^j\}_{j=1}^N) \right] \leq \widetilde{\mathcal{O}}(\frac{N^{\sfrac{3}{4}}}{\sqrt{T}} + \frac{1}{\sqrt[4]{N}})$ for any $i$.
For strongly monotone $\vecF$, choosing $\tau = \sfrac{1}{\sqrt[3]{N}}$ and $\varepsilon = \sfrac{1}{\sqrt{N}}$ yields $\Exop\left[\setE^i_{\text{exp}}(\{\vecpi^j\}_{j=1}^N) \right] \leq \widetilde{\mathcal{O}}(\frac{N^{\sfrac{3}{4}} \lambda^{-\sfrac{1}{2}}}{\sqrt{T}} + \frac{\lambda^{-\sfrac{1}{2}}}{\sqrt[3]{N}})$.
\end{corollary}

The dependence of $N$ of the sample complexity in the bandit case is worse compared to the full feedback setting as expected: intuitively the agents must take turns to estimate the payoffs of each action in bandit feedback.
Furthermore, while our problem framework is different and a direct comparison is difficult in terms of bounds, we point out that classical MMAB results such as \citep{lugosi2022multiplayer} have a linear dependence on $N$, while in our case the dependence on $N$ scales with $N^{\sfrac{3}{4}}$.
We emphasize that the time-dependence is sublinear in terms of $N$, up to the non-vanishing finite population bias.
As in the full feedback case, the non-vanishing finite population bias in the bandit feedback case scales with $\mathcal{O}(\sfrac{1}{\sqrt[4]{N}})$ or $\mathcal{O}(\sfrac{1}{\sqrt[3]{N}})$, rather than $\mathcal{O}(\sfrac{1}{\sqrt{N}})$ which would match Theorem~\ref{theorem:mfg_ne}.
Note that the dependence of the bias on $N$ varies in various mean-field game results \citep{saldi2019approximate}, but asymptotically is known to converge to zero as $N\rightarrow\infty$, as our explicit finite-agent bound also demonstrates.

Finally, we note that as expected the algorithm for bandit feedback has a worse dependency on the number of actions.
This is as expected due to the fact that (i) the importance sampling estimator increases variance on payoff estimators by a factor of $K$, and (2) in other words, a factor of $\mathcal{O}(K)$ is introduced in order to explore all actions.

\section{Experimental Results}\label{sec:experiments_main}

We validate the theoretical results of our work on numerical and two real-world experiments.
The details of the setup are presented in greater detail in the appendix (Section~\ref{sec:experiments_detailed}) and the code is provided in supplementary materials.
Our experiments assume bandit feedback, although in the appendix we also include results under the full feedback model.
We provide an overview of our setup first.

\textbf{Numerical problems.}
Firstly, we formulate three numerical problems, which are based on examples suggested in Sections~\ref{sec:assumptions_examples}.
We randomly generate monotone payoff operators that are linear (\textsc{Linear}) and a payoff function that admits a KL divergence potential (\textsc{kl}).
We also analyze a particular ``beach bar process'' (\textsc{bb}), a stateless version of the example presented in \citep{perrin2020fictitious}.
For numerical examples, we use $K=5$, and vary between various population sizes in $N=\{20,50,100, 200, 500, 1000\}$ to quantify the effect of finite $N$ and compare with theoretical bias bounds.
We set the parameters $\varepsilon, \tau$ using the theoretical values from Corollary~\ref{corollary:bandit}.

\textbf{Traffic flow simulation.}
Using the UTD19 dataset \citep{loder2019understanding}, we evaluate our algorithms on traffic congestion data on three different routes for accessing the city center of Zurich (\textsc{utd}).
The UTD19 dataset features many closed loop sensors across various urban roads in Zurich, providing granular measurements of road occupancy and traffic flow.
We use these real-world stationary detector measurements to approximate traveling times as a function of route occupancy with a kernelized regression model on three routes.
We then evaluate our algorithms on estimated traveling times.
The real-world data from the UTD19 dataset suggests the travel duration (and therefore the payoffs) are in fact non-increasing in the occupancy of each segment, and hence can be reasonably modeled by a monotone payoff.
Further empirical evaluation of monotonicity is discussed in Appendix~\ref{appendix:exp:traffic}.

\textbf{Access to the Tor network.}
We also run experiments on accessing the Tor network, which is a large decentralized network for secure communications and an active area of research in computer security \citep{jansen2014sniper,mccoy2008shining}.
The Tor network consists of many decentralized servers and access to the network is achieved by communicating with one of many entry guard servers (a list advertised publicly \citep{torentryguards}, some hosted by universities).
As the entry guards serve a wide public of users and each user is free to choose an entry point, the network is an ideal setting to test our algorithms.
We simulate 100 independent agents accessing Tor by choosing among $K=5$ entry servers every minute and use the real-world ping delays for bandit feedback.
Our empirical evaluations show that monotonicity is also a reasonable assumption for the Tor network experiment due to the fact that response times tend to be non-decreasing in occupation: we provide an empirical argument in Appendix~\ref{appendix:exp:tor}.

Overall, our experiments in both models and real-world use cases support our theory.
Firstly, our experiments verify learning in the sense of decreasing maximum exploitability $\max_{i\in\setN} \setE^i_{\text{exp}}(\{\vecpi^j\}_{j=1}^N)$ and mean distance to MF-NE given by $\frac{1}{N}\sum_i \|\vecpi^* - \vecpi^i_t\|_2$ during IL (Figure\ref{figure:main}-b,c).
As expected from Corollary~\ref{corollary:bandit}, the maximum exploitability oscillates but remains bounded (due to the effect of finite $N$).
Furthermore, the agents converge to policies closer to the MF-NE as $N$ increases (Figure\ref{figure:main}-c), empirically demonstrating that the MF-NE is an accurate description of the limiting behaviour of SMFG as $N\rightarrow\infty$.
The scale of the excess exploitability in our experiments decreases rapidly in $N$, allowing our results to be significant even for hundreds or thousands of agents.

\begin{figure}[h]
\begin{tabular}{ccc}
  \includegraphics[width=0.3\linewidth]{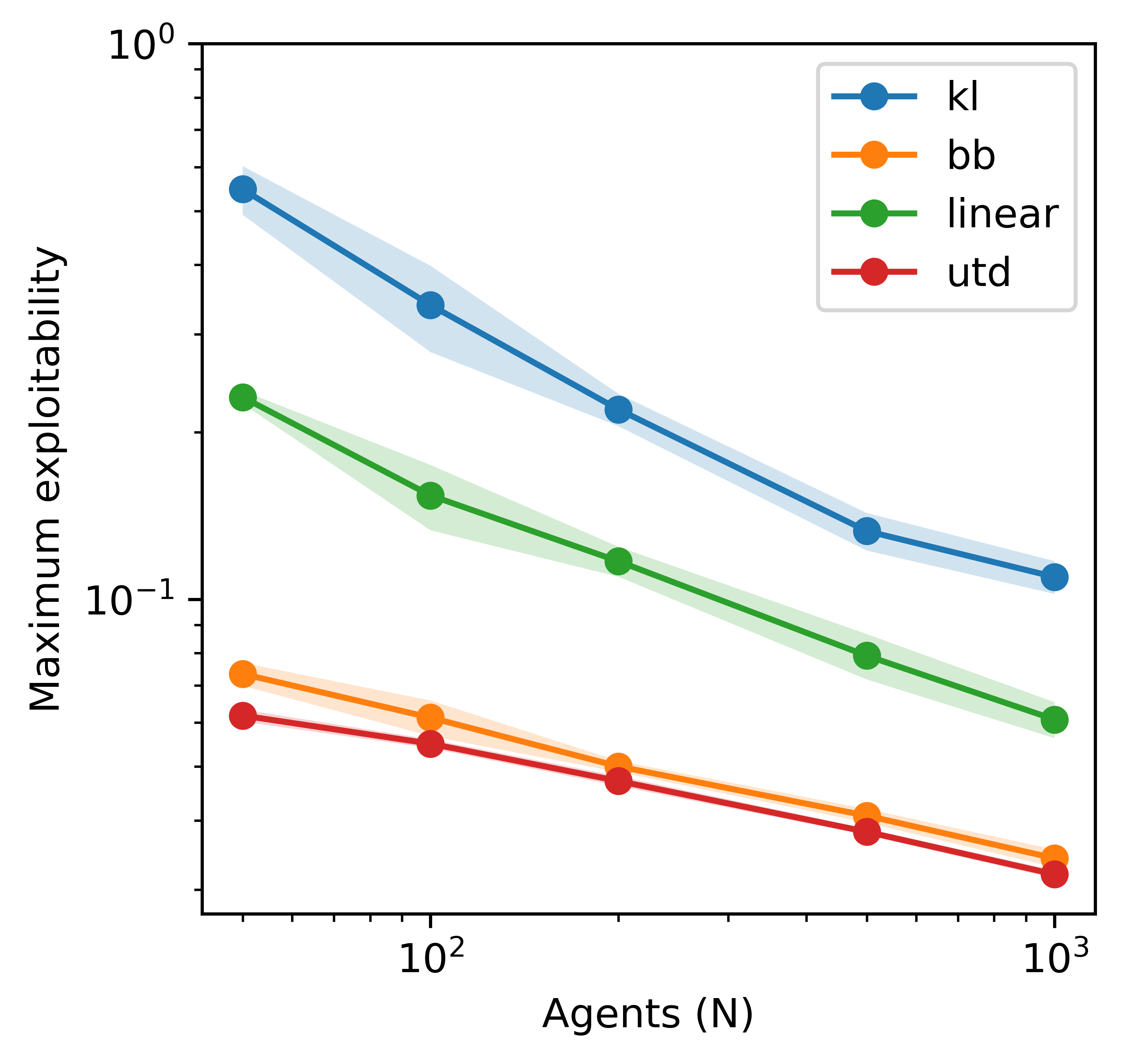} & 
  \includegraphics[width=0.3\linewidth]{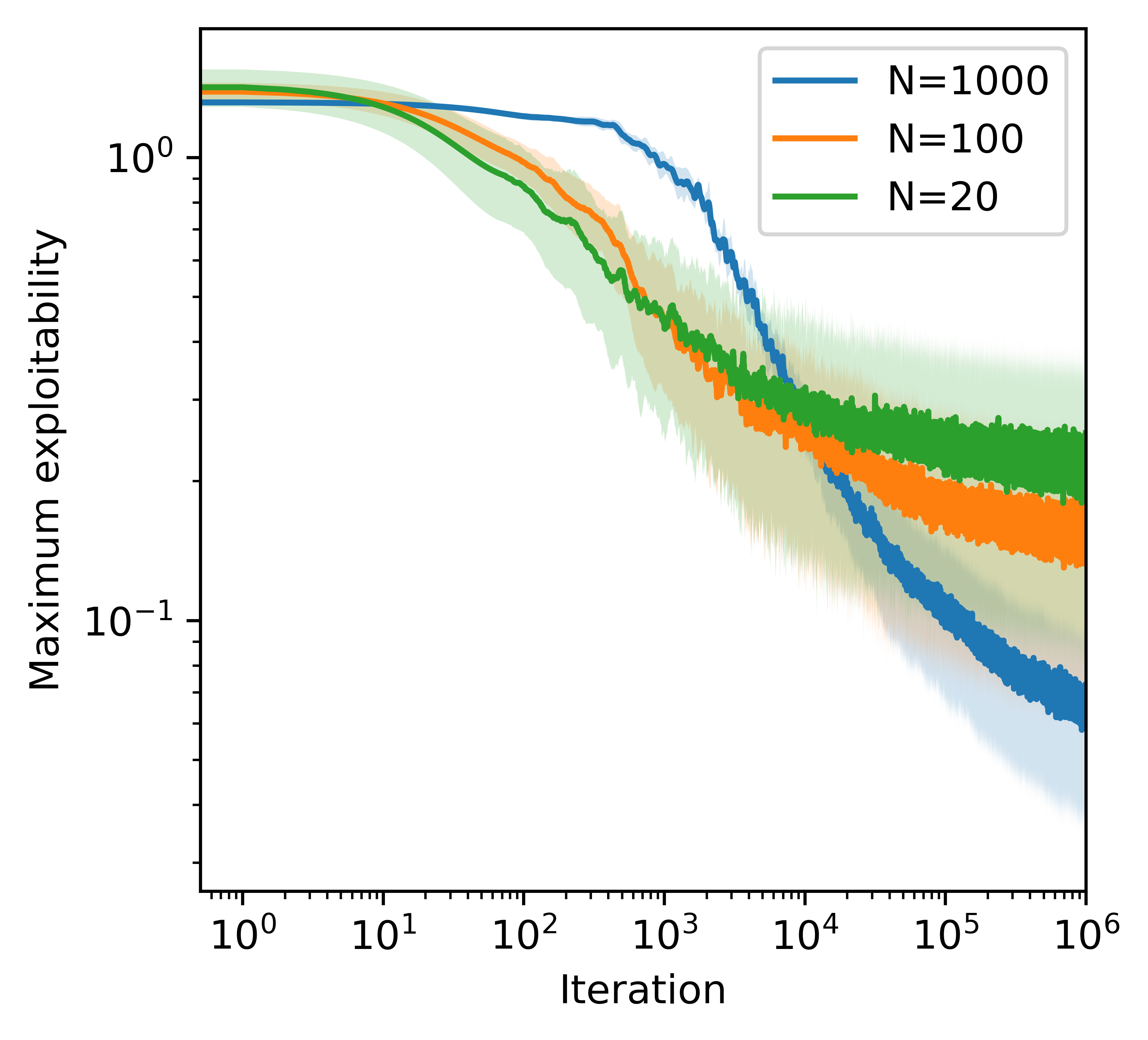} &   
  \includegraphics[width=0.3\linewidth]{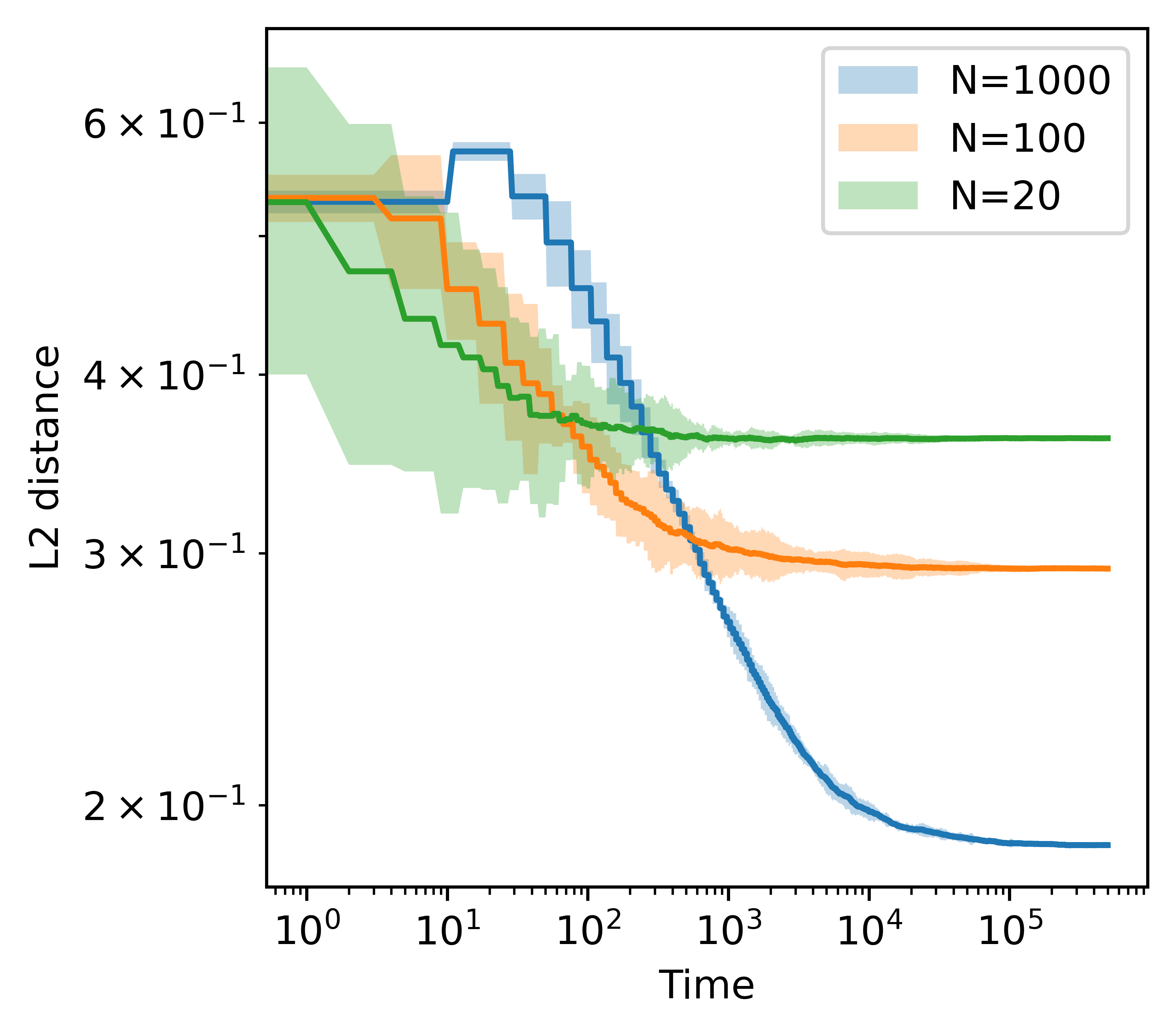} \\
(a) & (b) & (c)
\end{tabular}
\caption{
Results for numerical problems \textsc{kl}, \textsc{bb}, \textsc{linear}, \textsc{utd}.
(a) Maximum exploitability of $N$ agents at convergence as a function of $N$ for different problems,
(b) The max. exploitability among $N$ agents during training with linear payoff (\textsc{linear}), for different $N$,
(c) The mean $\ell_2$ distance of agent policies during training to the MF-NE in the Zurich traffic flow simulation problem (\textsc{utd}).
}
\label{figure:main}
\end{figure}

\begin{figure}[h]
\begin{tabular}{ccc}
 \includegraphics[width=0.3\linewidth]{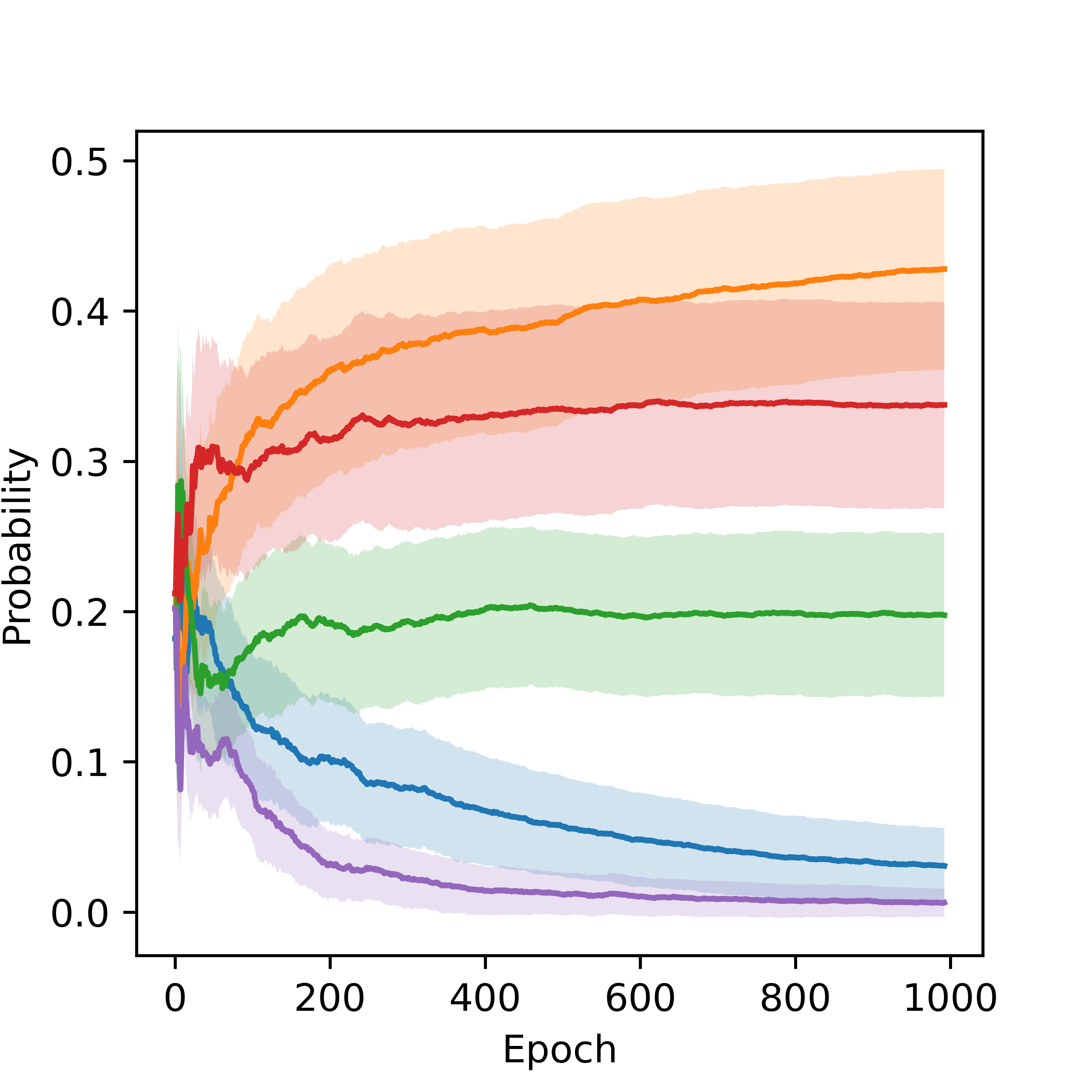} &   \includegraphics[width=0.3\linewidth]{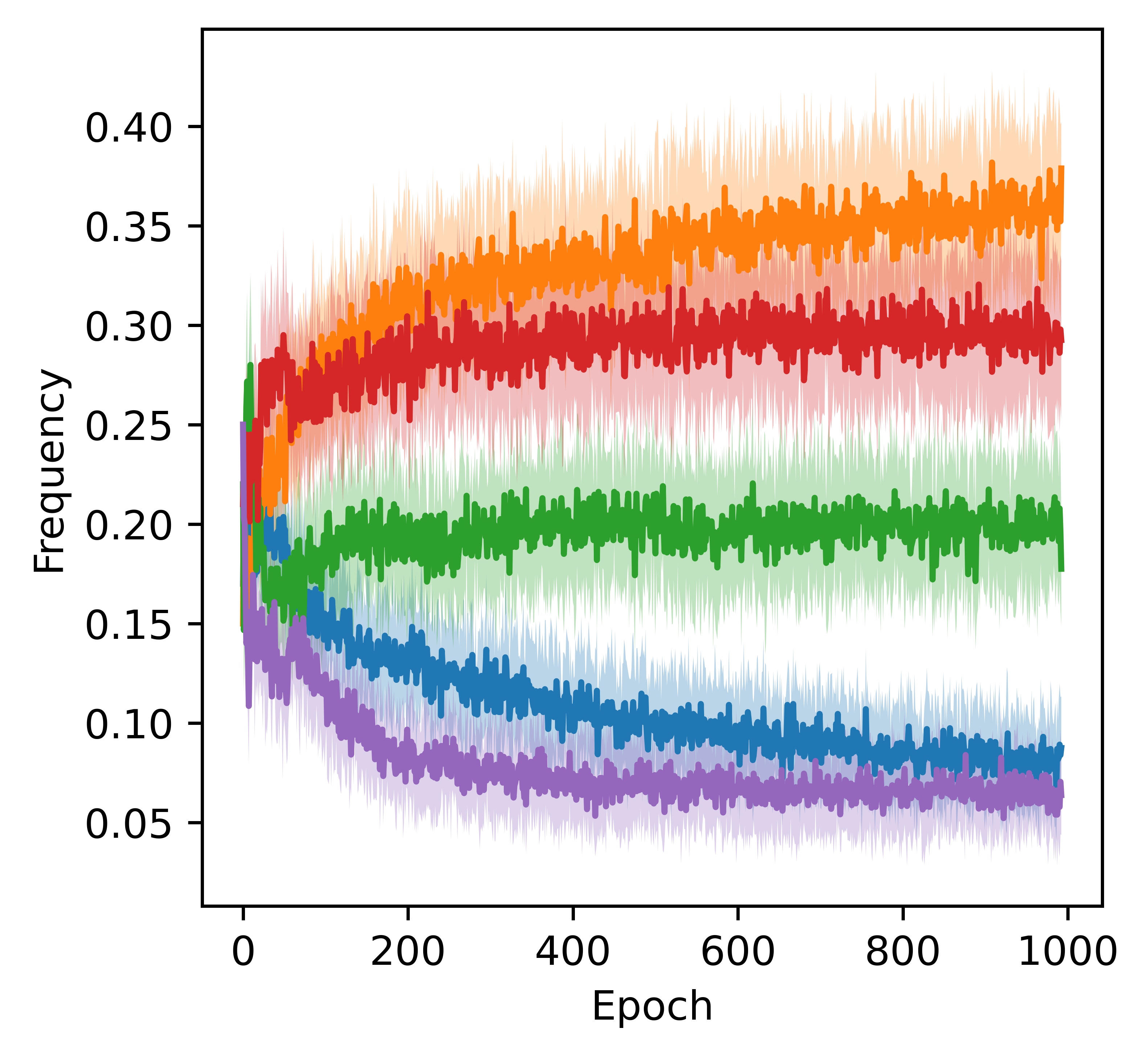} & \includegraphics[width=0.3\linewidth]{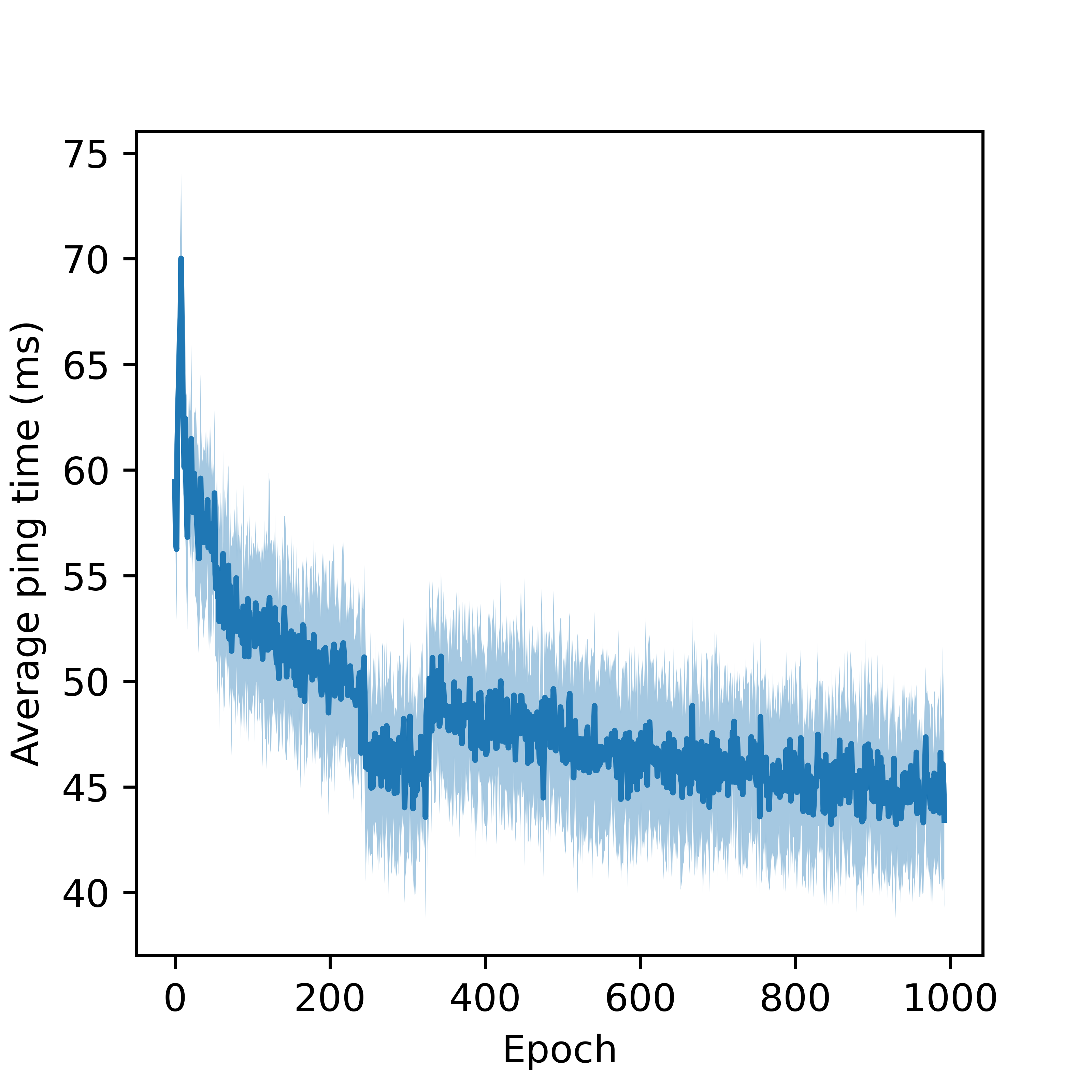}\\
(a) & (b) & (c)
\end{tabular}
\caption{
Results for the Tor network experiment.
(a) Average policies (probability distribution) over 5 servers of the 100 agents in the Tor network access experiment,
(b) Empirical distribution of agents over Tor entry servers during training on 5 servers (different colors indicate different entry servers),
(c) Average waiting times for Tor network access during training.
}
\label{figure:tor_exp}
\end{figure}

Our experimental results in traffic congestion and server access also support our theoretical claims.
Despite having no \emph{a priori} assumption of monotonicity, our methods efficiently converge.
We note that in the case of the network access experiments, we refrain from simulating thousands of agents to minimize network impact.
This implies that the delays have a high dependence on external factors as the behavior of other users will be dominant.
Nevertheless, our experiment indicates that our algorithm can produce competitive results in the wild (Figure~\ref{figure:tor_exp}).

Finally, we note that while there are no algorithms that have theoretical guarantees in the $N$ agent IL setting, we compare our results with a heuristic extension of online mirror descent \citep{perolat2022scaling} in the numerical problems, see Appendix~\ref{appendix:comparison_omd}.
Our results suggest a naive extension of OMD to our setting will lead to divergence or cycling.

\section{Discussion and Conclusion}

In this work, we proposed a finite-agent, independent learning algorithm for learning SMFGs under bandit feedback, and established finite sample guarantees. Our results hinge on the variational inequality perspective of the mean-field game limit,  demonstrating that efficient independent learning is achievable. 

As the first study on decentralized learning on SMFGs, our results are based on the standard (strong) monotonicity assumption on VIs.
However, research on VIs has identified several more general yet still tractable conditions that might be more applicable to real-world SMFGs, such as  the weak Minty condition \citep{diakonikolas2021efficient}, the variationally stable condition \citep{mertikopoulos2019learning}, pseudo-monotone operators \citep{karamardian1976complementarity}, and generalized monotonicity \citep{kotsalis2022simple}.
More abstractly, an intriguing open question is whether the existence of efficient algorithms for solving the (possibly non-monotone) MF-VI always implies the existence of decentralized algorithms for learning SMFG.
At present, several challenges remain when generalizing our results beyond monotonicity: firstly, our current method requires the expected policy deviation to converge to zero, restricting the algorithmic tools that can be deployed.
Moreover, our proof strategy relies on the convergence of iterates $\vecpi_t^i$ to the unique regularized solution $\vecpi^*$ due to the contractivity of the regularized ascent operator $\Gamma^{\eta,\tau}$. Yet for more general VIs, obtaining such a contractive operator in the $\ell_2$-norm might be difficult.

From a game-theoretic and learning perspective, another relevant question is whether our VI approach to decentralized learning can be extended to Markovian MFGs. 
The current theory does not immediately suggest our analysis could be extended to Markovian games.
Specifically, it is not clear if the Markovian Nash equilibrium would still be characterized compactly as a VI,  or when the corresponding operator would be monotone.
The question of IL in stationary MFGs has been tackled by \cite{yardim2023policy} but with poor sample complexity; future work could address this problem using more general VIs.
Finally, while our results focus on expected exploitability, they could potentially be generalized to guarantee convergence to low exploitability policies with high probability.

\section*{Acknowledgements}

We thank the editor and reviewers for their detailed comments and insightful feedback. This project was supported by Swiss National Science Foundation (SNSF) under the framework of NCCR Automation and SNSF Starting Grant.
S. Cayci's research was funded under the Excellence Strategy of the Federal Government and the L{\"a}nder.

\bibliography{sn-bibliography}

\appendix

\section{Table for Notation}
\begin{longtable}{ p{.23\textwidth}  p{.77\textwidth} } 
\underline{General notation:} & \\
$\Delta_\setX$ & probability simplex on discrete set $\setX$ \\ 
$[M]$ & $:= \{1, \ldots, M\}$, for any $M\in\mathbb{N}_{>0}$ \\
$\Delta_{\setX, N}$ & $:= \{ \vecu = \{u_i\}_i \in\Delta_\setX | N u_i \in \mathbb{N}_{\geq 0} \}$, for discrete set $\setX$, $N\in\mathbb{N}_{>0}$ \\
$\mathbb{M}^{D_1,D_2}$ & $D_1 \times D_2$ matrices \\
$\mathbb{S}_{++}^D$ & positive definite $D \times D$ matrices \\
$\Pi_K$ & projection onto convex, compact set $K\subset \mathbb{R}^H$\\
$\vece_a$ & $\in \mathbb{R}^\setA$, standard unit vector with coordinate $a$ set to 1 \\
 &  \\ 
\underline{For SMFGs:} & \\
$\vecF$ & payoff function, $\vecF: \Delta_\setA \rightarrow [0,1]^K$  \\
$\lambda$ & strong monotonicity modulus of $\vecF$ \\  
$L$ & Lipschitz modulus of $\vecF$ \\
$\tau$ & Tikhonov regularization parameter \\ 
$\vecpi^*$ & unique solution of $\tau$ regularized \eqref{eq:mfg_rvi_statement} \\
$K$ & number of actions \\ 
$N$ & number of players \\  
$\setN$ & $:= \{1, \ldots, N \}$, the set of players \\ 
$\setA$ & the set of actions, $|\setA| = K$. \\ 
$\sigma^2$ & upper bound of the standard deviation of received payoff \\
$\widehat{\vecmu}$ & $:= \frac{1}{N} \sum_{i=1}^N \vece_{a^i}$, when actions $\{a^i\}_{i=1}^N$ clear in context \\
$V^i(\vecpi^1, \ldots, \vecpi^N)$ & expected payoff of player $i$ under strategy profile $(\vecpi^1, \ldots, \vecpi^N)$ \\
$\setE^i_{\text{exp}}(\{\vecpi^j\}_{j=1}^N)$ & $:= \max_{\vecpi'\in \Delta_\setA} V^i(\vecpi', \vecpi^{-i}) - V^i(\vecpi^1, \ldots, \vecpi^N)$, exploitability \\
$\widehat{\vecmu}(\{a^i\}_{i=1}^N)$ & $:= \frac{1}{N} \sum_{i=1}^N \vece_{a^i}$, the action distribution induced by particular $\{a^i\}_{i=1}^N \in \setA^N$ \\
 &  \\ 
\underline{For full feedback:} & \\
$t$ & round of play \\
$a^i_t$ & $\in\setA$ action take by player $i$ at round $t$ \\
$\widehat{\vecmu}_t $ & $:= \frac{1}{N} \sum_{i=1}^N \vece_{a_t^i}$, empirical distribution over actions $\setA$ on round $t$ \\ 
$\vecr^i_t$ & $:= \vecF(\widehat{\vecmu}_t) + \vecn_t^i$, payoff vector observed by player $i$ \\ 
$\eta_t$ & learning rate \\
$ \bar{\vecmu}_t $ & $ := \frac{1}{N} \sum_{i=1}^N \vecpi_t^i,$  mean policy at round $t$\\ 
$e_t^i$ & $:= \|\vecpi^i_t - \bar{\vecmu}_t \|_2^2 $, deviation of policy of player $i$ \\
$u_t^i $ & $:= \Exop\left[\| \vecpi_t^i - \vecpi^* \|_2^2\right]$ \\
 &  \\ 
\underline{For bandit feedback:} & \\
$h$ & epoch of play \\
$t$ & round of play in epoch \\
$a^i_{h,t}$ & $\in\setA$ action take by player $i$ at round $t$ of epoch $h$ \\
$T_h$ & number of rounds in epoch $h$ \\
$\varepsilon$ & exploration probability \\
$\widehat{\vecmu}_{h,t} $ & $:= \frac{1}{N} \sum_{i=1}^N \vece_{a_{h, t}^i}$, empirical distribution over actions $\setA$ on round $t$ of epoch $h$ \\ 
$\vecr^i_{h,t}$ & $:= \vecF(\widehat{\vecmu}_{h,t}) + \vecn_{h,t}^i$, (unobserved) payoff vector by player $i$ \\ 
$r^i_{h,t}$ & $:= \vecr^i_{h,t} (a^i_{h,t})$, payoff observed by player $i$ at round $t$ of epoch $h$ \\
$X_{h,t}^i$ & $\sim \operatorname{Ber}(\varepsilon)$ indicator variable, $1$ if $i$ explores at round $t$ of epoch $h$ \\
$\widehat{\vecr}^i_h $ & $:= K r^i_{h,t}\vece_{a^i_{h,t}}$ if $X_{h,t}^i=1$, importance sampling estimate of player $i$ \\
$\eta_h$ & learning rate \\
$ \bar{\vecmu}_h $ & $ := \frac{1}{N} \sum_{i=1}^N \vecpi_h^i,$  mean policy at epoch $h$\\ 
$e_h^i$ & $:= \|\vecpi^i_h - \bar{\vecmu}_h \|_2^2 $, deviation of policy of player $i$ \\
$u_h^i $ & $:= \Exop\left[\| \vecpi_h^i - \vecpi^* \|_2^2\right]$ \\
\end{longtable}

\section{A Detailed Comparison to the Setting in \cite{gummadi2013mean}}\label{sec:detailed_comparison}

Since specific keywords seem to correspond to the works on mean-field approximations with bandits, we provide a greater discussion of our setting and the results by \citet{gummadi2013mean}.
In general, our settings and models are very different, hence almost none of the results between our work and Gummadi et al. are transferable to the other.
Our problem formulation, analysis, and results are fundamentally different from their setting due to the following points.

\textbf{Stationary equilibrium vs Nash equilibrium.}
The most critical difference between the two works is the solution concepts.
Our setting is competitive, as a natural extension, the solution concept is that of a Nash equilibrium where each agent has no incentive to change their policy.
On the other hand, the setting of Gummadi et al. need not be competitive or collaborative and this distinction is not significant for their framework, their goal is to characterize convergence of the population to a stationary distribution.
Their main results show that if a particular policy map $\sigma: \mathbb{Z}_{\geq 0}^{2 n} \rightarrow \Delta_\setA$ is prescribed on agents, the population distribution will converge to a steady state.
The equilibrium concept of \cite{gummadi2013mean} is not \emph{Nash}, rather stationarity.

\textbf{Optimality considerations.}
As a consequence of analyzing stationarity, the results by \citep{gummadi2013mean} do not analyze or aim to characterize optimality.
In their analysis, a fixed map $\sigma: \mathbb{Z}_{\geq 0}^{2 n} \rightarrow \Delta_\setA$ is assumed to be the policy/strategy of a continuum of (i.e., infinitely many) agents, which maps observed loss/win counts (from Bernoulli distributed arm rewards) to arm probabilities.
The stationary distribution in general obtained from $\sigma$ in \cite{gummadi2013mean} does not have optimality properties, for instance, a fixed agent will can have arbitrary large exploitability.
The main goal of \cite{gummadi2013mean} is to prove the convergence of the population distribution to a steady state behaviour.

\textbf{Algorithms.}
As a consequence of the previous points, Gummadi et al. abstract away any algorithmic considerations to the fixed map $\sigma$ and the particular algorithms employed by agents do not directly have significance in terms of their theoretical conclusions.
Since we analyze optimality in our setting, we require a specific algorithm to be employed (TRPA and Algorithm~\ref{alg:bandit}).

\textbf{Independent learning.}
In our setting, the notion of learning and independent learning become significant since we are aiming to obtain an approximate NE.
Hence, our theoretical results bound the expected exploitability (Theorems~\ref{theorem:expert_short}, \ref{theorem:bandit_short}) in terms of number of samples.
In the work of \cite{gummadi2013mean}, the main aim is convergence to a steady state rather than learning.

\textbf{Population regeneration.}
Finally, to be able to obtain a contractive mapping yielding a population stationary distribution/steady state, \cite{gummadi2013mean} assume that the population regenerates at a constant rate $\beta$, implying agents are constantly being replaced by oblivious agents that have not observed the game dynamics.
This smooths the dynamics by introducing a forgetting mechanism to game participants.
Our results on the other hand are closer to the traditional bandits/independent learning setting.
For instance, this would correspond to non-vanishing exploitability scaling with $\mathcal{O}(\beta)$ in our system as agents constantly ``forget'' what they learned.

\textbf{Other model differences.}
In our setting, we assume general noisy rewards while in \cite{gummadi2013mean}, the rewards are Bernoulli random variables with success probability dependent on the population.

\section{Formalizing Learning Algorithms}\label{section:alg_formalization}

In this section, we formalize the concept of an independent learning algorithm in the full feedback and bandit feedback setting.
In general, we formalize the notion of an algorithm as a map $A_t^i: \mathcal{H}_{t}^i \rightarrow \Delta_\setA$ that maps the set of past observations of agent $i$ at time $t$ to action selection probabilities.
The definition of the set $\mathcal{H}_{t}^i$ varies between the feedback models.

\begin{definition}[Learning algorithm with full feedback]
\label{definition:alg_expert}
An independent learning algorithm with full information $\mathbf{A} = \{ A_t^i\}_{i,t}$ is a sequence of mappings for each player with
\begin{align*}
    &A_t^i : \Delta_\setA^{t-1} \times \setA^{t-1} \times [0,1]^{(t-1) \times K} \rightarrow \Delta_\setA, \text{ for all $t > 0$}, \\
    &A_0^i \in \Delta_\setA,
\end{align*}
that maps past $t-1$ observations from previous rounds to a mixed strategy on actions $\setA$ at time $t > 0$ for each agent $i$.
\end{definition}

Naturally, we are interested in algorithms that yield a good approximation of the NE in expectation.
More explicitly, we will be interested in designing algorithms that converge to policy profiles with low expected explotability.

\begin{definition}[Rational learning algorithm with full feedback]
Let $\mathbf{A}$ be an algorithm with full feedback as defined in Definition~\ref{definition:alg_expert}.
We call $\mathbf{A}$ $\delta$-rational if it holds that for all $i$, the induced mixed strategies $\vecpi_t^i$ under $\vecpi_0^i = A_0^i, \vecpi_t^i = A_t^i(\vecpi^i_0, \ldots, \vecpi_{t-1}^i, a_0^i, \ldots, a_{t-1}^i, \vecr_{0}^i, \ldots, \vecr_{t-1}^i)$ satisfy
\begin{align*}
    \lim_{t\rightarrow \infty} \Exop[ \setE^i_{\text{exp}}(\{\vecpi_t^j \}_{j=1}^N)] \leq \delta, \text{ for all } i \in \setN.
\end{align*}
\end{definition}

Note that while not specified in the definition above, we will also be interested in the rate of convergence of the exploitability term for a rational algorithm.
Since we are solving the SMFG at the finite-agent regime, we will be interested in $\delta$-rational algorithms that have $\delta \rightarrow 0$ as $N\rightarrow\infty$, that is, the non-vanishing bias should scale inversely with the number of agents.

Finally, we also formalize the concepts of a learning algorithm and $\delta$-rationality in the bandit setting.

\begin{definition}[Algorithm with bandit feedback]
\label{definition:alg_bandit}
An algorithm with bandit feedback $\mathbf{A} = \{ A_t^i\}_{i,t}$ is a sequence of mappings for each player with
\begin{align*}
    &A_t^i : \Delta_\setA^{t-1} \times \setA^{t-1} \times [0,1]^{(t-1) } \rightarrow \Delta_\setA, \text{ for all $t > 0$}, \\
    &A_0^i \in \Delta_\setA,
\end{align*}
that maps past $t-1$ observations from previous rounds at all times $t > 0$ (only including the payoffs of the \emph{played} actions) to a probability distribution on actions $\setA$.
\end{definition}

\begin{definition}[Rational algorithm with bandit feedback]
Let $\mathbf{A}$ be an algorithm with bandit feedback as defined in Definition~\ref{definition:alg_bandit}.
We call $\mathbf{A}$ $\delta$-rational if it holds that for all $i$, the induced (random) mixed strategies $\vecpi_t^i$ under $\vecpi_0^i = A_0^i, \vecpi_t^i = A_t^i(\vecpi^i_0, \ldots, \vecpi_{t-1}^i, a_0^i, \ldots, a_{t-1}^i, r_{0}^i, \ldots, r_{t-1}^i)$ satisfy
\begin{align*}
    \lim_{t\rightarrow \infty} \Exop[ \setE^i_{\text{exp}}(\{\vecpi_t^j \}_{j=1}^N)] \leq \delta, \text{ for all } i \in \setN.
\end{align*}
\end{definition}

\section{Basic Inequalities}\label{app:basic_inequalities}

In our proofs, we will need to repeatedly bound certain recurrences and sums.
In this section, we present useful inequalities to this end.

\begin{lemma}[Harmonic partial sum bound]\label{lemma:harmonic}
For any integers $s,\bar{s}$, constant $c\in\mathbb{R}$ such that $1 \leq \bar{s} < s$, $p \neq -1$, and $a \geq 0$, it holds that
\begin{align*}
   \log (s + a) - \log (\bar{s} + a ) + \frac{1}{s + a} &\leq \sum_{n = \bar{s}}^{s} \frac{1}{n + a} \leq \frac{1}{\bar{s} + a} + \log ( s + a) - \log (\bar{s} + a), \\
   \frac{(s + a)^{p+1}}{p+1} - \frac{(\bar{s} + a)^{p+1}}{(p+1)} + (\bar{s} + a)^p &\leq \sum_{n = \bar{s}}^{s} (n + a)^p \leq \frac{(s + a)^{p+1}}{p+1} - \frac{(\bar{s} + a)^{p+1}}{p+1} + (s + a)^p, \text{ if } p \geq 0 \\
   \frac{(s + a)^{p+1}}{p+1} - \frac{(\bar{s} + a)^{p+1}}{p+1} + (s + a)^p &\leq \sum_{n = \bar{s}}^{s} (n + a)^p \leq \frac{(s + a)^{p+1}}{p+1} - \frac{(\bar{s} + a)^{p+1}}{p+1} + (\bar{s} + a)^p, \text{ if } p \leq 0
\end{align*}
\end{lemma}
\begin{proof}
Let $f_1:\mathbb{R}_{\geq 0} \rightarrow \mathbb{R}_{\geq 0}$ be a non-decreasing positive function and $f_2:\mathbb{R}_{\geq 0} \rightarrow \mathbb{R}_{\geq 0}$ be a non-increasing positive function.
Then it holds that
\begin{align*}
    \int_{x = \bar{s}}^s f_1(x) dx + f_1(\bar{s}) \leq \sum_{n=\bar{s}}^s f_1(n) \leq \int_{x = \bar{s}}^s f_1(x) dx  + f_1(s), \\
    \int_{x = \bar{s}}^s f_2(x) dx + f_2(s) \leq \sum_{n=\bar{s}}^s f_2(n) \leq \int_{x = \bar{s}}^s f_2(x) dx + f_2(\bar{s}).
\end{align*}
The result follows from a simple integral bound with $\int \frac{1}{x} dx = \log x$ and $\int x^p dx = \frac{x^{p+1}}{p+1}$.
\end{proof}

We state a certain recurrence inequality that appears several times in our analysis as a lemma, in order to shorten some proofs.

\begin{lemma}[General error recurrence]\label{lemma:general_recurrence}
Let $c_0 \geq 0, c_1 \geq 0, \gamma > 1, a \geq 0$ be arbitrary constants such that $a \geq \gamma$.
Furthermore, let $\{u_t\}_{t=0}^\infty$ be a sequence of non-negative numbers such that for all $t \geq 0$, it holds that
\begin{align*}
    u_{t+1} \leq \frac{c_0}{t+a} + \frac{c_1}{(t+a)^2} + \left( 1 - \frac{\gamma}{t+a}\right) u_t.
\end{align*}
Then, for all values of $t\geq 0$, it holds that:
\begin{align*}
u_{t+1} \leq &\frac{u_0 a ^ \gamma + c_1 (a^{\gamma-2} + 1) (1+a^{-1})^\gamma + c_0 (1+a^{-1})^\gamma a^{\gamma-1}}{\left(t+a\right)^\gamma}  \\
    &\quad + \frac{c_0 + c_1 (1+a^{-1})^\gamma (\gamma - 1)^{-1}}{t+a} + \frac{c_1}{(t+a)^2} + \gamma^{-1}(1+a^{-1})^\gamma c_0 
\end{align*}
    
\end{lemma}
\begin{proof}
We note that inductively, we have
\begin{align*}
u_{t+1} \leq &\frac{c_0}{t+a} + \frac{c_1}{(t+a)^2} + u_0 \prod_{s=0}^{t} \left( 1 - \frac{\gamma}{s+a} \right) 
     + \sum_{s=0}^{t-1} \left( \frac{c_0}{s+a} + \frac{c_1}{(s+a)^2} \right) \prod_{s'=s+1}^{t} \left(1 - \frac{\gamma}{s'+a} \right).
\end{align*}
Using the inequality $1 + x \leq e^{x}$, we obtain
\begin{align*}
u_{t+1} \leq &\frac{c_0}{t+a} + \frac{c_1}{(t+a)^2} + u_0 \prod_{s=0}^{t} \exp\left\{- \frac{\gamma}{s+a}\right\} + \sum_{s=0}^{t-1} \left( \frac{c_0}{s+a} + \frac{c_1}{(s+a)^2} \right) \prod_{s'=s+1}^{t} \exp\left\{ - \frac{\gamma}{s'+a}\right\} \\
\leq &\frac{c_0}{t+a} + \frac{c_1}{(t+a)^2} + u_0 \exp\left\{- \sum_{s=0}^{t}\frac{\gamma}{s+a}\right\} + \sum_{s=0}^{t-1} \left( \frac{c_0}{s+a} + \frac{c_1}{(s+a)^2} \right)  \exp\left\{ -\sum_{s'=s+1}^{t} \frac{\gamma}{s'+a}\right\}. 
\end{align*}
Here, using Lemma~\ref{lemma:harmonic}, since $a - 1 > 0$, it holds that
\begin{align*}
  \sum_{s=0}^{t}\frac{\gamma}{s+a} &\geq \sum_{s = 1}^{t+1} \frac{\gamma}{s+(a - 1)} \geq \gamma \log(t+a) - \gamma \log a = \log \left(\frac{(t+1)^\gamma}{a^\gamma}\right) \\
  \sum_{s'=s+1}^{t} \frac{\gamma}{s'+a} &\geq \gamma \log(t + a) - \gamma \log(s + a + 1) \geq \log \left\{ \frac{(t+a)^\gamma}{(s+a+1)^\gamma} \right\},
\end{align*}
therefore $u_{t+1}$ can be further upper bounded by
\begin{align*}
u_{t+1} \leq &\frac{c_0}{t+a} + \frac{c_1}{(t+a)^2} + u_0 \frac{a^\gamma}{ \left(t+a\right)^\gamma} + \sum_{s=0}^{t-1} \left( \frac{c_0}{s+a}  + \frac{c_1}{(s+a)^2} \right) \left(\frac{s + a + 1}{t+a}\right)^\gamma \\
\leq &\frac{c_0}{t+a} + \frac{c_1}{(t+a)^2} + \frac{u_0 a ^ \gamma}{\left(t+a\right)^\gamma} + (t+a)^{-\gamma} \sum_{s=0}^{t-1} \left( \frac{c_0}{s+a} + \frac{c_1}{(s+a)^2} \right) \left(s+a+1\right)^\gamma \\
\leq &\frac{c_0}{t+a} + \frac{c_1}{(t+a)^2} + \frac{u_0 a ^ \gamma}{\left(t+a\right)^\gamma} + (t+a)^{-\gamma} (1 + a^{-1})^\gamma \sum_{s=0}^{t-1} \left( c_0 \left(s+a\right)^{\gamma-1} + c_1\left(s+a\right)^{\gamma-2} \right) .
\end{align*}
The last term can be bound with the corresponding integral (see Lemma~\ref{lemma:harmonic}), yielding (since $\gamma-1 > 0$):
\begin{align*}
    \sum_{s=0}^{t-1}(s+a)^{\gamma - 1} &\leq \frac{(t+a)^\gamma}{\gamma} + a ^ {\gamma - 1}.
\end{align*}
For the term $\sum_{s=0}^{t-1}(s+1)^{\gamma - 2}$, we analyzing the two cases $1 < \gamma \leq 2$ and $\gamma > 2$ using Lemma~\ref{lemma:harmonic} we obtain
\begin{align*}
    \sum_{s=0}^{t-1}(s+a)^{\gamma - 2} &\leq \frac{(t+a)^{\gamma-1}}{\gamma - 1} + a^{\gamma - 2} + 1.
\end{align*}
The two inequalities combined yield the stated bound,
\begin{align*}
    u_{t+1} \leq & \frac{c_0}{t+a} + \frac{c_1}{(t+a)^2} + \frac{u_0 a ^ \gamma}{\left(t+a\right)^\gamma} + (t+a)^{-\gamma} (1 + a^{-1})^\gamma \sum_{s=0}^{t-1} \left( c_0(s+a)^{\gamma-1} + c_1(s+a)^{\gamma - 2} \right) \\
    \leq & \frac{c_0}{t+a} + \frac{c_1}{(t+a)^2} + \frac{u_0 a ^ \gamma}{\left(t+a\right)^\gamma} + \gamma^{-1}(1+a^{-1})^\gamma c_0 + \frac{c_0 (1+a^{-1})^\gamma a^{\gamma-1}}{(t+a)^\gamma} \\ 
    &\quad +\frac{c_1 (1+a^{-1})^\gamma}{(t+a)(\gamma - 1)} + \frac{c_1 (a^{\gamma-2} + 1) (1+a^{-1})^\gamma}{(t+a)^\gamma} \\
    \leq & \frac{u_0 a ^ \gamma + c_1 (a^{\gamma-2} + 1) (1+a^{-1})^\gamma + c_0 (1+a^{-1})^\gamma a^{\gamma-1}}{\left(t+a\right)^\gamma}  \\
    &\quad + \frac{c_0 + c_1 (1+a^{-1})^\gamma (\gamma - 1)^{-1}}{t+a} + \frac{c_1}{(t+a)^2} + \gamma^{-1}(1+a^{-1})^\gamma c_0 .
\end{align*}
\end{proof}

\section{Proofs of Technical Results}

\subsection{Proof of Lemma~\ref{lemma:technical_bound_1N}}\label{app:technical_lemma}
\begin{proof}
We introduce an auxiliary random variable $\bar{a}^i$ which is independent from other players' actions $\{ a^j\}_{j=1}^N$ and has distribution $\vecpi^i$, that is, we introduce the random variable $\bar{a}^i$ as an identically distributed independent copy of $a^i$.
Then, it holds by simple computation that
\begin{align*}
V^i(\vecpi^1, \ldots, \vecpi^N) = & \Exop \Big[ \vecF\Big( \frac{1}{N} \sum_{j=1}^N \vece_{a^j}, a^i\Big) \Big]\\
= &\Exop \Big[\vecF\Big(\frac{1}{N}\Big(\sum_{j=1, j\neq i}^N \vece_{a^j} + \vece_{\bar{a}^i}\Big), a^i \Big)\Big] \\ 
    &+ \Exop \Big[ \vecF\Big(\frac{1}{N}\sum_{j=1}^N \vece_{a^j}, a^i\Big) -\vecF\Big( \frac{1}{N}\Big(\sum_{j=1, j\neq i}^N \vece_{a^j} + \vece_{\bar{a}^i}\Big), a^i\Big)\Big].
\end{align*}
For the first term above, we observe that
\begin{align*}
   \Exop \Big[\vecF\Big( \frac{1}{N}\Big(\sum_{j=1, j\neq i}^N \vece_{a^j} + \vece_{\bar{a}^i}\Big), a^i\Big)\Big] = &\Exop \Big[ \Exop \Big[\vecF\Big( \frac{1}{N}\Big(\sum_{j=1, j\neq i}^N \vece_{a^j} + \vece_{\bar{a}^i}\Big), a^i\Big) \Big| a^i\Big]\Big] \\
   = &\Exop \Big[ \Exop \Big[\vece_{a^i}^\top \vecF\Big(\frac{1}{N}\Big(\sum_{j=1, j\neq i}^N \vece_{a^j} + \vece_{\bar{a}^i}\Big)\Big) \Big| a^i\Big]\Big] \\
   = &\Exop [ \vece_{a^i}^\top] \Exop \Big[ \vecF\Big(\frac{1}{N}\Big(\sum_{j=1, j\neq i}^N \vece_{a^j} + \vece_{\bar{a}^i}\Big)\Big)\Big] \\
   = & \vecpi^{i, \top} \Exop [\vecF(\widehat{\vecmu})],
\end{align*}
since $\{a^j \}_{j=1}^N$ and $(\bar{a}^i,\veca^{-i})$ are identically distributed by the independence of both $\bar{a}^i$ and $a^i$ from $\veca^{-i}$, and since
$(\bar{a}^i,\veca^{-i})$ is independent of $a^i$.

The second term above can be bounded using
\begin{align*}
    \Big|\vecF\Big(&\frac{1}{N}\sum_{j=1}^N \vece_{a^j}, a^i\Big) -\vecF\Big( \frac{1}{N}\Big(\sum_{j=1, j\neq i}^N \vece_{a^j} + \vece_{\bar{a}^i}\Big), a^i\Big)\Big| \\
    = & \Big|\vece_{a^i}^\top\vecF\Big(\frac{1}{N}\sum_{j=1}^N \vece_{a^j}\Big) -\vece_{a^i}^\top \vecF\Big( \frac{1}{N}\Big(\sum_{j=1, j\neq i}^N \vece_{a^j} + \vece_{\bar{a}^i}\Big)\Big)\Big| \\
    \leq & \left\| \vece_{a^i} \right\|_2 \Big\| \vecF\Big(\frac{1}{N}\sum_{j=1}^N \vece_{a^j}\Big) -\vecF\Big( \frac{1}{N}\Big(\sum_{j=1, j\neq i}^N \vece_{a^j} + \vece_{\bar{a}^i}\Big)\Big) \Big\|_2 \\
    \leq & \frac{L}{N} \left\| \vece_{a^i} - \vece_{\bar{a}^i}\right\|_2 \leq \frac{L\sqrt{2}}{N}.
\end{align*}
The last line follows from the fact that $\vecF$ is $L$-Lipschitz.
\end{proof}

\subsection{Proof of Lemma~\ref{lemma:phi_lipschitz}}\label{sec:extended_proof_lema_lipschitz_phi}

We first prove the fact that $V^i$ is Lipschitz.
In the proof, we denote the empirical action distribution induced by actions $\{a^j\}_{j=1}^N\in\setA^N$ by $\widehat{\vecmu}(\{a^j\}_{j=1}^N) \in \Delta_\setA$ to simplify notation.
We analyze two cases of Lipschitz moduli stated in the lemma separately. In the first case, we compute the Lipschitz moduli $L_{i,i}$ for any $i\in\setN$ as follows:
\begin{align*}
    |V^i&(\vecpi, \vecpi^{-i}) - V^i (\vecpi', \vecpi^{-i})| \\
    \leq & \Big| \Exop \left[ \vecF\left(\widehat{\vecmu}(\{a^k\}_{k=1}^N), a^i\right) \middle|
a^j \sim \vecpi^j, \forall j \neq i, a^i \sim \vecpi
\right] \\
    &\quad - \Exop \left[ \vecF\left(\widehat{\vecmu}(\{a^k\}_{k=1}^N),a^i\right) \middle|
a^j \sim \vecpi^j, \forall j \neq i, a^i \sim \vecpi' \right] \Big| \\
\leq & \big| \sum_{\substack{a^j \in \setA \\ j\neq i}}  \sum_{a^i \in \setA}\vecpi(a^i) \vecF(\widehat{\vecmu}(\{a^k\}_{k=1}^N), a^i)\prod_{j\neq i} \vecpi^j(a^j) - \sum_{\substack{a^j \in \setA \\ j\neq i}} \sum_{a^i \in \setA} \vecpi'(a^i) \vecF(\widehat{\vecmu}(\{a^k\}_{k=1}^N), a^i) \prod_{j\neq i} \vecpi^j(a^j) 
 \Big| \\
\leq & \sum_{\substack{a^j \in \setA \\ j\neq i}}  \Big|\sum_{a^i \in \setA}\left[ \vecpi(a^i) - \vecpi'(a^i)\right] \vecF(\widehat{\vecmu}(\{a^k\}_{k=1}^N), a^i) \Big| \prod_{j\neq i} \vecpi^j(a^j) \\
\leq & \sum_{\substack{a^j \in \setA \\ j\neq i}}  \left\| \vecpi - \vecpi'\right\|_2 \sqrt{\sum_{a^i \in \setA} \vecF(\widehat{\vecmu}(\{a^k\}_{k=1}^N), a^i)^2} \prod_{j\neq i} \vecpi^j(a^j) 
\leq \left\| \vecpi - \vecpi'\right\|_2 \sqrt{K}.
\end{align*}
where we use Jensen's inequality in the penultimate step and the Cauchy-Schwartz inequality in the final step.

In the second case, for any $k \neq i$, it holds that
\begin{align*}
|V^i&(\vecpi, \vecpi^{-k}) - V^i (\vecpi', \vecpi^{-k})| \\
    \leq & \Exop \left[ \vecF\left(\widehat{\vecmu}(\{a^l\}_{l=1}^N), a^i\right) \middle|
a^j \sim \vecpi^j, \forall j \neq k, a^k \sim \vecpi
\right] - \Exop \left[ \vecF\left(\widehat{\vecmu}(\{a^l\}_{l=1}^N),a^i\right) \middle|
a^j \sim \vecpi^j, \forall j \neq k, a^k \sim \vecpi' \right] \\
    \leq & \Big| \sum_{\substack{a^j \in \setA \\ j\neq k}} \prod_{j\neq k} \vecpi^j(a^j) \sum_{a^k \in \setA}\vecpi(a^k) \vecF(\widehat{\vecmu}(\{a^l\}_{l=1}^N), a^i) - \sum_{\substack{a^j \in \setA \\ j\neq k}} \prod_{j\neq k} \vecpi^j(a^j) \sum_{a^k \in \setA} \vecpi'(a^k) \vecF(\widehat{\vecmu}(\{a^l\}_{l=1}^N), a^i) \Big| \\
\leq & \sum_{\substack{a^j \in \setA \\ j\neq k}} \prod_{j\neq k} \vecpi^j(a^j) \Big|\sum_{a^k \in \setA}\left[ \vecpi(a^k) - \vecpi'(a^k)\right] \vecF(\widehat{\vecmu}(\{a^l\}_{l=1}^N), a^i) \Big|
\end{align*}
In this case, note that for any $a, a' \in \setA, \veca\in \setA^K$, we have $|\vecF(\widehat{\vecmu}(a, \veca^{-k}), a^i) - \vecF(\widehat{\vecmu}(a', \veca^{-k}), a^i)| \leq \|\vecF(\widehat{\vecmu}(a, \veca^{-k})) - \vecF(\widehat{\vecmu}(a', \veca^{-k})) \|_2 \leq \sfrac{L\sqrt{2}}{N}$.
That is, the set $\{\vecF(\widehat{\vecmu}(a, \veca^{-k}), a^i) \, : a\in\setA \} \subset \mathbb{R}$ has diameter $\sfrac{2L}{\sqrt{N}}$, and
there exists a constant $v_{k} \in \mathbb{R}$ such that $|\vecF(\widehat{\vecmu}(a, \veca^{-k}), a^i) - v_{k}| \leq \sfrac{2L\sqrt{2}}{N}$ for all $a$.
Then,
\begin{align*}
|V^i&(\vecpi, \vecpi^{-k}) - V^i (\vecpi', \vecpi^{-k})| \\
    \leq & \sum_{\substack{a^j \in \setA \\ j\neq k}} \prod_{j\neq k} \vecpi^j(a^j) \Big|\sum_{a^k \in \setA}\left[ \vecpi(a^k) - \vecpi'(a^k)\right] \left[ \vecF(\widehat{\vecmu}(\{a^l\}_{l=1}^N), a^i) - v_{k}\right] \Big| \\
    \leq & \sum_{\substack{a^j \in \setA \\ j\neq k}} \prod_{j\neq k} \vecpi^j(a^j) \|\vecpi - \vecpi'\|_2 \sqrt{\sum_{a^k\in \setA} \left[ \vecF(\widehat{\vecmu}(\{a^l\}_{l=1}^N), a^i) - v_{k}\right]^2 } \\
    \leq & \sum_{\substack{a^j \in \setA \\ j\neq k}} \prod_{j\neq k} \vecpi^j(a^j) \|\vecpi - \vecpi'\|_2 \frac{2L\sqrt{2K}}{N} \leq \|\vecpi - \vecpi'\|_2 \frac{2L \sqrt{2K}}{N}.
\end{align*}

We establish the Lipschitz continuity of $\setE^i_{\text{exp}}$ with the above inequalities.
\begin{align*}
    |\setE^i_{\text{exp}} &(\vecpi, \vecpi^{-i}) - \setE^i_{\text{exp}}(\vecpi', \vecpi^{-i}) | \\
    \leq & \left|\max_{\overline{\vecpi}\in \Delta_\setA} V^i(\overline{\vecpi}, \vecpi^{-i}) - V^i(\vecpi, \vecpi^{-i}) - \max_{\overline{\vecpi}\in \Delta_\setA} V^i(\overline{\vecpi}, \vecpi^{-i}) + V^i(\vecpi', \vecpi^{-i})\right| \\
    \leq & \left| V^i(\vecpi', \vecpi^{-i}) - V^i(\vecpi, \vecpi^{-i})\right| 
    \leq \sqrt{K} \|\vecpi - \vecpi'\|_2.
\end{align*}
Similarly, for $k\neq i$,
\begin{align*}
|\setE^i_{\text{exp}} &(\vecpi, \vecpi^{-k}) - \setE^i_{\text{exp}}(\vecpi', \vecpi^{-k}) | \\
    \leq & \left|\max_{\overline{\vecpi}\in \Delta_\setA} [ V^i( \vecpi, \overline{\vecpi},\vecpi^{-k,i}) - V^i(\vecpi, \vecpi^{-k}) ] - \max_{\overline{\vecpi}\in \Delta_\setA} [V^i( \vecpi', \overline{\vecpi},\vecpi^{-k,i}) - V^i(\vecpi', \vecpi^{-k}) ]\right| \\
    \leq & \max_{\overline{\vecpi}\in \Delta_\setA} \left| V^i( \vecpi, \overline{\vecpi},\vecpi^{-k,i}) - V^i( \vecpi',\overline{\vecpi}, \vecpi^{-k,i}) \right| + \left| V^i(\vecpi, \vecpi^{-k}) - V^i(\vecpi', \vecpi^{-k})\right| \\
    \leq & \frac{4L\sqrt{2K}}{N} \| \vecpi - \vecpi'\|_2.
\end{align*}

\subsection{Extended Definitions for Bandit Feedback}\label{sec:bandit_extended_defs}

When analyzing the TRPA-Bandit dynamics, several random variables and events will be reused to assist analysis.
For brevity, we define them here.
We define the following random variables:
\begin{align*}
    \mathbbm{1}_{h,t}^i := &\mathbbm{1}\{\text{player $i$ explores at round $t$ of epoch $h$}\} = X_{h,t}^i \\
    E_{h,t}^i := &\{\mathbbm{1}_{h,t}^i = 1 \} \\
    \mathbbm{1}_{h}^i := &\mathbbm{1}\{\text{player $i$ explores at least once during epoch $h$}\} = \max_{t = 1, \ldots, T_h} \mathbbm{1}_{h,t}^i  \\
    E_{h}^i := &\{\mathbbm{1}_{h}^i = 1 \} = \bigcup_{t=1}^{T_h} E_{h,t}^i\\
    a_{h}^i := &\text{Last explored action in epoch $h$ by agent $i$,} \\
        &\text{and $a_0$ if no exploration occurred}. \\
    s_h^i := &\text{Timestep when exploration last occurred in epoch $h$ by agent $i$, }\\
        &\text{and $0$ if no exploration occurred.} \in \{1, \ldots, T_h \}
\end{align*}

\subsection{Proof of Lemma~\ref{lemma:exploration_bias_trpa_bandit}}\label{sec:proof_lemma_bandit_exploration_bias}

We will reuse the definitions of Section~\ref{sec:bandit_extended_defs}.
By the definition of the events and the probabilistic exploration scheme, we have $\widehat{\vecr}_h^i = K \left( \vecF(\widehat{\vecmu}_{s_h^{i},h}, a_h^{i}) + \vecn_{s_h^{i},h}^{i}(a_h^{i}) \right) \mathbbm{1}_h^i$.
Firstly, by the law of total expectations and the fact that $E_h^i$ are independent of $\mathcal{F}_h$,
\begin{align*}
    \Exop\left[\widehat{\vecr}_h^i \middle| \mathcal{F}_h\right] = &\Exop\left[\widehat{\vecr}_h^i \middle|E_{h}^i, \mathcal{F}_h\right] \Prob(E_{h}^i) + \Exop\left[\widehat{\vecr}_h^i\middle|\overline{E_{h}^{i}}, \mathcal{F}_h\right] \Prob(\overline{E_{h}^{i}}) \\
    = & \Exop\left[\widehat{\vecr}_h^i \middle|E_{h}^i, \mathcal{F}_h\right] - \underbrace{\Exop\left[\widehat{\vecr}_h^i\middle|E_{h}^{i}, \mathcal{F}_h\right] \Prob(\overline{E_{h}^{i}}) + \Exop\left[\widehat{\vecr}_h^i\middle|\overline{E_{h}^{i}}, \mathcal{F}_h\right] \Prob(\overline{E_{h}^{i}})}_{:=\vecb_h^i} \\
    = & \Exop\left[\widehat{\vecr}_h^i \middle|E_{h}^i, \mathcal{F}_h\right] + \vecb_h^i
\end{align*}
for $\vecb_h^i$ quantifying a bias induced due to the probability of no exploration.
We have that
\begin{align*}
    \| \vecb_h^i \|_2 \leq K\sqrt{K} \sqrt{1 + \sigma^2} \exp\left\{ -\varepsilon T_{h}\right\}
\end{align*}
since $\Exop\left[\widehat{\vecr}_h^i\middle|\overline{E_{h}^{i}}, \mathcal{F}_h\right] = 0$ and exploration probabilities are determined by independent random Bernoulli variables hence $\Prob(\overline{E_{h}^{i}}) = \left( 1 - \varepsilon \right)^{T_h} \leq \exp\left\{ -\varepsilon T_{h}\right\}.$
To further characterize the bias, we introduce a coupling argument.
Define independent random variables $\bar{a}_h^j \sim \varepsilon \vecpi_\text{unif} + (1-\varepsilon) \vecpi_h^j$ for all $j\in\setN$ and $\bar{a}^i_{h, exp} \sim \vecpi_\text{unif}$.
By the definition, it holds that (where $\widehat{\vecmu}(\{ \bar{a}_h ^ j \}_j) \in \Delta_\setA$ denotes the empirical distribution induced by actions $\{ \bar{a}_h ^ j \}_j$):
\begin{align*}
    \|\Exop \left[\widehat{\vecr}_h^i \middle|E_{h}^i, \mathcal{F}_h\right] - \Exop\left[ \vecF(\widehat{\vecmu}(\{ \bar{a}_h ^ j \}_j)) \middle| \mathcal{F}_h \right]\|_2 
    \leq & \|\Exop \left[\vecF(\widehat{\vecmu}(\bar{a}_{h,exp}^i, \bar{a}_h ^ {-i})) \middle| \mathcal{F}_h \right] - \Exop\left[ \vecF(\widehat{\vecmu}(\{ \bar{a}_h ^ j \}_j))  \middle| \mathcal{F}_h\right]\|_2 \\
    \leq & \Exop \left[\|\vecF(\widehat{\vecmu}(\bar{a}_{h,exp}^i, \bar{a}_h ^ {-i})) - \vecF(\widehat{\vecmu}(\{ \bar{a}_h ^ j \}_j)) \|_2 \middle| \mathcal{F}_h\right] \leq \frac{2 L }{N}.
\end{align*}
With the additional bound $\Exop \left[\|\vecF(\varepsilon \vecpi_\text{unif} + (1-\varepsilon) \bar{\vecmu}_h) - \vecF(\widehat{\vecmu}(\{ \bar{a}_h ^ j \}_j)) \|_2 \middle| \mathcal{F}_h \right] \leq \frac{2 L }{\sqrt{N}} $, we obtain the lemma.

\subsection{Proof of Lemma~\ref{lemma:bandit_main_recurrence}}\label{sec:proof_lemma_bandit_recurrence}

We will reuse the definitions of Section~\ref{sec:bandit_extended_defs}.
We formulate a recurrence for the main error term of interest, $\| \vecpi_{t+1}^i - \vecpi^*\|_2^2$,
Repeating the steps presented in Lemma~\ref{lemma:full_error_recurrence}, (noting $\alpha_h := 1 - \tau\eta_h$), our proof strategy is to analyze the three terms in the following decomposition:
\begin{align*}
    \| \vecpi_{h+1}^i - \vecpi^*\|_2^2  
    \leq & \underbrace{\eta_h^2\| \widehat{\vecr}_h^i - \vecF(\vecpi_h^i) \|_2^2 + 2\eta_h^2 (\vecF(\vecpi_h^i) - \vecF(\vecpi^*)))^\top (\widehat{\vecr}_h^i - \vecF(\vecpi_h^i))}_{(a)} \\
     &+ \underbrace{2\eta_h\alpha_h (\vecpi_h^i - \vecpi^*)^\top (\widehat{\vecr}_h^i - \vecF(\vecpi_h^i))}_{(b)} + \underbrace{\|\alpha_h (\vecpi_h^i - \vecpi^*) + \eta_h (\vecF(\vecpi_h^i) - \vecF(\vecpi^*))\|_2^2 }_{(c)}
\end{align*}
Once again, we will need to upper bound the three terms above.
For the term $(a)$ we have $\Exop\left[ (a)\right] \leq 4\eta_t^2 K^3(\sigma^2 + 1)$, noting that $\|\widehat{\vecr}_h^i\|_2^2 \leq K^3$ almost surely.
Likewise, it still holds for the term $(c)$ by Lemma~\ref{lemma:contraction_pg} that
\begin{align*}
    (c) 
        \leq & \left(1 - 2 (\lambda + \tau) \eta_h + (L + \tau)^2 \eta_h^2\right) \| \vecpi_h^i - \vecpi^* \|_2^2.
\end{align*}
However, unlike the bound in Lemma~\ref{lemma:full_error_recurrence}, the exploration parameter $\varepsilon$ will cause additional bias in the term $(b)$.
Define the random vector $\widetilde{\vecr}_h^i = \Exop[ \widehat{\vecr}_h^i | E_h^i, \mathcal{F}_h]$.
\begin{align*}
(b) = & 2\eta_h\alpha_h (\vecpi_h^i - \vecpi^*)^\top (\widehat{\vecr}_h^i - \vecF(\vecpi_h^i)) \\
 = & 2\eta_h\alpha_h (\vecpi_h^i - \vecpi^*)^\top (\widehat{\vecr}_h^i - \widetilde{\vecr}_h^i) + 2\eta_h\alpha_h (\vecpi_h^i - \vecpi^*)^\top (\widetilde{\vecr}_h^i - \vecF(\bar{\vecmu}_h)) \\
    & + 2\eta_h\alpha_h (\vecpi_h^i - \vecpi^*)^\top (\vecF(\bar{\vecmu}_h) - \vecF(\vecpi_h^i)) \\
\leq &2\eta_h\alpha_h \left( \frac{\lambda}{4} \|\vecpi_h^i - \vecpi^* \|_2^2 + \frac{1}{\lambda} \|\widetilde{\vecr}_h^i - \vecF(\bar{\vecmu}_h)\|_2^2\right) \\
    & + 2\eta_h\alpha_h \left(\frac{\lambda}{4} \|\vecpi_h^i - \vecpi^* \|_2^2 + \frac{1}{\lambda} \|\vecF(\bar{\vecmu}_h) - \vecF(\vecpi_h^i)\|_2^2 \right) \\
    &+ 2\eta_h\alpha_h (\vecpi_h^i - \vecpi^*)^\top (\widehat{\vecr}_h^i - \widetilde{\vecr}_h^i) \\
\leq & 2 \eta_h \frac{\lambda}{2} \|\vecpi_h^i - \vecpi^* \|_2^2 + \frac{2\eta_h}{\lambda}\|\widetilde{\vecr}_h^i - \vecF(\bar{\vecmu}_h)\|_2^2 + \frac{2\eta_h}{\lambda} \|\vecF(\bar{\vecmu}_h) - \vecF(\vecpi_h^i)\|_2^2 \\
    &+2\eta_h\alpha_h (\vecpi_h^i - \vecpi^*)^\top (\widehat{\vecr}_h^i - \widetilde{\vecr}_h^i),
\end{align*}
and similarly, if $\lambda = 0$, we have
\begin{align*}
(b) = & 2\eta_h\alpha_h (\vecpi_h^i - \vecpi^*)^\top (\widehat{\vecr}_h^i - \vecF(\vecpi_h^i)) \\
 = & 2\eta_h\alpha_h (\vecpi_h^i - \vecpi^*)^\top (\widehat{\vecr}_h^i - \widetilde{\vecr}_h^i) + 2\eta_h\alpha_h (\vecpi_h^i - \vecpi^*)^\top (\widetilde{\vecr}_h^i - \vecF(\bar{\vecmu}_h)) \\
    & + 2\eta_h\alpha_h (\vecpi_h^i - \vecpi^*)^\top (\vecF(\bar{\vecmu}_h) - \vecF(\vecpi_h^i)) \\
\leq &2\eta_h\alpha_h \left( \frac{\tau\delta}{2} \|\vecpi_h^i - \vecpi^* \|_2^2 + \frac{1}{2\tau\delta} \|\widetilde{\vecr}_h^i - \vecF(\bar{\vecmu}_h)\|_2^2\right) \\
    & + 2\eta_h\alpha_h \left(\frac{\tau\delta}{2} \|\vecpi_h^i - \vecpi^* \|_2^2 + \frac{1}{2\tau\delta} \|\vecF(\bar{\vecmu}_h) - \vecF(\vecpi_h^i)\|_2^2 \right) \\
    &+ 2\eta_h\alpha_h (\vecpi_h^i - \vecpi^*)^\top (\widehat{\vecr}_h^i - \widetilde{\vecr}_h^i) \\
\leq & 2 \eta_h \tau\delta \|\vecpi_h^i - \vecpi^* \|_2^2 + \frac{\eta_h}{\tau\delta}\|\widetilde{\vecr}_h^i - \vecF(\bar{\vecmu}_h)\|_2^2 + \frac{\eta_h}{\tau\delta} \|\vecF(\bar{\vecmu}_h) - \vecF(\vecpi_h^i)\|_2^2 \\
    &+2\eta_h\alpha_h (\vecpi_h^i - \vecpi^*)^\top (\widehat{\vecr}_h^i - \widetilde{\vecr}_h^i).
\end{align*}

Denote $\vecpi_{unif} := \frac{1}{K}\vecone_K$.
The remaining error terms we can bound by (using the auxiliary coupling random actions $\bar{a}_h^j$ from Lemma~\ref{lemma:exploration_bias_trpa_bandit}):
\begin{align*}
    |\Exop[2\eta_h\alpha_h (\vecpi_h^i - \vecpi^*)^\top (\widehat{\vecr}_h^i - \widetilde{\vecr}_h^i) | \mathcal{F}_h] | 
        \leq & |\Exop[2\eta_h\alpha_h (\vecpi_h^i - \vecpi^*)^\top (\widehat{\vecr}_h^i - \widetilde{\vecr}_h^i) | E_h^i, \mathcal{F}_h] \mathbb{P}(E_h^i) \\
        & + \Exop[2\eta_h\alpha_h (\vecpi_h^i - \vecpi^*)^\top (\widehat{\vecr}_h^i - \widetilde{\vecr}_h^i) | \overline{E_h^i}, \mathcal{F}_h]\mathbb{P}(\overline{E_h^i})| \\
    \leq & 2\eta_h \Exop\left[\|\vecpi_h^i - \vecpi^*\|_2 \|\widehat{\vecr}_h^i - \widetilde{\vecr}_h^i\|_2 | \overline{E_h^i}, \mathcal{F}_h\right] \mathbb{P}(\overline{E_h^i}) \\
    \leq & 8\eta_h K^{\sfrac{3}{2}} \sqrt{1 + \sigma^2} \mathbb{P}(\overline{E_h^i}) 
\end{align*}
and $\mathbb{P}(\overline{E_h^i}) \leq \exp\{-\varepsilon T_h\}$.
Furthermore,
\begin{align*}
    \Exop[&\|\widetilde{\vecr}_h^i - \vecF(\bar{\vecmu}_h)\|_2^2 | \mathcal{F}_{h}] \\
    \leq & 2\Exop\left[ \|\widetilde{\vecr}_h^i - \vecF(\widehat{\vecmu}(\bar{a}_{h,exp}^i, \bar{a}_h ^ {-i})) \|_2^2\middle|\mathcal{F}_{h}\right] + 2\Exop\left[ \|\vecF(\widehat{\vecmu}(\bar{a}_{h,exp}^i, \bar{a}_h ^ {-i})) - \vecF(\bar{\vecmu}_h)\|_2^2 | \mathcal{F}_{h}\right] \\
    \leq &  \frac{8 L^2}{N}  +  2L^2 \Exop[ \|\widehat{\vecmu}(\bar{a}_{exp}^i, \bar{a} ^ {-i}) - \frac{1}{N}\sum_{j=1}^N \vecpi^j_h\|_2^2 | \mathcal{F}_{h}] \\
    \leq &  \frac{8 L^2}{N}  +  4L^2 \Exop[ \|\widehat{\vecmu}(\bar{a}_{exp}^i, \bar{a} ^ {-i}) - \frac{1}{N}\sum_{j\neq i} (\varepsilon\vecpi_{unif} + (1-\varepsilon)\vecpi^j_h) - \frac{\vecpi_{unif}}{N} \|_2^2 | \mathcal{F}_{h}] \\
        &+ 4L^2 \Exop[ \| \varepsilon\frac{1}{N}\sum_{j\neq i} ( \varepsilon\vecpi^j_h - \vecpi_{unif}) +  \frac{\vecpi^i_h - \vecpi_{unif}}{N} \|_2^2 | \mathcal{F}_{h}] \\
    \leq &  \frac{8 L^2}{N}  +  \frac{8L^2}{N} + 4L^2 (2\varepsilon^2 + \frac{4}{N}) \\
    \leq & \frac{64 L ^2}{N} + 8 L^2 \varepsilon^2,
\end{align*}
and finally, using the trivial Lipschitz continuity property: $\|\vecF(\bar{\vecmu}_h) - \vecF(\vecpi_h^i)\|_2^2 \leq L^2 \|\bar{\vecmu}_h - \vecpi_h^i\|_2^2 = L^2 e_h^i$.

\subsection{Proof of Lemma~\ref{lemma:bandit_pol_deviation}}\label{sec:proof_lemma_bandit_pol_dev}

We will reuse the definitions of Section~\ref{sec:bandit_extended_defs}.
Once again, repeating the derivations from Lemma~\ref{lemma:policy_variations_bound_trpa_full}, we have that
\begin{align*}
    \| \vecpi^i_{h+1} - \vecpi^j_{h+1} \|_2^2 
    = & (1 - \tau \eta_h)^2 \| \vecpi^i_h - \vecpi^j_h \|_2^2 + \eta_h^2 \|  \widehat{\vecr}_h^i - \widehat{\vecr}_h^j \|_2^2 + 2 (1 - \tau \eta_h) \eta_h (\vecpi^i_h - \vecpi^j_h) ^ \top ( \widehat{\vecr}_h^i - \widehat{\vecr}_h^j ).
\end{align*}
Unlike in the expert feedback case, the last term does not vanish in expectation when bounding policy deviation.
\begin{align*}
    \Exop \left[ \| \vecpi^i_{h+1} - \vecpi^j_{h+1} \|_2^2 | \mathcal{F}_h \right] \leq & (1 - \tau \eta_h)^2 \| \vecpi^i_h - \vecpi^j_h \|_2^2 + \Exop \left[ \eta_h^2 \|  \widehat{\vecr}_h^i - \widehat{\vecr}_h^j \|_2^2 | \mathcal{F}_h \right] \\
        &+ 2 (1 - \tau \eta_h) \eta_h (\vecpi^i_h - \vecpi^j_h) ^ \top \Exop\left[\widehat{\vecr}_h^i - \widehat{\vecr}_h^j | \mathcal{F}_h \right] \\
    \leq & (1 - \tau \eta_h)^2 \| \vecpi^i_h - \vecpi^j_h \|_2^2 + \Exop \left[ \eta_h^2 \|  \widehat{\vecr}_h^i - \widehat{\vecr}_h^j \|_2^2 | \mathcal{F}_h \right] \\
        &+ 2 \eta_h (\vecpi^i_h - \vecpi^j_h) ^ \top \left[ \Exop\left[\widehat{\vecr}_h^i \middle|E_{h}^i, \mathcal{F}_h\right] + \vecb_h^i - \Exop\left[\widehat{\vecr}_h^j \middle|E_{h}^j, \mathcal{F}_h\right] - \vecb_h^j \right]  \\
    \leq & (1 - \tau \eta_h)^2 \| \vecpi^i_h - \vecpi^j_h \|_2^2 + \Exop \left[ \eta_h^2 \|  \widehat{\vecr}_h^i - \widehat{\vecr}_h^j \|_2^2 | \mathcal{F}_h \right] \\
        &+ 2 \eta_h (\vecpi^i_h - \vecpi^j_h) ^ \top \left[ \Exop\left[\widehat{\vecr}_h^i \middle|E_{h}^i, \mathcal{F}_h\right] - \Exop\left[\widehat{\vecr}_h^j \middle|E_{h}^j, \mathcal{F}_h\right] \right] + 8 \eta_h \exp\left\{ -\varepsilon T_{h}\right\},
\end{align*}
where the last line follows from the bound on $\vecb_h^i$ in Lemma~\ref{lemma:exploration_bias_trpa_bandit}.
Furthermore, using Young's inequality, we obtain
\begin{align*}
    \Exop \left[ \| \vecpi^i_{h+1} - \vecpi^j_{h+1} \|_2^2 | \mathcal{F}_h \right] \leq & (1 - \tau \eta_h)^2 \| \vecpi^i_h - \vecpi^j_h \|_2^2 + \Exop \left[ \eta_h^2 \|  \widehat{\vecr}_h^i - \widehat{\vecr}_h^j \|_2^2 | \mathcal{F}_h \right] + 8 \eta_h \exp\left\{ -\varepsilon T_{h}\right\} \\
        &+ \frac{\tau\eta_h}{2} \|\vecpi^i_h - \vecpi^j_h\|_2^2 + \frac{ \eta_h\tau^{-1}}{2}\left\| \Exop\left[\widehat{\vecr}_h^i \middle|E_{h}^i, \mathcal{F}_h\right] - \Exop\left[\widehat{\vecr}_h^j \middle|E_{h}^j, \mathcal{F}_h\right] \right\|_2^2 \\
     \leq & (1 - \tau \eta_h)^2 \| \vecpi^i_h - \vecpi^j_h \|_2^2 + \Exop \left[ \eta_h^2 \|  \widehat{\vecr}_h^i - \widehat{\vecr}_h^j \|_2^2 | \mathcal{F}_h \right] + 8 \eta_h \exp\left\{ -\varepsilon T_{h}\right\} \\
        &+ \frac{\tau\eta_h}{2} \|\vecpi^i_h - \vecpi^j_h\|_2^2 + \frac{2\eta_h\tau^{-1} L ^ 2}{N^2}.
\end{align*}
With the choice of $T_h = \varepsilon^{-1} \log (h+2)$ and noting that $\| \vecpi^i_{h} - \vecpi_h^j\|_2 \leq 2$, we obtain the recurrence
\begin{align*}
    \Exop \left[ \| \vecpi^i_{h+1} - \vecpi^j_{h+1} \|^2_2 \right] \leq & \left(1 - \frac{\sfrac{3}{2}}{h+2}\right) \Exop \left[ \| \vecpi^i_h - \vecpi^j_h \|_2^2\right] + \frac{4\tau^{-2} K^3 (\sigma^2 + 1) + 8\tau^{-1} + 4}{(h+2)^2}   \\
        & + \frac{4\tau^{-2} L ^ 2}{N^2 (h+2)}.
\end{align*}
Hence, by invoking the recurrence lemma (Lemma~\ref{lemma:general_recurrence}, with $a=2, c_0 = \frac{4\tau^{-2} L ^ 2}{N^2}, c_1 = 4\tau^{-2} K^3 (\sigma^2 + 1) + 8\tau^{-1} + 4, \gamma = \sfrac{3}{2}, u_0 = 0$), we have
\begin{align*}
    \Exop \left[ \| \vecpi^i_{h+1} - \vecpi^j_{h+1} \|_2^2 \right] \leq & \frac{24\tau^{-2} K^3 (\sigma^2 + 1) + 48\tau^{-2} + 24}{h+2} + \frac{16\tau^{-2} L ^ 2}{N^2}.
\end{align*}
The statement in the lemma follows as in the full feedback case.

\subsection{Proof of Theorem~\ref{theorem:bandit_short}}\label{sec:bandit_theorem_full}

Using Lemma~\ref{lemma:bandit_main_recurrence}, for the strongly monotone case $\lambda > 0$ we have that
\begin{align*}
    u_{h+1}^i \leq & 4 \eta_h^2 K^3(1 + \sigma^2) + 8\eta_h^2(L+\tau)^2 + 8 K^{\sfrac{3}{2}}\eta_h \sqrt{1+\sigma^2}  \exp\{-\varepsilon T_h\} \\
        &+128\eta_h\lambda^{-1} L^2 N^{-1} + 16\eta_h\lambda^{-1}L^2\varepsilon^2 + 2\eta_h\lambda^{-1} L^2 \Exop\left[e_h^i\right] \\
        &+\left(1 - 2 \eta_h(\sfrac{\lambda}{2} + \tau)\right) u_{h}^i,
\end{align*}
and for the monotone case it holds that 
\begin{align*}
    u_{h+1}^i \leq &  4\eta_h^2 K^3(\sigma^2 + 1) + 8 \eta_h^2 (L+\tau)^2 + 8 K^{\sfrac{3}{2}} \eta_h \sqrt{1+\sigma^2} \exp\{-\varepsilon T_h\} \\
    &+64\tau^{-1} \eta_h \delta^{-1}L^2 N^{-1}+8\tau^{-1} \eta_h \delta^{-1}L^2 \varepsilon^{2} + \tau^{-1} \eta_h\delta^{-1}L^2 \Exop\left[e_h^i\right] \\  
        &+ \left(1 - 2 \tau \eta_h (1-\delta)\right) u_{h}^i.
\end{align*}
By Lemma~\ref{lemma:bandit_pol_deviation}, we know that
\begin{align*}
    \Exop[e_h^i] \leq \frac{12\tau^{-2} K^3 (\sigma^2 + 1) + 48\tau^{-2} + 24}{h+2} + \frac{16\tau^{-2} L ^ 2}{N^2}.
\end{align*}
Placing this bound as well as $T_h = \lceil \varepsilon^{-1} \log (h+2) \rceil$ and $\eta_h = \frac{\tau^{-1}}{h+2}$, we obtain the recurrences which are solved by using Lemma~\ref{lemma:general_recurrence}.
In the monotone case, we pick $\delta=\sfrac{1}{4}$ as before.

The bound in the statement of the theorem in the main body of the paper follows from the fact that the lengths of the exploration epochs scale with $T_h = \mathcal{O}(\varepsilon^{-1} \log (h+2)) = \widetilde{\mathcal{O}}(\varepsilon^{-1})$.

\section{Details of Experiments}\label{sec:experiments_detailed}

\textbf{Setup.}
All experiments were run on single core of an AMD EPYC 7742 CPU, a single experiment with 1000 independent agents and 100000 iterations takes roughly 1 hour.
For all experiments, we use parameters $\tau, \epsilon$ as implied by Corollaries~\ref{corollary:expert},\ref{corollary:bandit}.
Projections to the probability simplex were implemented using the algorithm by \cite{duchi2008efficient}.

\subsection{Problem Generation Details}

We provide further details on how we generate/simulate the SMFG problems.

\textbf{Linear payoffs.}
We generate a payoff map
\begin{align*}
    \vecF_{lin}(\vecmu) := (\matS + \matX) \vecmu + \vecb
\end{align*}
for some $\matS \in \mathbb{S}_{++}^{K\times K}$ and $\matX$ anti-symmetric matrix, which makes $ \vecF_{lin}$ monotone.
We randomly sample $\matS$ from a Wishart distribution (which has support contained in positive definite matrices), generate $\matX$ by computing $\frac{\matU - \matU^\top}{2}$ for a random matrix $\matU$ with entries sampled uniformly from $[0,1]$ and $\vecb$ having entries uniformly sampled from $[0,1]$.

\textbf{Payoffs with KL potential.}
Next, we construct the following payoff operator $\vecF_{KL}$ for some reference distribution $\vecmu_{\text{ref}} \in \Delta_\setA$:
\begin{align*}
    \Phi_{KL}(\vecmu) &:= \operatorname{D}_{KL}(\gamma\vecmu + (1-\gamma)\vecmu_{\text{ref}}||\vecmu_{\text{ref}}) \\
    \vecF_{KL}(\vecmu) &:= \nabla \Phi_{KL} (\vecmu) \\
        \vecF_{KL}(\vecmu, a) &= \gamma \log\left(\frac{\gamma \vecmu(a) + (1-\gamma) \vecmu_{\text{ref}}(a)}{\vecmu_{\text{ref}}(a)}\right) + \gamma
\end{align*}
Note that as $\Phi_{KL}$ is convex, $\vecF_{KL}$ is monotone.
In our experiments, we use $\gamma = 0.1$, and we generate $\vecmu_{\text{ref}}$ by sampling $K$ uniform random variables in $[0,1]$ and normalizing.

\textbf{Beach bar process.}
Following the example given in \cite{perrin2020fictitious}, we use the action set $\setA = \{1, \ldots, K \}$ for potential locations at the beach and assume a bar is located at $x_{bar} := \lfloor\frac{K}{2} \rfloor$.
Taking into the proximity to the bar and the occupancy measure over actions (i.e., the crowdedness of locations at the beach), the payoff map is given by:
\begin{align*}
    \vecF_{bb}(\vecmu, a) = 1-\frac{|a - x_{bar}|}{K} - \alpha \log(1 + \vecmu(a)).
\end{align*}
The above payoff map is monotone.
We use $\alpha = 1$ for our experiments.

\subsection{Learning curves - Full Feedback}

We provide the learning curves under full feedback for various choices of the number of agents $N \in \{ 20, 100, 1000 \}$.
The errors in terms of maximum exploitability and distance to MF-NE are presented in Figure~\ref{figure:expert_exp_curves} and Figure~\ref{figure:expert_l2_curves} respectively.

\begin{figure}[h]
\centering
\begin{tabular}{cc}
  \includegraphics[width=0.35\linewidth]{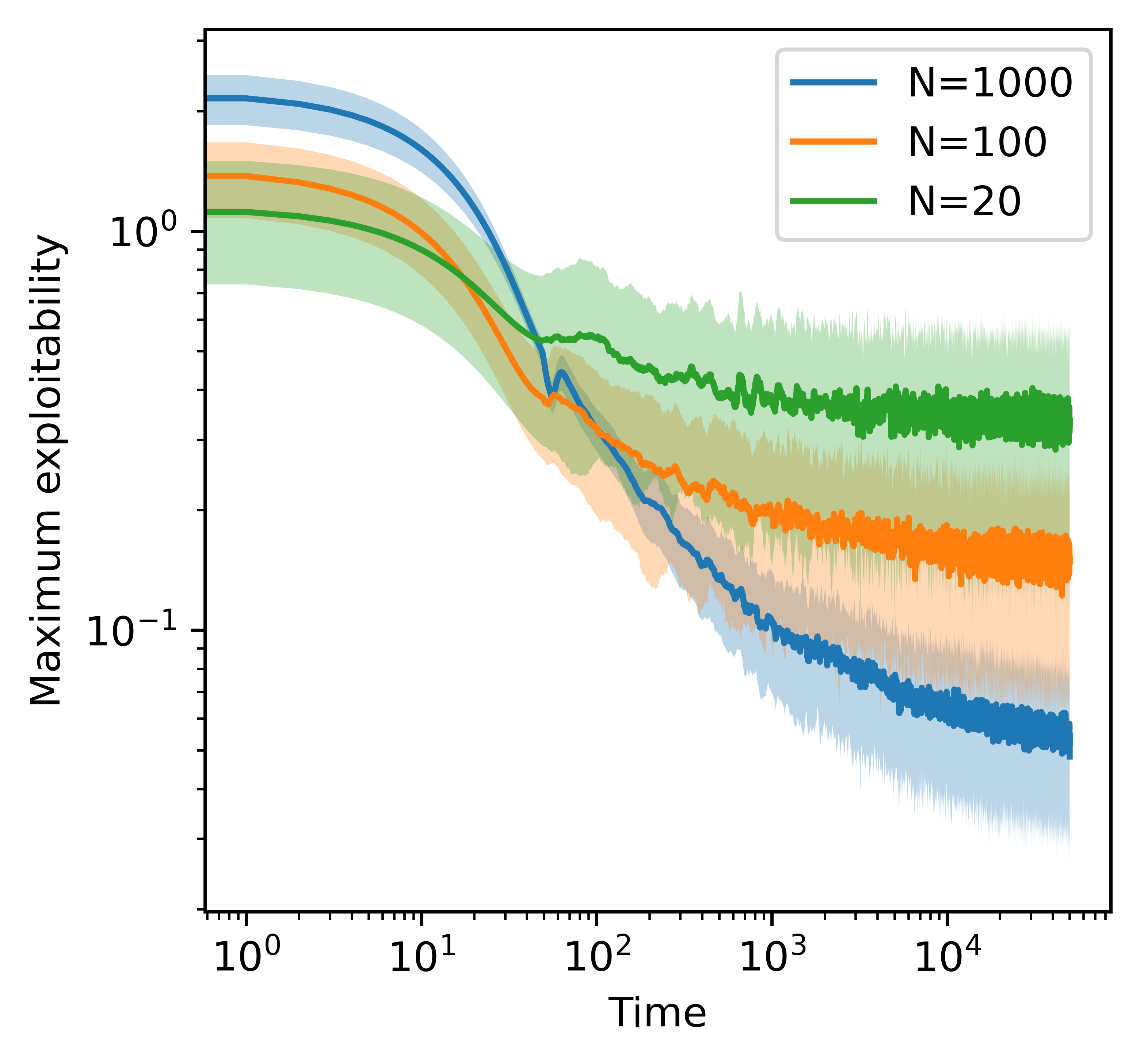} &   \includegraphics[width=0.35\linewidth]{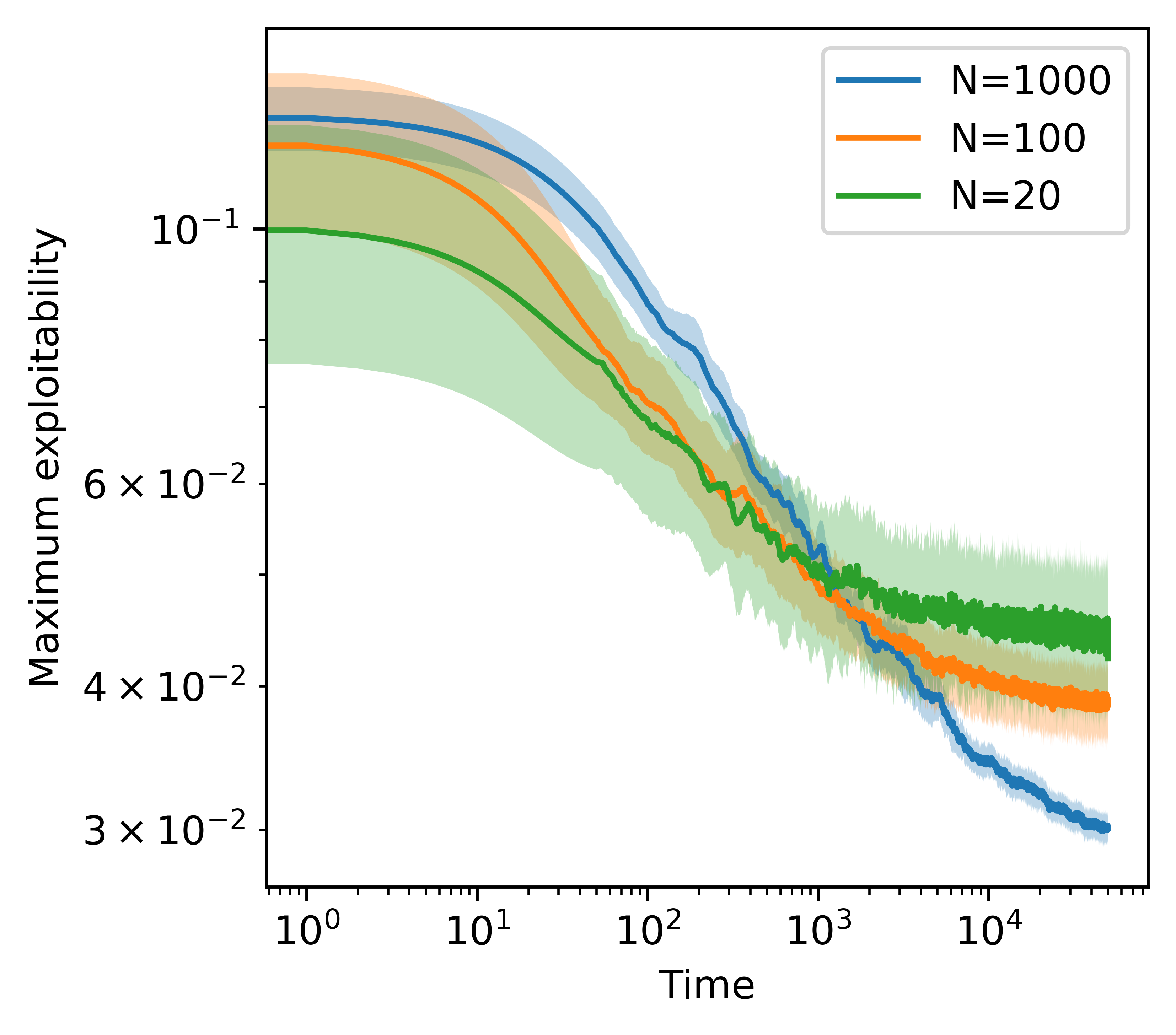} \\
(a) & (b) \\
\includegraphics[width=0.35\linewidth]{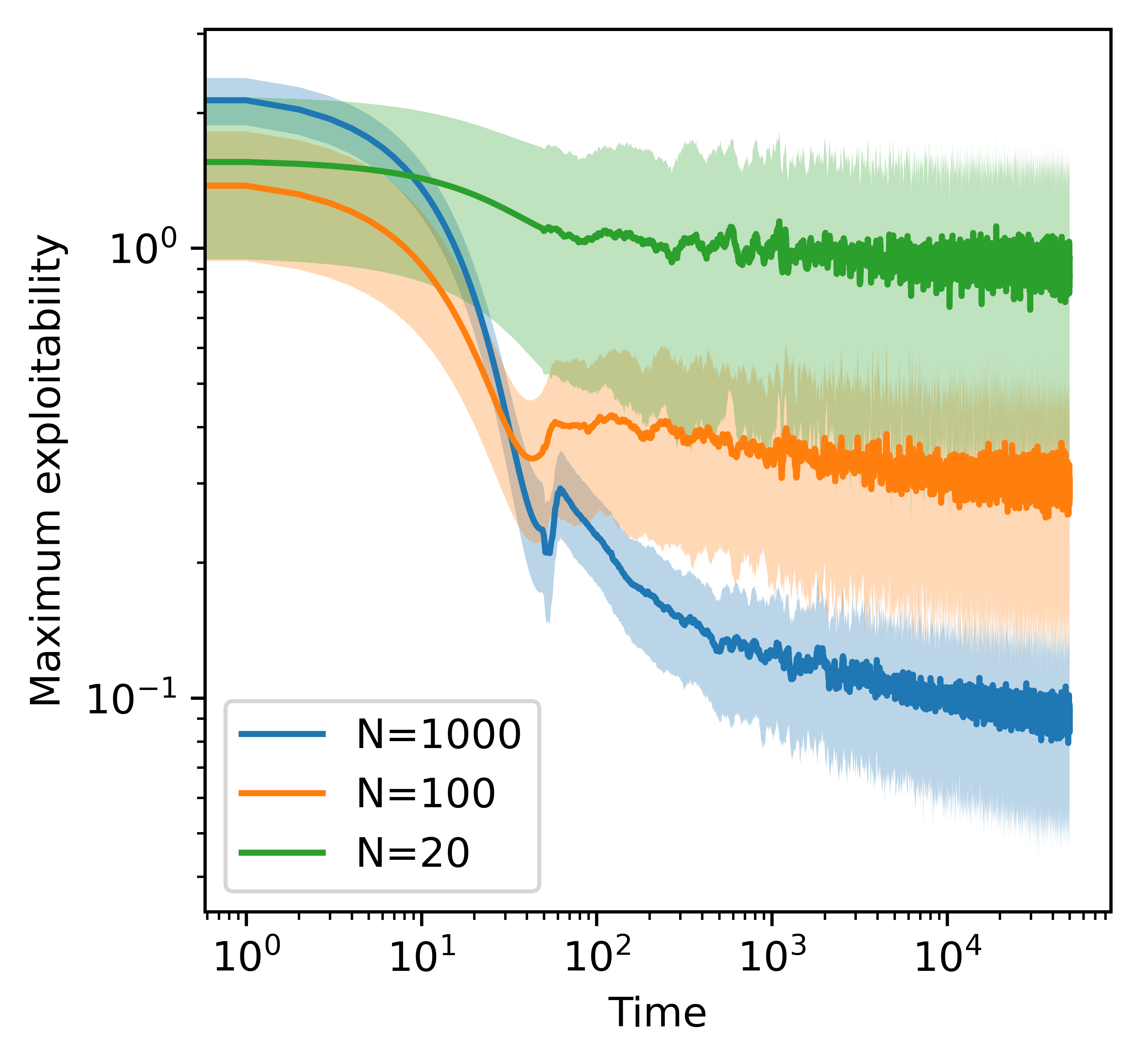} &   \includegraphics[width=0.35\linewidth]{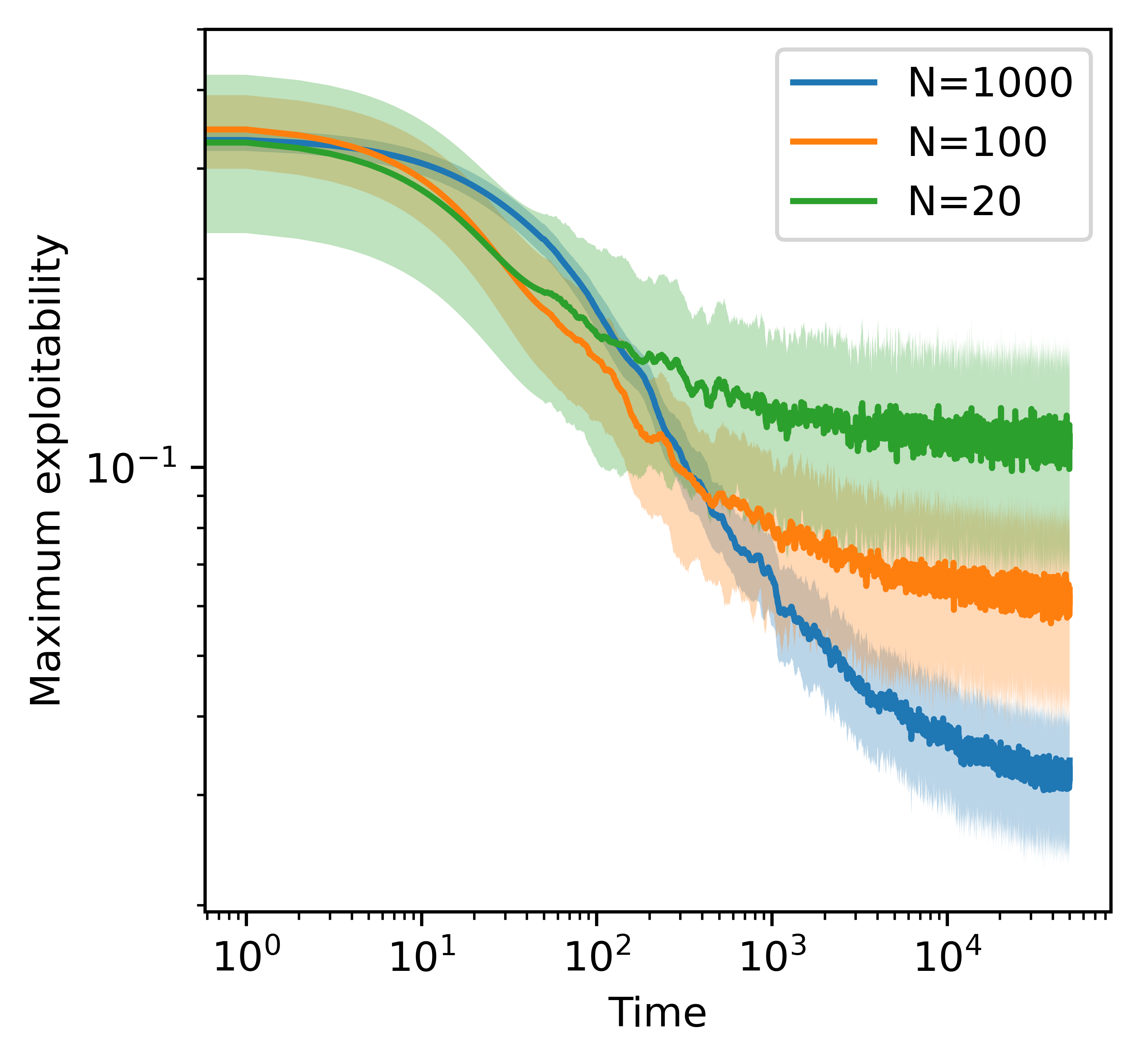} \\
(c) & (d)
\end{tabular}
\caption{
The (smoothed) maximum exploitability $\max_{i\in\setN} \phi^i(\{\vecpi^j\}_{j=1}^N)$ among $N$ agents throughout learning with full feedback for three different $N$, on the problems (a) linear payoffs, (b) exponentially decreasing payoffs, (c) payoffs with KL potential and (d) the beach bar payoffs.
}
\label{figure:expert_exp_curves}
\end{figure}

As expected, the games with larger number of players $N$ converge to better approximate NE in the sense that the final maximum exploitability is smaller at convergence.
Furthermore, in most cases the exploitability converges slightly slower with more agents, also supporting the theoretical finding that there is a dependence on $N$.
As before, the exploitability curves have oscillations at later stages of the training, even though they remain upper bounded as foreseen the theoretical results.
This does not contradict our results as long as for larger $N$, the upper bound on the oscillations us smaller.
The confidence intervals plotted in figures support a high-probability upper bound on the maximum exploitability as one would expect.

\begin{figure}[h]
\centering
\begin{tabular}{cc}
  \includegraphics[width=0.35\linewidth]{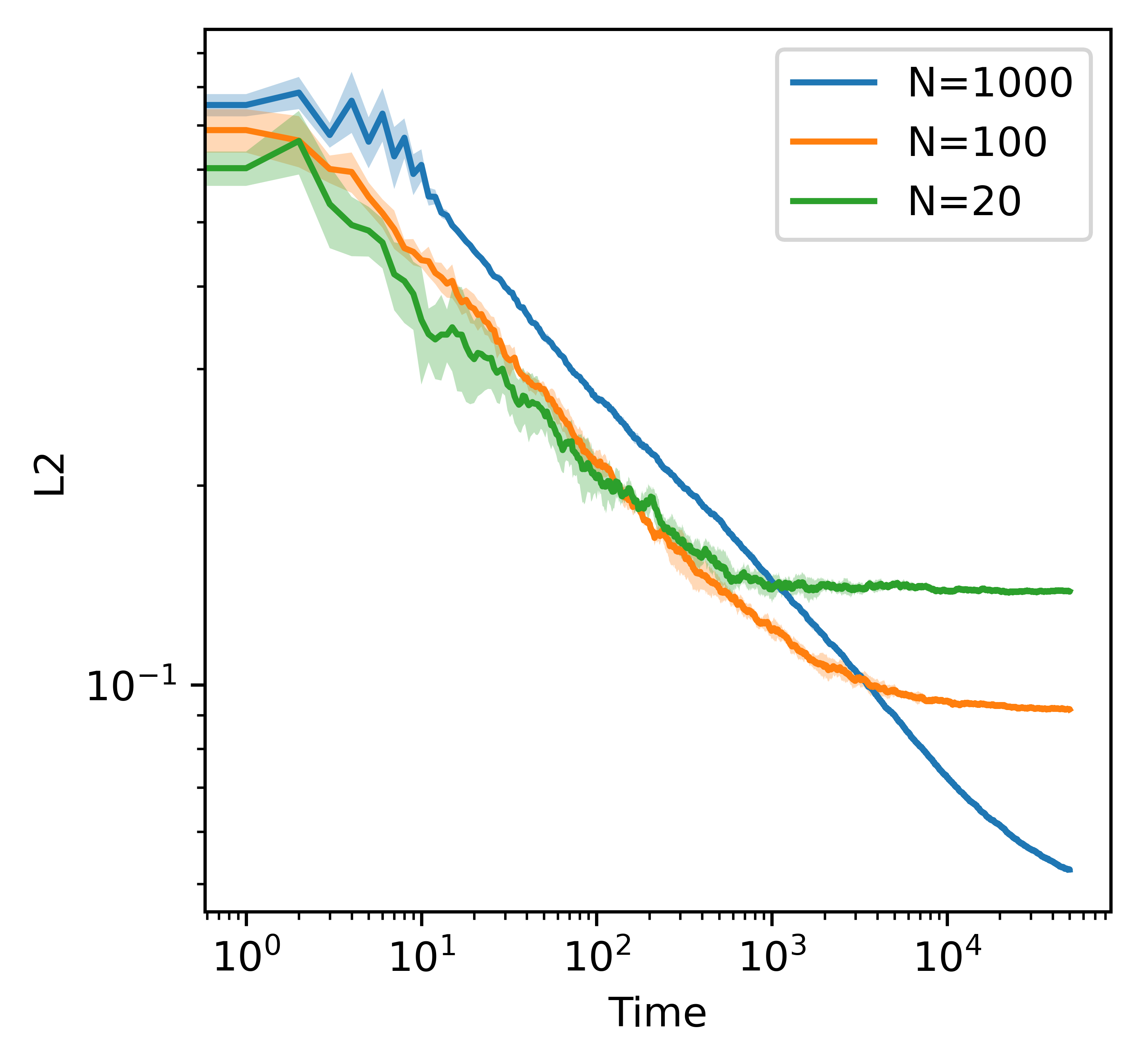} &   \includegraphics[width=0.35\linewidth]{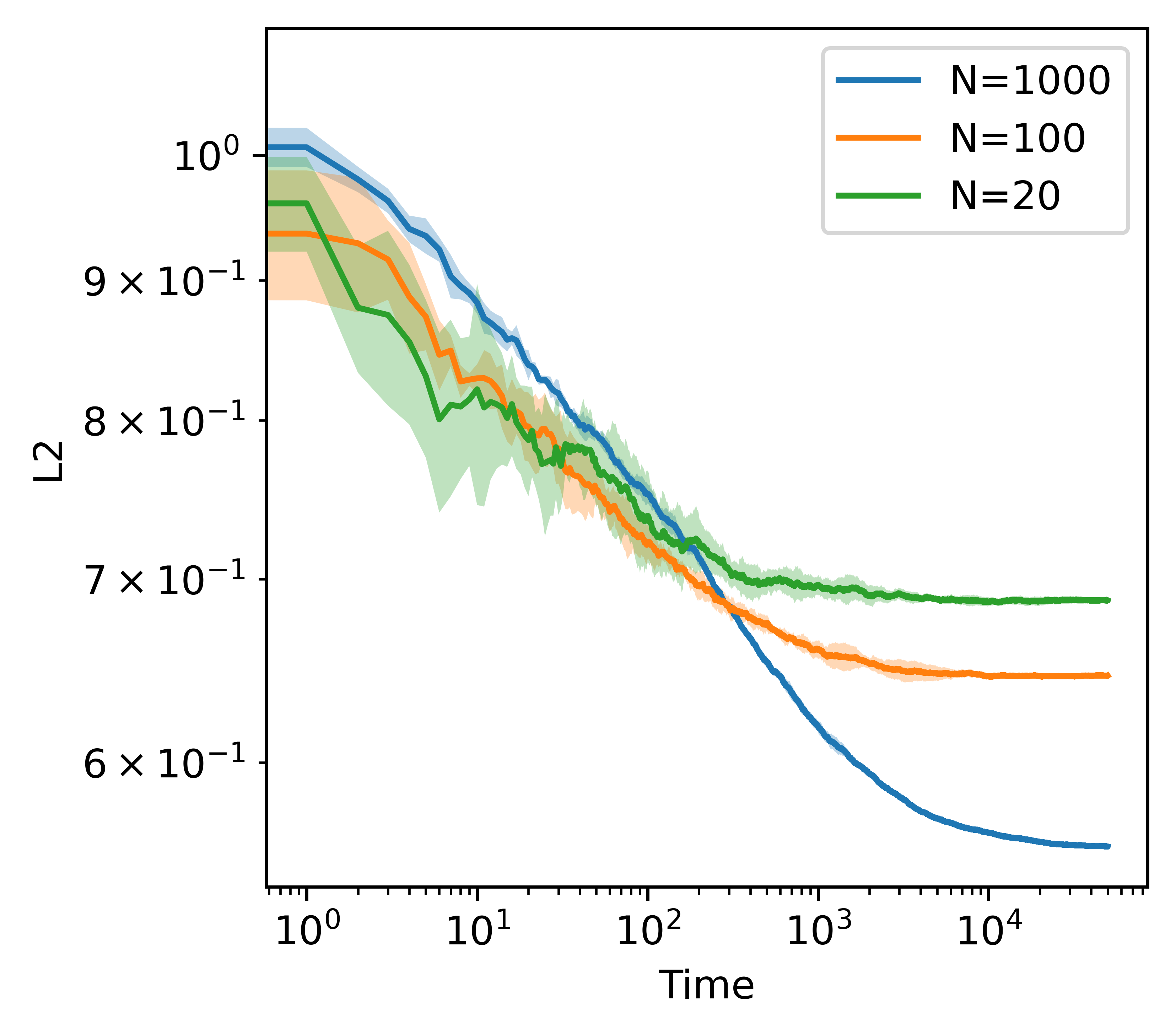} \\
(a) & (b) \\
\includegraphics[width=0.35\linewidth]{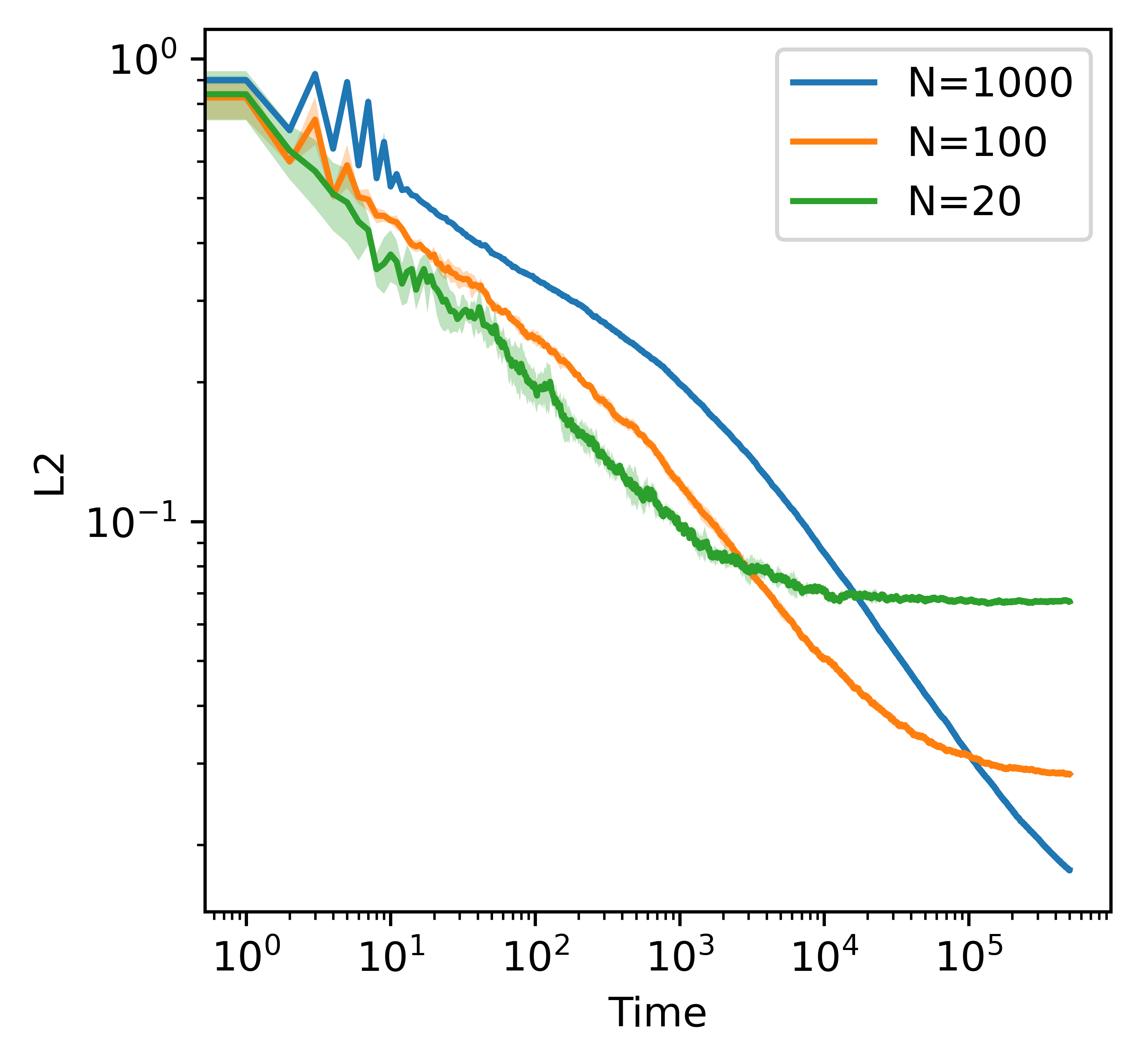} &   \includegraphics[width=0.35\linewidth]{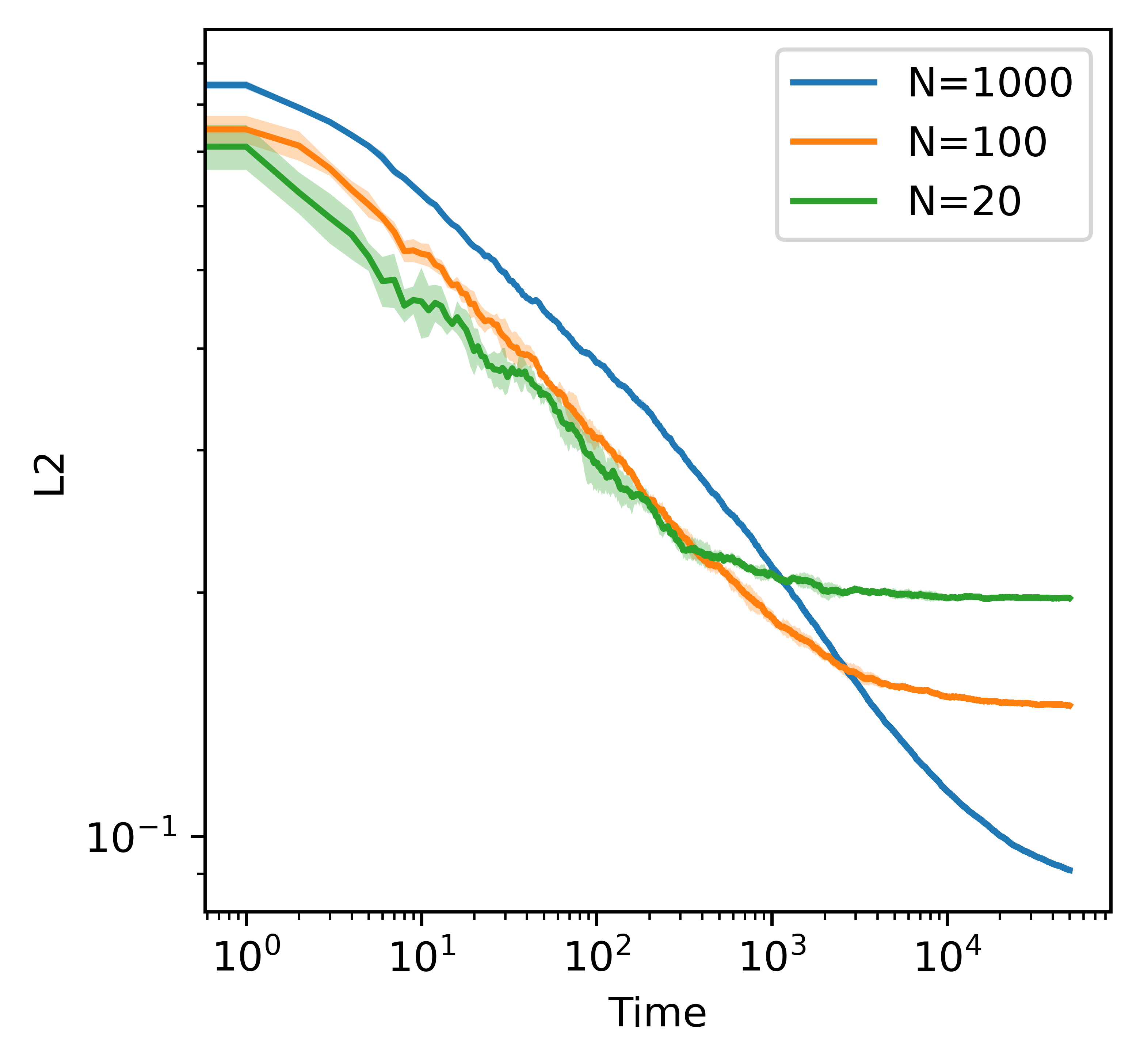} \\
(c) & (d)
\end{tabular}
\caption{
The mean $\ell_2$ distance to MF-NE given by $\frac{1}{N}\sum_{i\in\setN} \| \vecpi^i - \vecpi^*\|_2$ with $N$ agents throughout learning with full feedback for three different $N$, on the problems (a) linear payoffs, (b) exponentially decreasing payoffs, (c) payoffs with KL potential and (d) the beach bar payoffs.
}
\label{figure:expert_l2_curves}
\end{figure}

\subsection{Learning curves - Bandit Feedback}

We provide the learning curves under bandit feedback for various choices of the number of agents $N \in \{ 50, 100, 1000 \}$.
The errors in terms of maximum exploitability and distance to MF-NE are presented in Figure~\ref{figure:bandit_exp_curves} and Figure~\ref{figure:bandit_l2_curves} respectively.

\begin{figure}[h]
\centering
\begin{tabular}{cc}
  \includegraphics[width=0.35\linewidth]{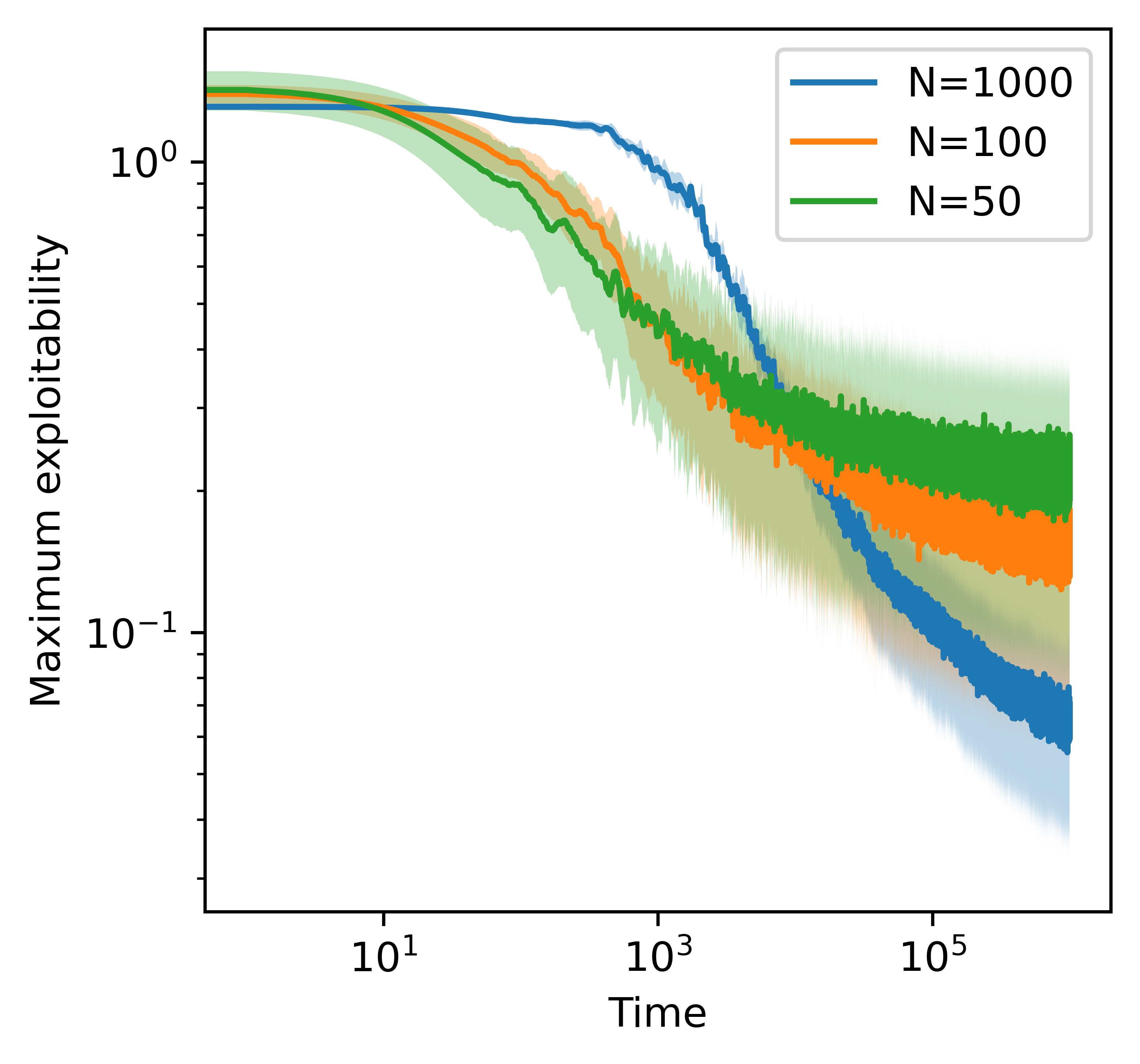} &   \includegraphics[width=0.35\linewidth]{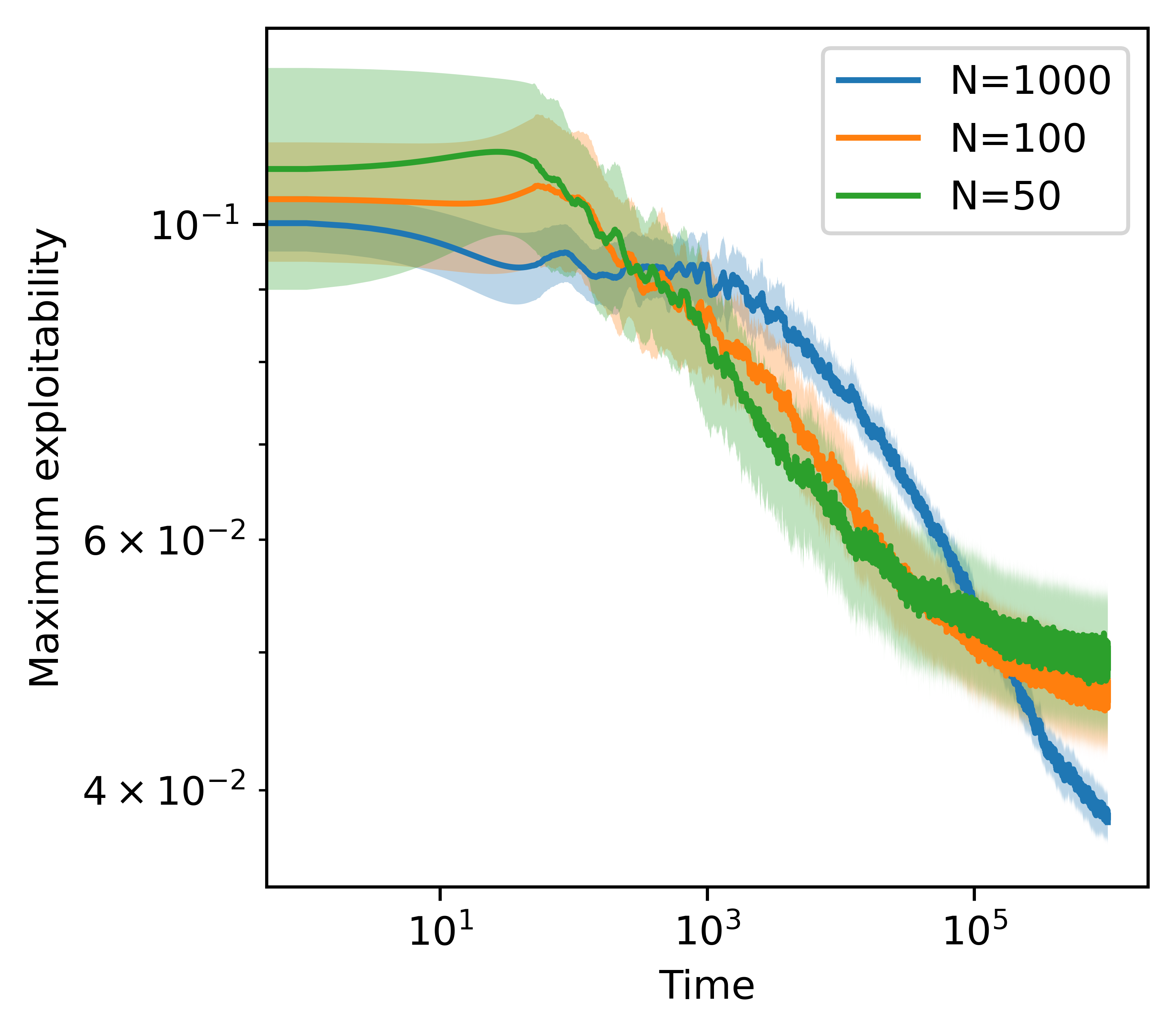} \\
(a) & (b) \\
\includegraphics[width=0.35\linewidth]{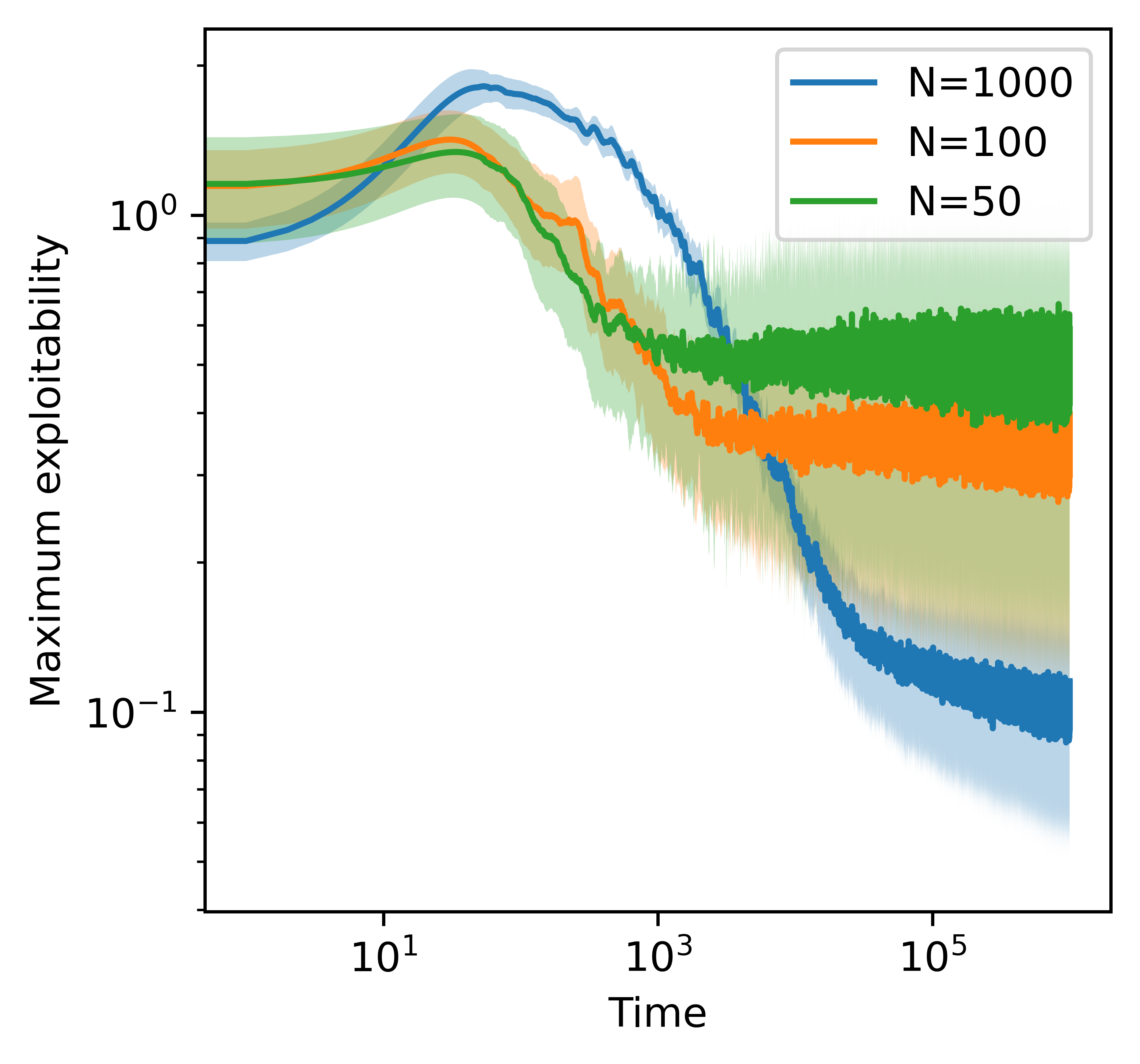} &   \includegraphics[width=0.35\linewidth]{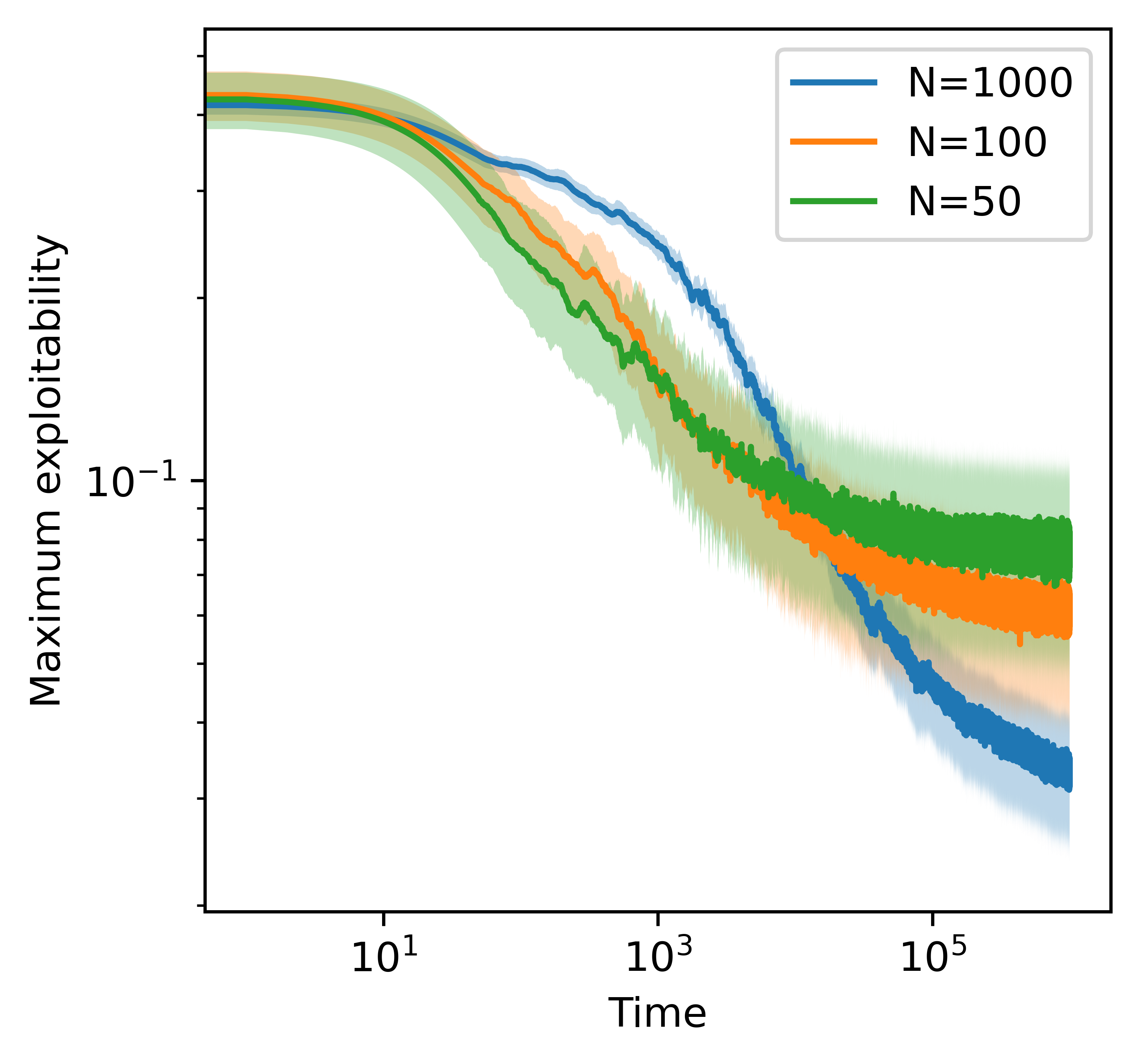} \\
(c) & (d)
\end{tabular}
\caption{
The (smoothed) maximum exploitability $\max_{i\in\setN} \phi^i(\{\vecpi^j\}_{j=1}^N)$ among $N$ agents throughout learning with bandit feedback for three different $N$, on the problems (a) linear payoffs, (b) exponentially decreasing payoffs, (c) payoffs with KL potential and (d) the beach bar payoffs.
}
\label{figure:bandit_exp_curves}
\end{figure}

As in the case of full feedback, the curves converge to smaller values as $N$ increases.
Furthermore, one straightforward observation is that the variance at early stages of learning is much higher than in the full feedback case.
This can be due to the added variance of the importance sampling estimator constructed through exploration epochs.
As exploration occurs in shorter duration at early epochs, the variance between agent policies will be high as well, explaining the initial increase in exploitability in certain toy experiments in Figure~\ref{figure:bandit_exp_curves}.

\begin{figure}[h]
\centering
\begin{tabular}{cc}
  \includegraphics[width=0.35\linewidth]{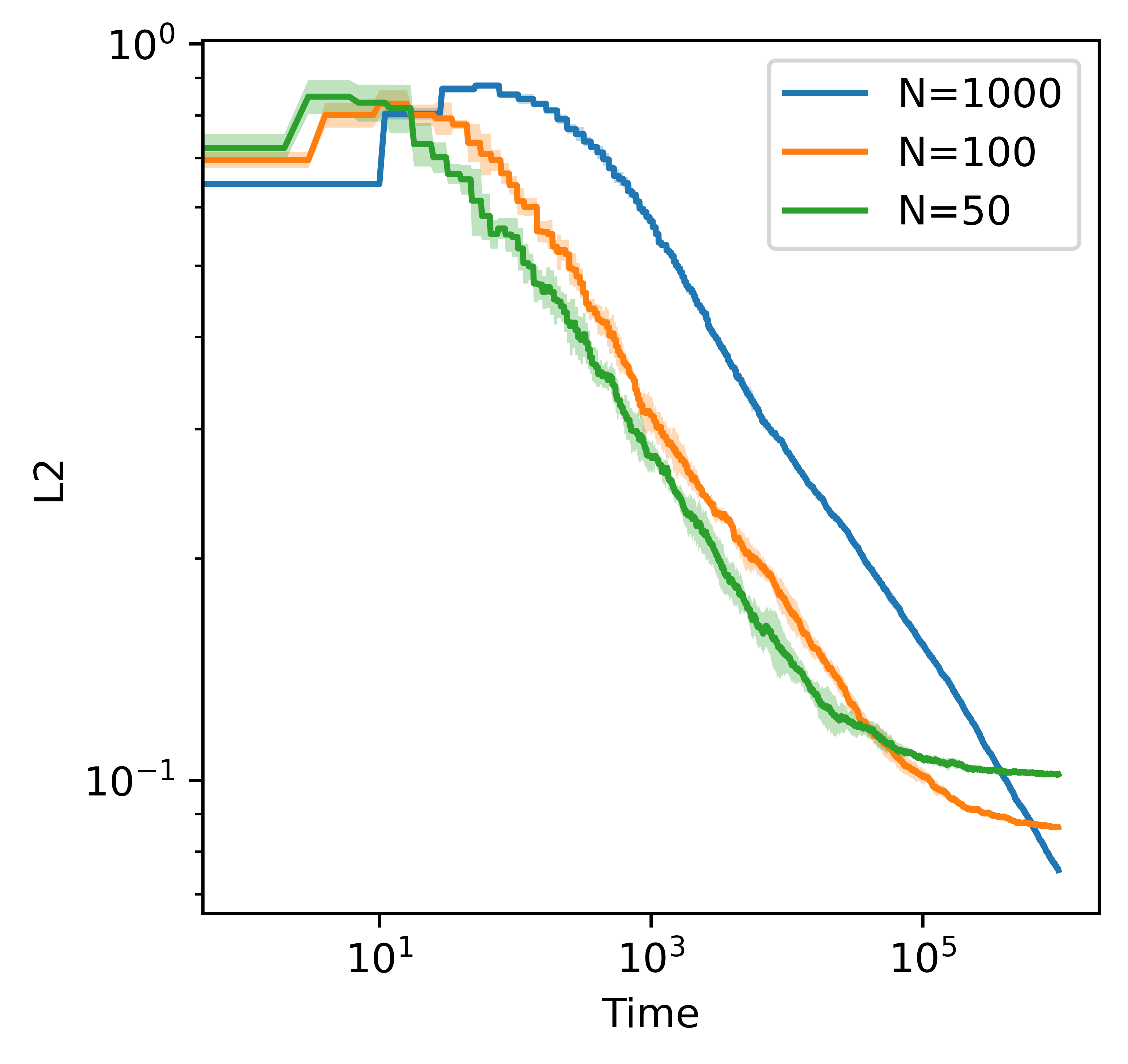} &   \includegraphics[width=0.35\linewidth]{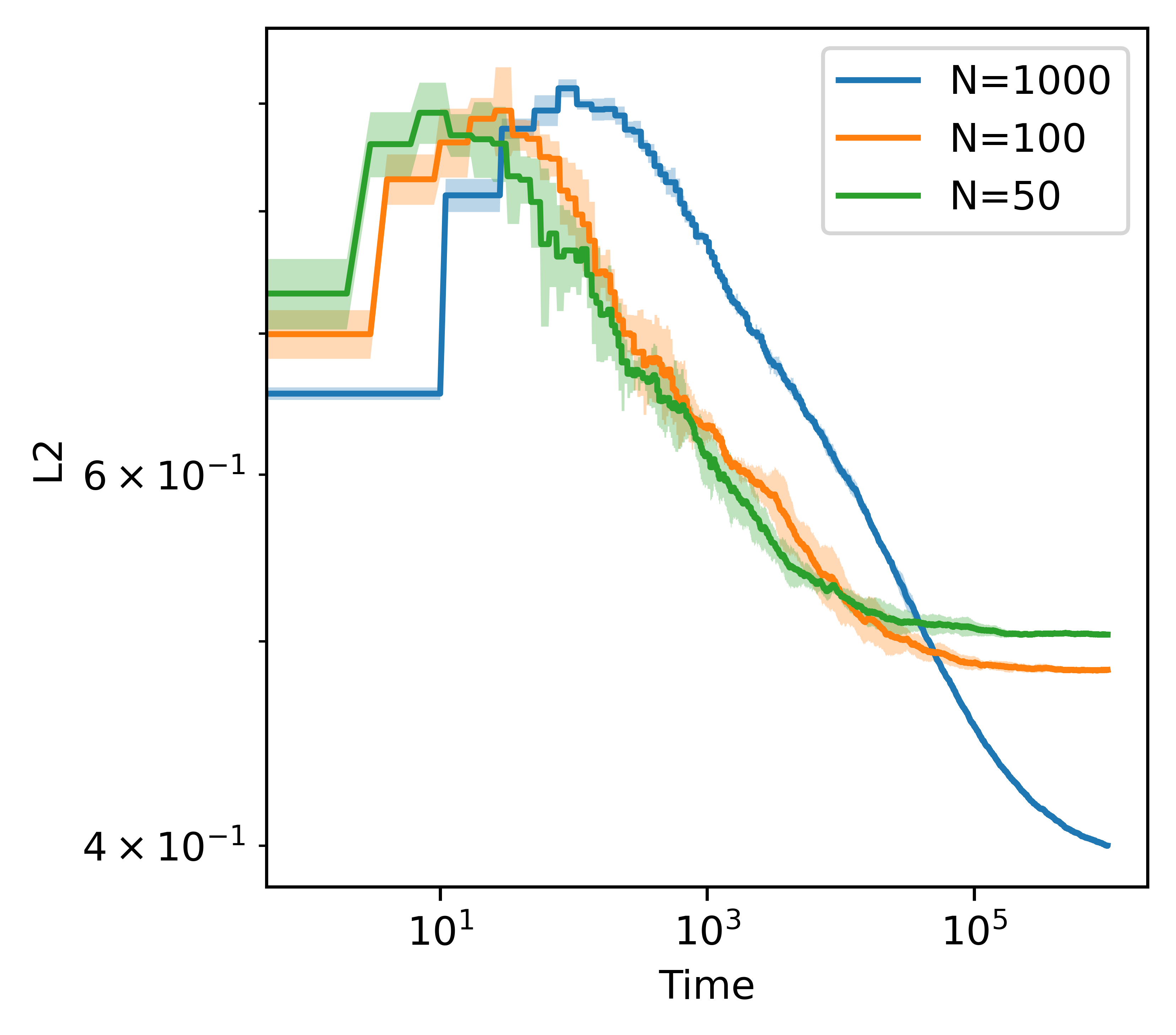} \\
(a) & (b) \\
\includegraphics[width=0.35\linewidth]{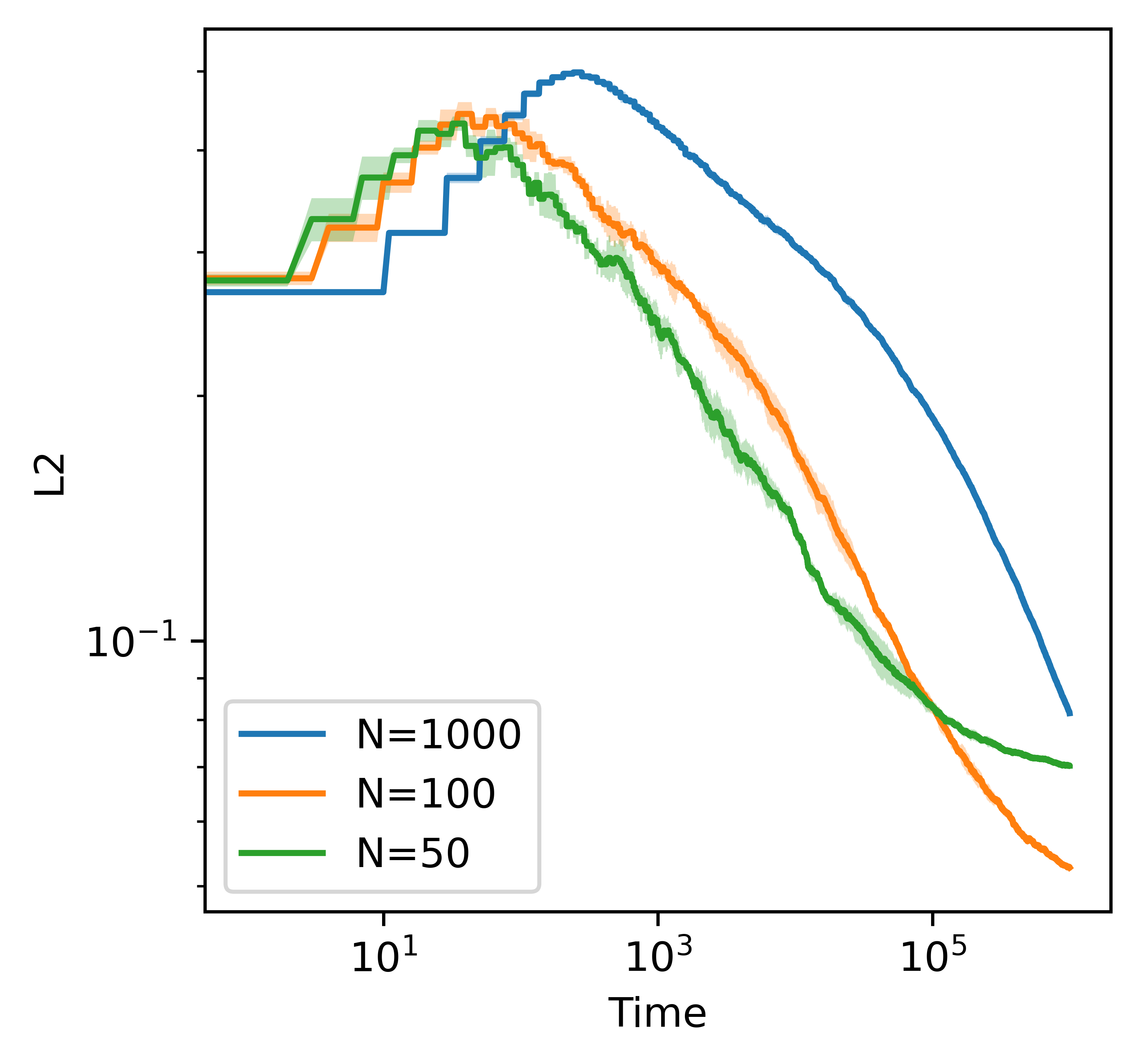} &   \includegraphics[width=0.35\linewidth]{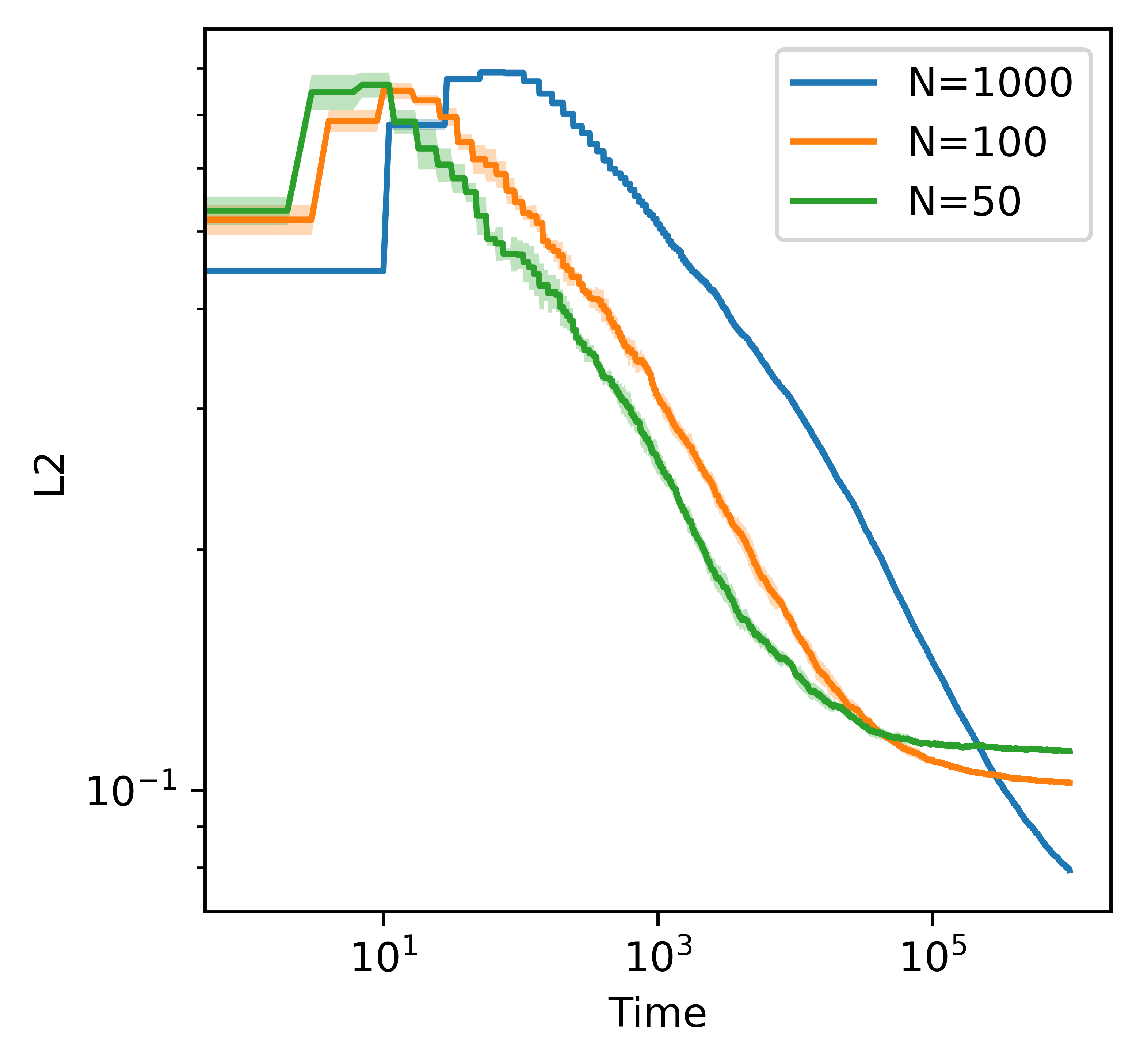} \\
(c) & (d)
\end{tabular}
\caption{
The mean $\ell_2$ distance to MF-NE given by $\frac{1}{N}\sum_{i\in\setN} \| \vecpi^i - \vecpi^*\|_2$ with $N$ agents throughout learning with bandit feedback for three different $N$, on the problems (a) linear payoffs, (b) exponentially decreasing payoffs, (c) payoffs with KL potential and (d) the beach bar payoffs.
}
\label{figure:bandit_l2_curves}
\end{figure}

Furthermore, comparing the observations for bandit feedback (Figure~\ref{figure:bandit_exp_curves}) and full feedback cases (Figure~\ref{figure:expert_exp_curves}), we empirically confirm that learning take more iterations in the bandit case.
This is likely due to the fact that exploration occurs probabilistically, inducing additional variance in the policy updates that increases with $N$ and incorporates an additional logarithmic term in the theoretical bounds.
However, the number of exploration epochs in the bandit case is comparable to the number of time steps in the full feedback case.

Finally, we also emphasize the fact that in earlier stages of training with bandit feedback, the cases with $N=1000$ had much higher exploitability and $\ell_2$ distance to MF-NE at earlier time steps.
This is due to the fact that policy updates occur with larger intervals in between when $N$ is large, as can be seen in Algorithm~\ref{alg:bandit}.
This can be explicitly observed in Figure~\ref{figure:bandit_l2_curves}, as the policies of agents are constant in between policy updates.
However, at later stages as the time-dependent term in the bound on epxloitability in Corollary~\ref{corollary:bandit} disappears, we observe that the experiments for larger $N$ converge to a better policy (i.e., one with lower exploitability) than the cases with smaller $N$ as the theory suggests.

Finally, comparing Figures~\ref{figure:bandit_exp_curves},\ref{figure:bandit_l2_curves}, we see that in certain experiments for some $N$ despite having a lower exploitability we observe a greater $\ell_2$ distance to MF-NE.
This likely due to fact that the non-vanishing bias term in exploitability and $\ell_2$ distance have differing dependence on problem parameters such as $L, \lambda, K$.
Therefore, for instance for the KL potential payoffs, we observe a greater mean $\ell_2$ distance to MF-NE but a smaller exploitability for $N=1000$.

\subsection{Comparison with Multiplicative Weight Updates}\label{appendix:comparison_omd}

The SMFG problem of interactive learning without communications and bandit feedback has not been studied in the literature previously, and there is no alternative algorithm with theoretical guarantees to the best of our knowledge in this setting.
However, in Figure~\ref{figure:fig_mwu} we provide comparison with a heuristic extension of the OMD algorithm proposed by \citet{perolat2022scaling} which was analyzed in the case of an infinite agent game.
In general, independent learning with $N$ agents introduces unique challenges that cause the algorithm to cycle or diverge.

\begin{figure}[h]
\begin{tabular}{cc}
 \includegraphics[width=0.45\linewidth]{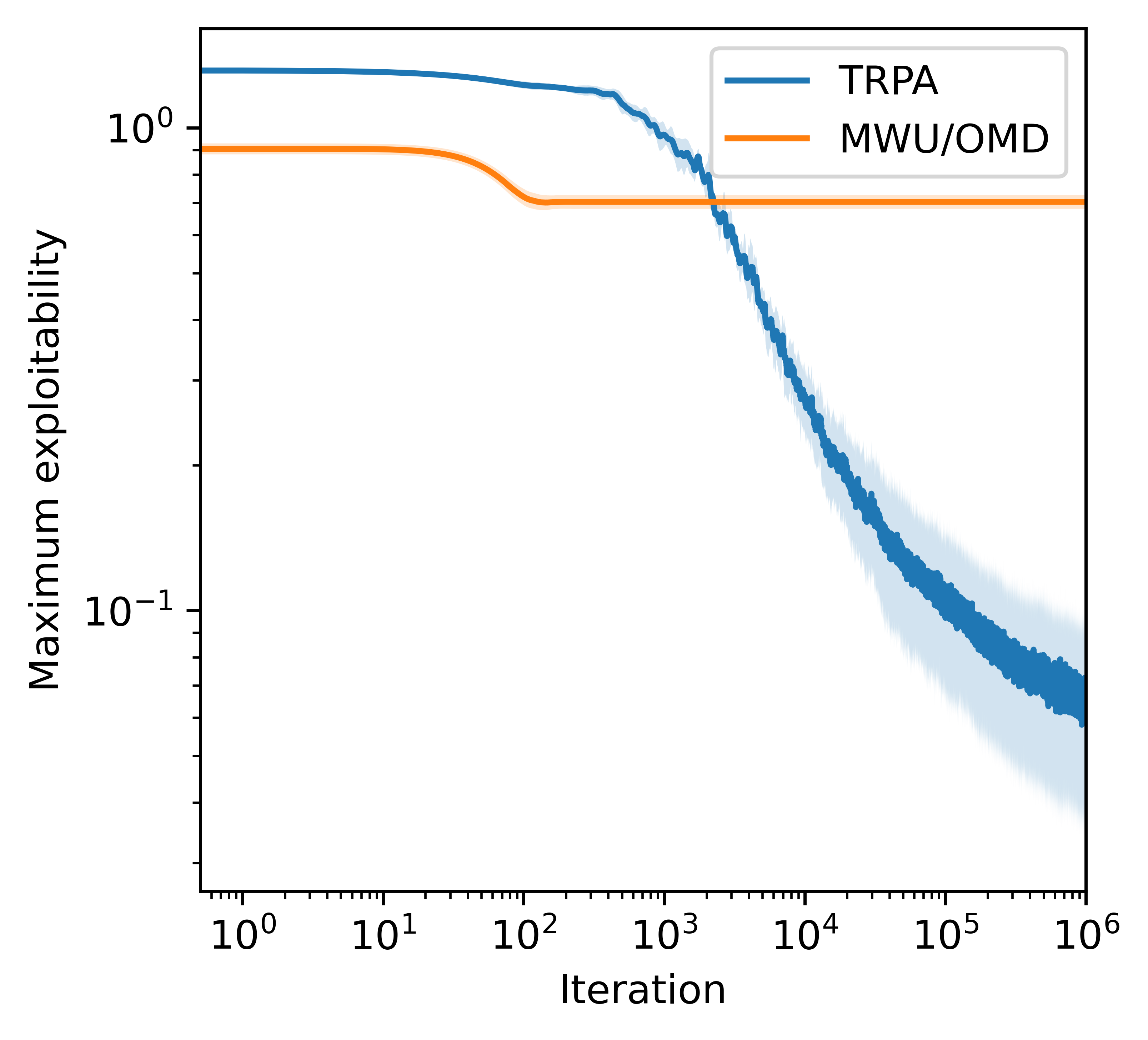} &   \includegraphics[width=0.45\linewidth]{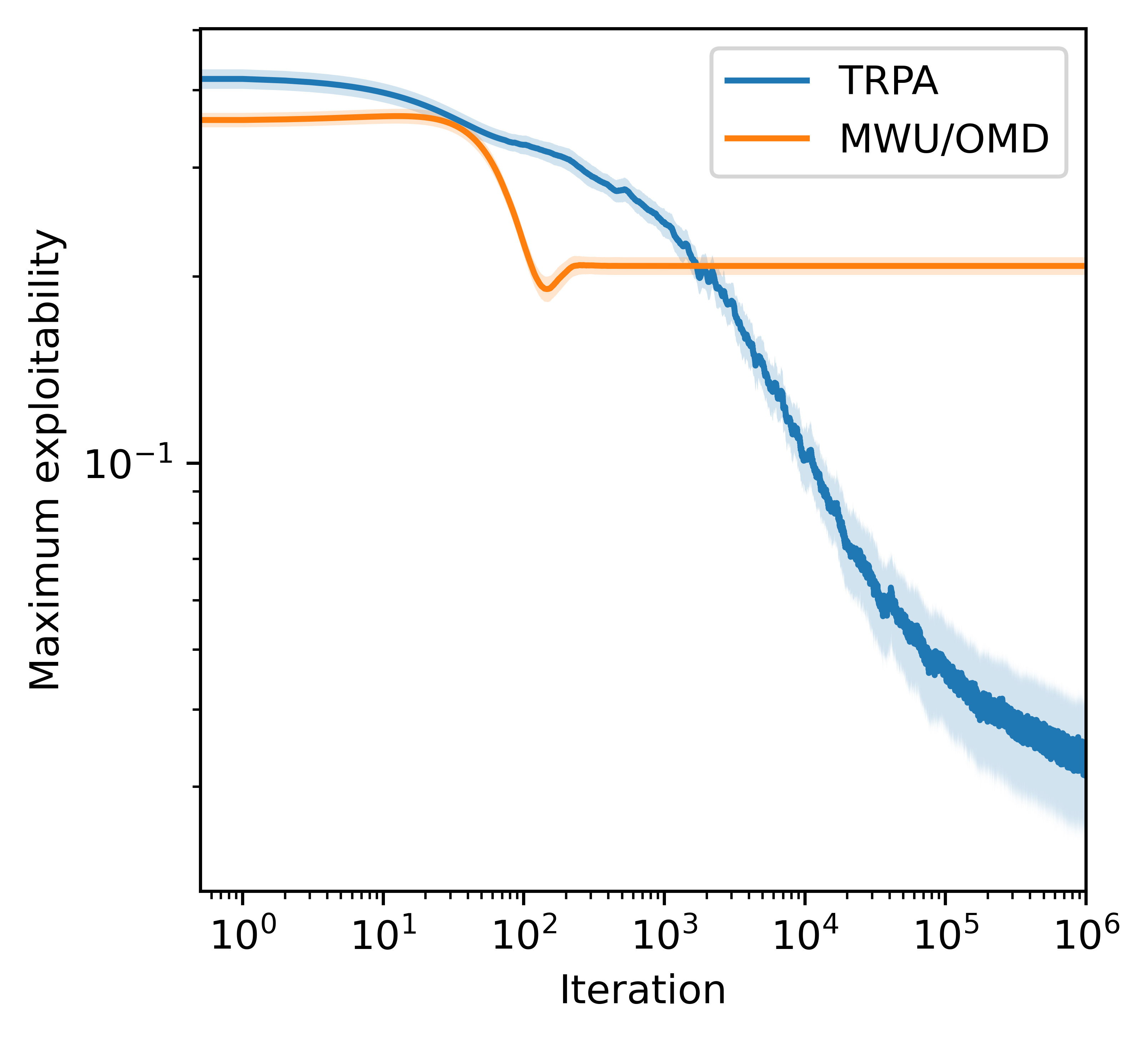} \\
 \includegraphics[width=0.45\linewidth]{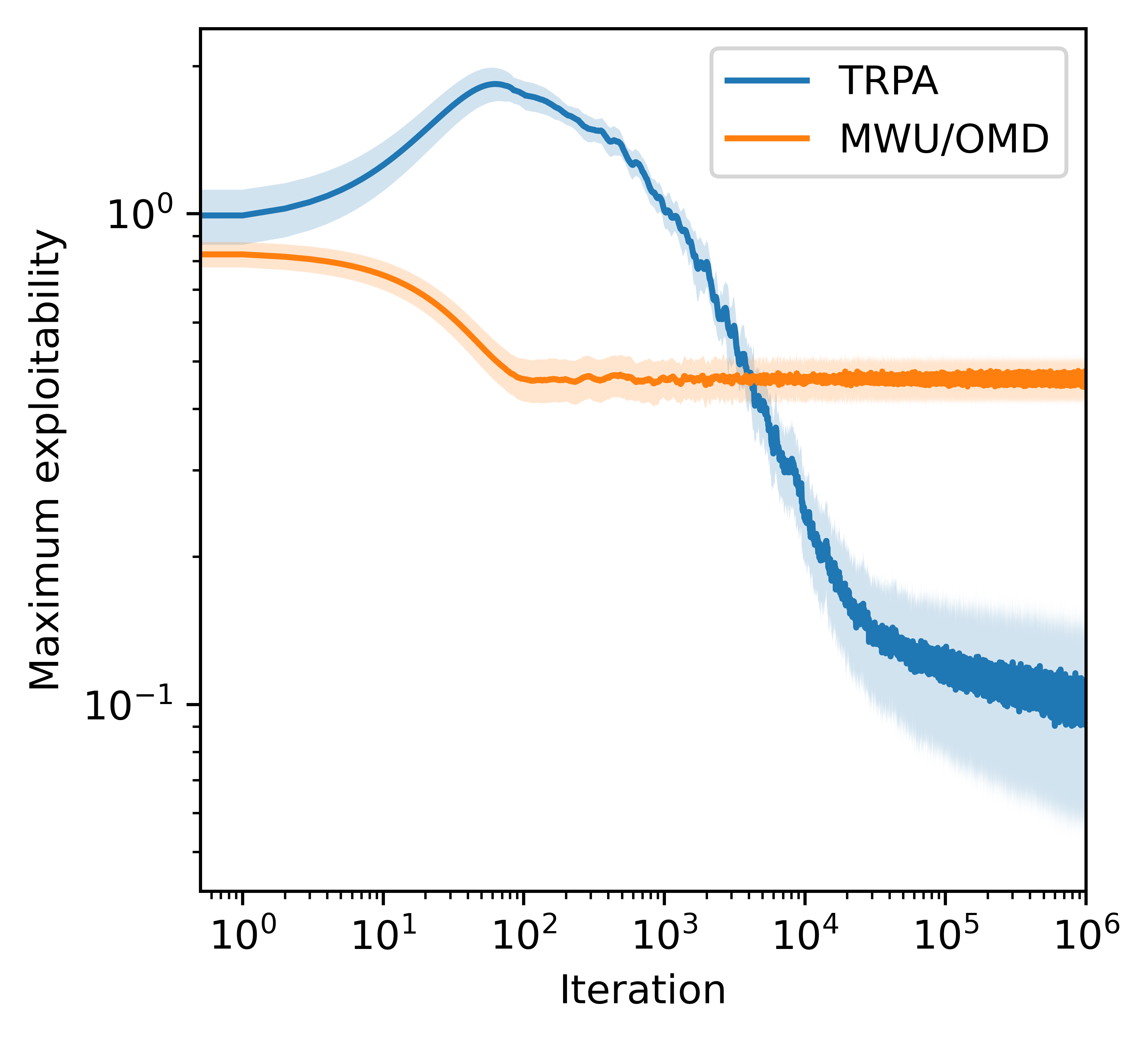} &   \includegraphics[width=0.45\linewidth]{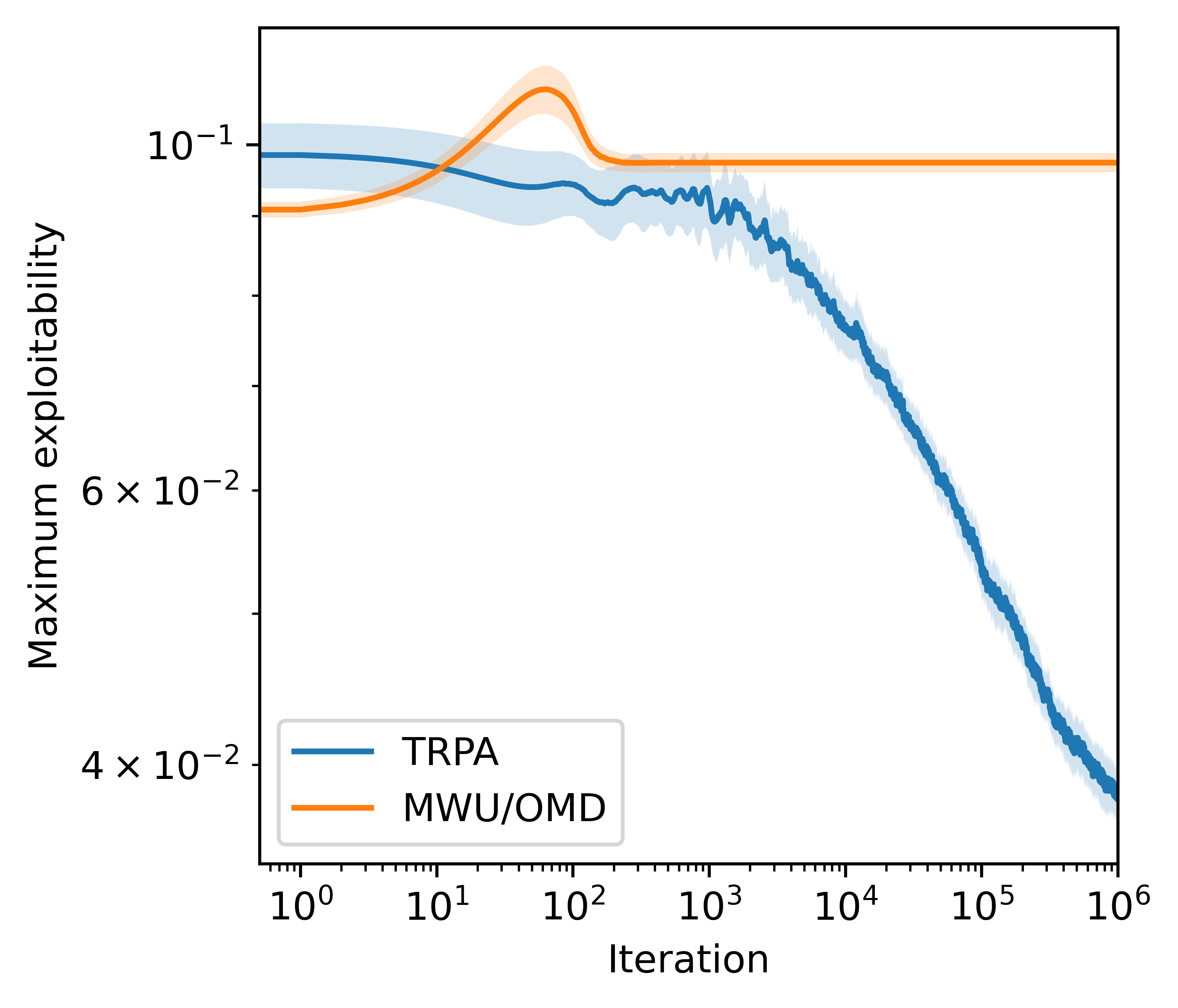}
\end{tabular}
\caption{
Comparison of maximum exploitability between our TRPA-Bandit and multiplicative weight updates (equivalently OMD of \citet{perolat2022scaling}) in benchmarks \textsc{linear}, \textsc{bb}, \textsc{kl}, \textsc{exp} (left to right). $N=1000$ agents, averaged over 10 experiments.
}
\label{figure:fig_mwu}
\end{figure}

\subsection{Experiments on Traffic Congestion}\label{appendix:exp:traffic}

\textbf{UTD19 and closed-loop sensors.}
The UTD19 dataset contains measurements by closed-loop sensors which report the fraction of the time a particular section of the route remains occupied (i.e., a car is located in between sensors placed on the sides).
The data consists of measurements every 5 minutes, from various sensors across 41 European cities.
The dataset contains 2 weeks of data collected by sensors placed around Zurich.

\textbf{Payoff models.}
We fit kernelized ridge regression models to model the flow as a function of occupancy using the UTD19 dataset.
We use an RBF kernel and a regularization of $\alpha=1.0$ for all models.
The dependence of traffic speed on route occupancy in two sample paths in the dataset can be seen in Figure~\ref{figure:utd_models_occupancy_speed}-(a, b).
The measurements motivate the monotonicity assumption as both the original data and the fitted regression model, increased occupancy leads to increased travel time, at least, when ignoring interactions.
We compute a proxy for the travel velocity using the flow and occupancy measurements on each route, and a scaling factor $c_{\text{dist}}$ due to varying lengths of each route, leading to the estimated travel time
\begin{align*}
    T_{travel} &= c_{\text{dist}} \frac{\text{flow}}{\text{occupancy}}.
\end{align*}
We use $-T_{travel}$ as the reward for each agent.

\begin{figure}[h]
\centering
\begin{tabular}{ccc} 
\includegraphics[width=0.31\linewidth]{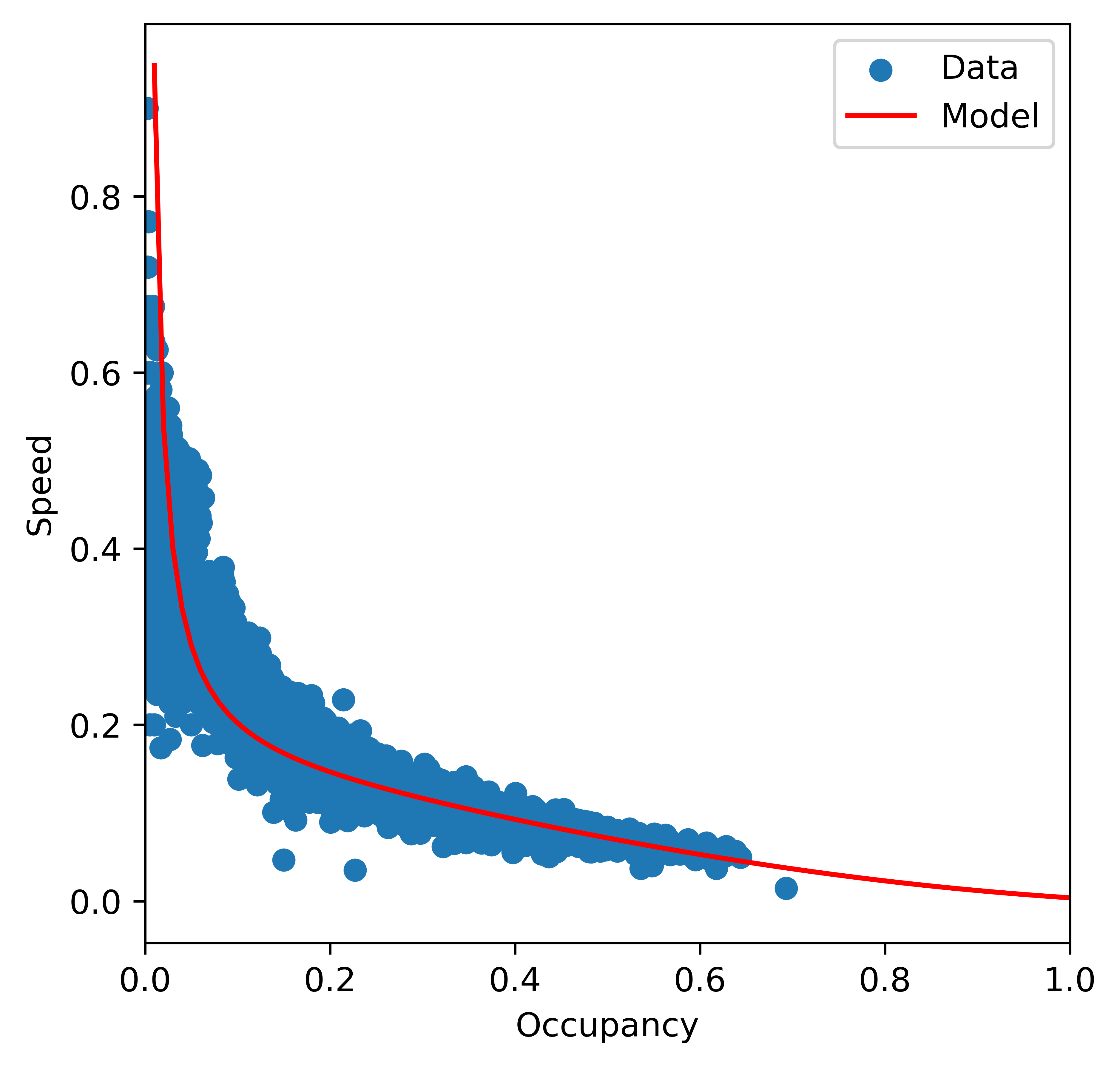} & \includegraphics[width=0.31\linewidth]{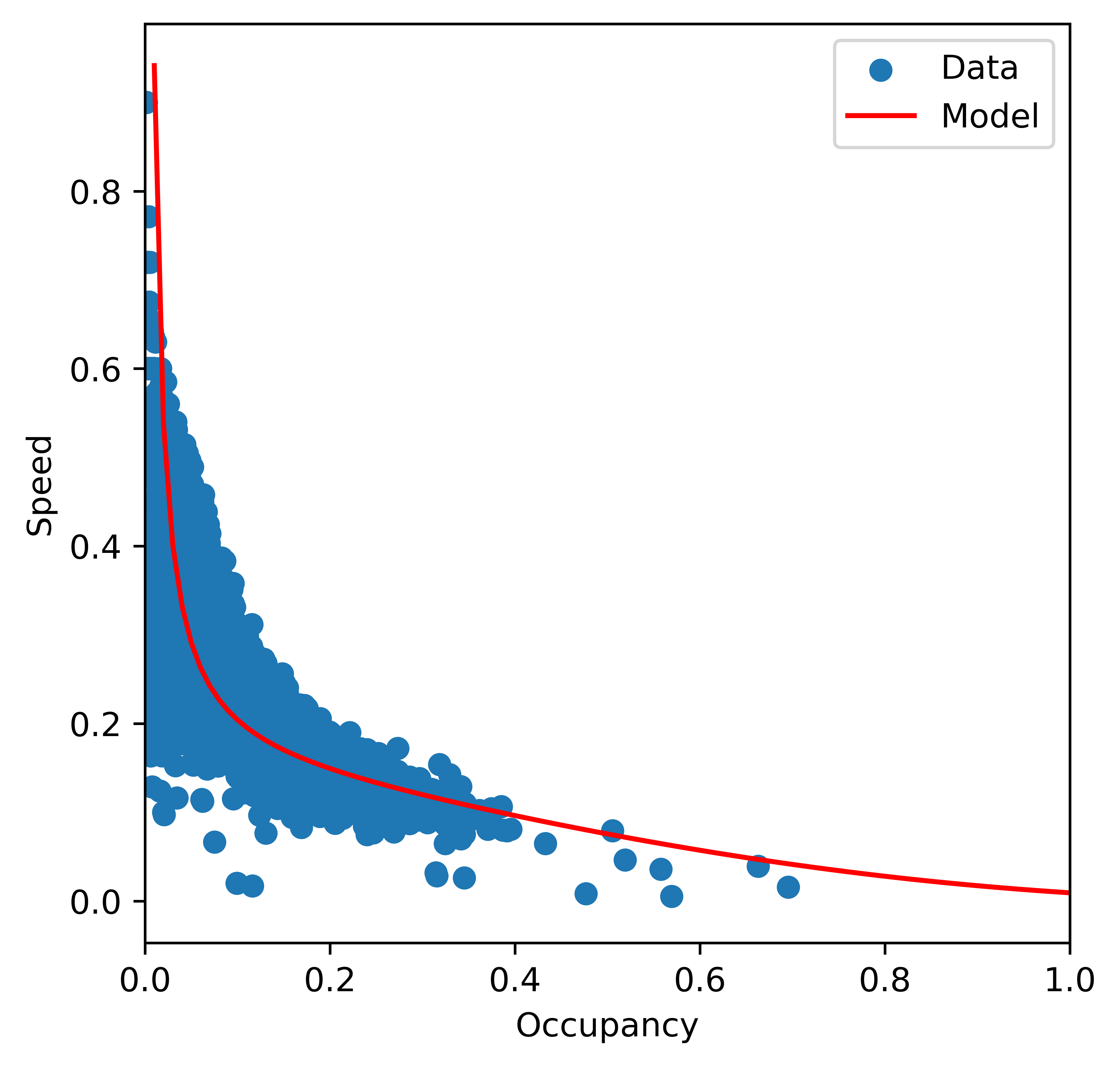} & \includegraphics[width=0.3\linewidth]{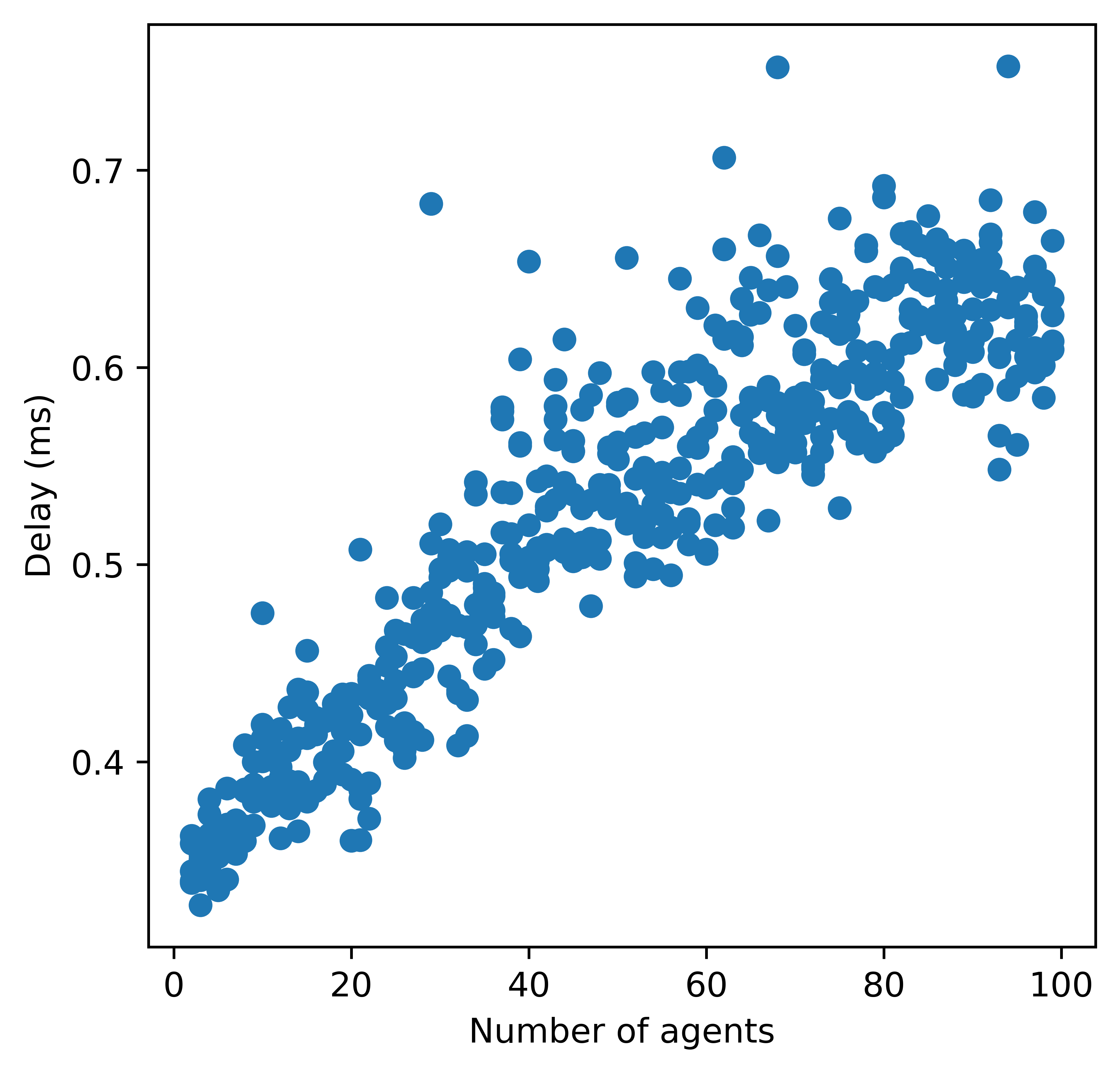} \\
(a) & (b) & (c)
\end{tabular}
\caption{
(a, b) Data and fitted models of occupancy vs. speed on two different urban roads (sensors K324D11 and K418D12) from the UTD19 dataset.
The red line indicates the fitted model predictions of the kernelized regression model.
(c) Measured ping times of a selected Tor entry node as a function of parallel requests (i.e., number of agents accessing the node). }
\label{figure:utd_models_occupancy_speed}
\end{figure}

\subsection{Experiments on Network Access}\label{appendix:exp:tor}

For the Tor network access experiment, we randomly chose 5 entry guard servers (the complete list available publicly \cite{torentryguards}) in various geographic locations, among the servers that have the longest recorded uptime.
To simulate access to each server, we ping each node 5 consecutive times and average the delays to compute the cost.
As expected, due to varying bandwidths/computational power, each server has different sensitivities to load in terms of delay.
We plot measured ping times vs the number of agents simultaneously accessing a particular node in Figure~\ref{figure:utd_models_occupancy_speed}-(c).
The ping times were measured to be stochastic but clearly sensitive to occupancy, motivating a monotone payoff operator.

We use parameters $\tau = 0.01, \varepsilon = 0.3$ for the experiments in this section.
The arbitrary choice is due to the fact that in the presence of external agents in the game that do not use TRPA-Bandit (in this case, thousands of other users accessing the Tor network), the theoretically optimal parameters implied by Corollary~\ref{corollary:bandit} can not be used.
While more realistic simulations that are also closer to the theory could be run by keeping $N$ larger and simulating a Tor access rather than simple pings, we refrain from this in order to minimize the footprint of our experiments on the Tor network.

\end{document}